\documentclass[11pt,reqno]{amsart}

\usepackage{geometry}
\usepackage[hidelinks]{hyperref}
\usepackage{placeins}

\usepackage{amsmath,amsthm,amssymb,amsfonts,bbm,graphicx,caption}
\usepackage{xcolor}
\usepackage[ruled,vlined,linesnumbered]{algorithm2e}
\usepackage{enumitem}

\theoremstyle{plain}
\newtheorem{theorem}{Theorem}[section]
\newtheorem{proposition}{Proposition}[section]
\newtheorem{lemma}{Lemma}[section]

\newtheorem{assumption}{Assumption}[section]

\theoremstyle{definition}
\newtheorem{definition}{Definition}[section]

\newtheorem{remark}{Remark}[section]

\newcommand\independent{\protect\mathpalette{\protect\independenT}{\perp}}
\def\independenT#1#2{\mathrel{\rlap{$#1#2$}\mkern2mu{#1#2}}}

\newcommand{\R}{\mathbb{R}}

\newcommand{\N}{\mathbb{N}}
\newcommand{\E}{\mathbb{E}}

\newcommand{\calX}{\mathcal{X}}
\newcommand{\calY}{\mathcal{Y}}
\newcommand{\calC}{\mathcal{C}}
\newcommand{\calP}{\mathcal{P}}

\newcommand{\calQ}{\mathcal{Q}}
\newcommand{\calU}{\mathcal{U}}

\newcommand{\calN}{\mathcal{N}}

\newcommand{\calH}{\mathcal{H}}

\renewcommand{\tilde}{\widetilde}
\renewcommand{\hat}{\widehat}

\newcommand{\tildef}{\tilde{f}}
\newcommand{\tildeg}{\tilde{g}}
\newcommand{\tildeh}{\tilde{h}}

\newcommand{\barf}{\bar{f}}
\newcommand{\barg}{\bar{g}}
\newcommand{\barh}{\bar{h}}

\DeclareMathOperator{\op}{op}

\DeclareMathOperator{\diam}{diam}
\DeclareMathOperator{\argmax}{argmax}

\newcommand{\KL}[2]{\mathsf{KL}(#1 \,\|\, #2 )}

\begin{document}
\title[Sinkhorn algorithms for entropic VQR]{Sinkhorn algorithms for entropic vector quantile regression}

\thanks{K. Kato is partially supported by NSF grant DMS-2413405.}

\author[K. Kato]{Kengo Kato}
\address[K. Kato]{
Department of Statistics and Data Science, Cornell University.
}
\email{kk976@cornell.edu}

\author[B. Wang]{Boyu Wang}
\address[B. Wang]{
Department of Statistics and Data Science, Cornell University.
}
\email{bw563@cornell.edu}

\begin{abstract}
Vector quantile regression (VQR) is an optimal transport (OT)-based framework that extends linear quantile regression to vector-valued response variables and can be formulated as an OT problem with a mean-independence constraint. In this paper, we study Sinkhorn-type algorithms for VQR with entropic regularization, building on the authors' previous work on its duality theory. We first study a direct adaptation of the classical Sinkhorn iteration based on solving the full Schr\"{o}dinger-type system characterizing the dual potentials, which requires solving an implicit functional equation at each iteration. Since the implicit equation does not admit a closed-form solution in general, we consider a modified algorithm that replaces the implicit update with a gradient ascent step, resulting in a computational scheme that does not involve any subroutine. The latter algorithm can be viewed as an implementable version of the former, idealized scheme that serves as a theoretical benchmark. For both algorithms, and for general compactly supported marginals, we establish linear convergence in both the dual objective value and the iterates. A key innovation in our analysis is the derivation of explicit quantitative bounds on the dual potentials and Sinkhorn iterates.
\end{abstract}
\date{First version: March 21, 2026. This version: \today}

\maketitle

\section{Introduction}
\subsection{Overview}

Optimal transport (OT) offers a versatile framework for comparing probability distributions and has seen a surge of applications in statistics, machine learning, and applied mathematics \cite{villani2009optimal,santambrogio2015optimal,panaretos2020invitation,chewi2025statistical}.
For given Borel probability measures $\mu,\nu$ on Polish metric spaces $\mathsf{X},\mathsf{Y}$, respectively, and a Borel nonnegative cost function $c: \mathsf{X} \times \mathsf{Y} \to [0,\infty)$, the Kantorovich OT problem is given by
\begin{equation}
\inf_{\pi \in \Pi(\mu,\nu)} \int_{\mathsf{X} \times \mathsf{Y}} c \, d\pi,
\label{eq: OT}
\end{equation}
where $\Pi(\mu,\nu)$ denotes the collection of couplings for $(\mu,\nu)$. Recall that any coupling $\pi \in \Pi(\mu,\nu)$ is a joint distribution with marginals $\mu,\nu$.

Among many statistical applications of OT, the seminal work by \cite{carlier2016vector} proposed an OT-based approach to extending linear quantile regression \cite{koenker1978regression, koenker2005quantile} to vector-valued response variables. 
Let $(X,Y) \in \R^{d_x + d_y}$ be a pair of covariate and response vectors, and let $\mu$ be a reference distribution on $\R^{d_y}$. Denoting by $\nu$ the joint distribution of $(X,Y)$, the vector quantile regression (VQR) problem introduced in \cite{carlier2016vector} consists of minimizing, among all couplings $\pi \in \Pi(\mu,\nu)$, the expected cost $\int c \, d\pi$ with $c(u,y) := \|u-y\|^2/2$, subject to a \textit{mean-independence} constraint\footnote{When we generate a version of $(X,Y)$ coupled with $U \sim \mu$, we use $(\tilde{X},\tilde{Y})$.},
\begin{equation}
\inf_{\pi \in \Pi(\mu,\nu)} \left \{\E\big[c(U,\tilde{Y})\big] : (U,\tilde{X},\tilde{Y}) \sim \pi, \ \E\big[\tilde{X} \mid U\big] = \E\big[\tilde{X}\big] \ \text{a.s.} \right \}.
\label{eq: vqr}
\end{equation}
When $d_y=1$, and under regularity conditions, the VQR problem (\ref{eq: vqr}) reduces to classical linear quantile regression; see Theorem 3.3 in \cite{carlier2016vector} and their follow-up work \cite{carlier2017vector}.

Since the work of \cite{cuturi2013sinkhorn}, entropic regularization has been widely used to approximately solve the OT problem (\ref{eq: OT}), as it enables efficient computation via the \textit{Sinkhorn algorithm} \cite{peyre2019computational}. For the general OT problem (\ref{eq: OT}), entropic regularization amounts to solving
\begin{equation}
\inf_{\pi \in \Pi(\mu,\nu)} \int_{\mathsf{X} \times \mathsf{Y}} c \, d\pi + \varepsilon \KL{\pi}{\mu \otimes \nu},
\label{eq: EOT}
\end{equation}
where $\varepsilon > 0$ is a regularization parameter and $\mathsf{KL}$ denotes the \textit{Kullback-Leibler divergence} (or relative entropy) defined by 
\begin{equation*}
\KL{P}{Q} := 
\begin{cases}
\int \log \frac{dP}{dQ}\, dP, & \text{if $P \ll Q$}, \\
\infty, & \text{otherwise}. 
\end{cases}
\end{equation*}
Under suitable conditions on the marginals (cf. \cite{nutz2021introduction}), the entropic OT problem (\ref{eq: EOT}) admits a unique optimal solution $\pi$, which has a density of the form
\begin{equation*}
\frac{d\pi}{d(\mu \otimes \nu)} (x,y) = e^{(f(x)+g(y)-c(x,y))/\varepsilon}. 
\end{equation*}
The functions $(f,g)$ solve the dual problem for (\ref{eq: EOT}) and are characterized by the system of equations,  known as the \textit{Schr\"{o}dinger system}, 
\begin{equation*}
\begin{split}
f(x) &= -\varepsilon \log \int e^{(g(y)-c(x,y))/\varepsilon} \, d\nu(y), \\
g(y) &= -\varepsilon \log \int e^{(f(x)-c(x,y))/\varepsilon} \, d\mu(x).
\end{split}
\end{equation*}
The Sinkhorn algorithm iteratively solves these two equations, which correspond to the Euler-Lagrange equations for the dual problem. Thus, the Sinkhorn algorithm can be viewed as dual block coordinate ascent. Although the algorithm itself dates back to the 1960s \cite{sinkhorn1967diagonal}, it has attracted substantial attention in recent years; see the literature review below. 

 For the VQR problem (\ref{eq: vqr}), the prior work \cite{carlier2022vector} considered entropic regularization and used gradient descent to solve the dual problem in the discrete setting. However, that work did not provide formal guarantees for the application of gradient descent to entropic VQR.
 The authors' recent work \cite{kato2026entropic} studied duality theory for entropic VQR in detail (see also \cite{carlier2025weak}), which reads
\begin{equation}
\inf_{\pi \in \Pi(\mu,\nu)} \left \{\E\big[c(U,\tilde{Y})\big] + \varepsilon \KL{\pi}{\mu \otimes \nu} : (U,\tilde{X},\tilde{Y}) \sim \pi, \ \E\big[\tilde{X} \mid U\big] = \E\big[\tilde{X}\big] \ \text{a.s.} \right \}.
\label{eq: evqr}
\end{equation}
In particular, \cite{kato2026entropic} showed that, under regularity conditions, the entropic VQR problem (\ref{eq: evqr}) admits a unique optimal solution $\pi$, which has a density of the form
\begin{equation*}
\frac{d\pi}{d(\mu\otimes \nu)}(u,x,y) = e^{(f(u)+\langle g(u),x\rangle + h(x,y)-c(u,y))/\varepsilon},
\end{equation*}
where $f: \R^{d_y} \to \R, g: \R^{d_y} \to \R^{d_x}$, and $h: \R^{d_x+d_y} \to \R$ are dual potentials solving the dual problem for (\ref{eq: evqr}). These dual potentials are characterized by a system of equations analogous to the Schr\"{o}dinger system, except that the equation corresponding to the vector-valued potential 
$g$ is implicit; see equations (\ref{eq: f})--(\ref{eq: h}) below. The presence of the extra implicit equation creates a significant challenge for both algorithm design and convergence analysis in entropic VQR.

In this paper, we study Sinkhorn-type algorithms for entropic VQR (\ref{eq: evqr}). 
We first study a direct adaptation of the classical Sinkhorn iteration based on \textit{exactly} solving the full Schr\"{o}dinger-like system characterizing the dual potentials. Since the update in the $g$-direction is not available in closed-form, we consider a modified algorithm that replaces the implicit update for the vector-valued potential with a gradient ascent step, resulting in a computational scheme that does not involve any subroutine. The latter algorithm can be viewed as an implementable version of the former that serves as a theoretical benchmark. The modified scheme is related to SISTA from  \cite{carlier2023sista}; see the literature review below for a comparison.

The main contribution of this work is to establish linear convergence of these algorithms, both in terms of the dual objective value and the iterates. In particular, our analysis allows for general marginals with compact supports and is not restricted to the discrete setting.
To the best of the authors' knowledge, this is the first paper to establish rigorous convergence guarantees for Sinkhorn-based algorithms for entropic VQR. 
A key innovation in our analysis is the derivation of explicit quantitative bounds on the dual potentials and Sinkhorn iterates, which is highly nontrivial and requires new ideas compared with standard entropic OT; see the proofs of Propositions \ref{prop: potential bound}, \ref{prop: sinkhorn potential}, and \ref{prop: sinkhorn potential modified}. In particular, for the modified Sinkhorn algorithm, it is \textit{a priori} unclear whether suitable coercivity estimates hold for the dual objective in the $g$-direction, which poses a challenge in establishing uniform-in-time estimates of the gradient of the dual objective in the $g$-direction, as well as finding a condition on the step size used in gradient ascent for which monotonicity of the dual objective values holds. 
Partly inspired by the proof of Proposition~2.7 in \cite{kato2026entropic}, we establish such coercivity estimates using theory of exponential families (cf. Chapter 3 in \cite{wainwright2008graphical}), which may be of independent interest; see Lemma~\ref{lem: monotonicity modified} and the remark following it. Combined with an induction argument, we establish quantitative bounds on the iterates in the $g$-direction, which yield bounds on other potentials. 

Given such quantitative estimates on the iterates, we establish linear convergence of both algorithms, by drawing on various modern techniques in optimization from \cite{attouch2013convergence,karimi2016linear,bolte2017error,lewis2025complexity}, among others.
Specifically, for both algorithms, we establish (i) (versions of) \textit{Polyak-{\L}ojasiewicz (PL) inequalities} for the dual objective along the iterates, and (ii) \textit{slope-ascent} conditions for the iterates. Having established these properties, the linear convergence results follow by adapting the analysis in \cite{bolte2017error}.
Finally, we conduct numerical experiments to evaluate the empirical performance of the modified Sinkhorn algorithm using both synthetic and real data, which is largely consistent with our theoretical results, although data-adaptive step selection rules would be left for further investigation.

\subsection{Related literature}
Quantiles and their related notions such as ranks play fundamental roles in statistics and probability. There is a large literature on extending these concepts from univariate random variables to vector-valued ones, where the main difficulty lies in the fact that there is no canonical ordering in $\R^d$ when $d \ge 2$; see \cite{hallin2021distribution,hallin2022measure,hallin2024multivariate}. Among others, \cite{chernozhukov2017monge} proposed viewing the \textit{Brenier map} \cite{brenier1991polar}, transporting an absolutely continuous reference measure $\mu$ on $\R^{d_y}$ onto the law of $Y \in \R^{d_y}$, as a vector quantile function of $Y$; see also \cite{ghosal2022multivariate}. The aforementioned references have established desirable structural properties of OT-based vector quantiles and ranks, which provide compelling arguments for the use of OT in modern data analysis. In the presence of covariates $X \in \R^{d_x}$, \cite{carlier2016vector} considered the conditional vector quantile function as the conditional Brenier map, transporting the reference measure $\mu$ onto the conditional distribution of $Y$ given $X$, which amounts to solving the Kantorovich OT problem with the \textit{independence constraint}, i.e., the problem (\ref{eq: vqr}) with the mean independence constraint $\E[\tilde{X} \mid U] = \E[\tilde{X}]$ replaced by the independence constraint $\tilde{X} \independent U$. See also the recent works by \cite{del2025nonparametric,baptista2025conditional} that consider (nonparametric) estimation of the conditional vector quantile function. Inspired by linear quantile regression for the univariate ($d_y=1$) response case \cite{koenker1978regression,koenker2005quantile}, \cite{carlier2016vector} considered modeling the conditional vector quantile function as an affine function of $X$ and showed that the induced coupling is an optimal solution to the VQR problem (\ref{eq: vqr}), which relaxes the independence constraint into the mean-independence constraint. 

The Sinkhorn algorithm (for standard entropic OT) has attracted considerable attention in recent years because of increasing interest in OT-based tools in various application domains. 
For discrete marginals, it can be viewed as a matrix scaling algorithm, and a contraction argument was used to establish its linear convergence under the Hilbert projective metric by \cite{franklin1989scaling}.
A similar argument was extended to the continuous setting by 
 \cite{chen2016entropic}. A different approach was taken by \cite{carlier2022linear}, where the Sinkhorn algorithm is viewed as dual block coordinate ascent and its linear convergence in the multi-marginal setting is established by adapting the analysis from \cite{beck2013convergence}. Our linear convergence proof for the (vanilla) Sinkhorn algorithm is essentially an adaptation of the approach of \cite{carlier2022linear}. As noted before, however, the presence of an implicit equation for one of the potentials complicates the analysis.
For various other guarantees for the Sinkhorn algorithm under different settings (but still in the standard entropic OT case), we refer the reader to \cite{ruschendorf1995convergence,benamou2015iterative,altschuler2017near,dvurechensky2018computational,marino2020optimal,berman2020sinkhorn,chakrabarty2021better, ghosal2025convergence, eckstein2025hilbert,chizat2026sharper,conforti2023quantitative} and references therein. 

The mean-independence constraint in VQR is reminiscent of martingale OT considered in the mathematical finance literature (see, e.g., \cite{beiglbock2013model, galichon2014stochastic}). An adaptation of the Sinkhorn algorithm to martingale OT with entropic regularization was considered by \cite{chen2026convergence}, where linear convergence is established for the algorithm. However, the setting of VQR substantially differs from martingale OT and requires a different analysis. In addition, \cite{chen2026convergence} considered only a block coordinate ascent algorithm that requires exactly maximizing the dual objective with respect to each potential. 

Other related references include \cite{carlier2023sista, gallouet2025metric}. In \cite{carlier2023sista}, the authors considered learning the transport cost by parameterizing it with a finite-dimensional parameter and optimizing a convex objective over the dual potentials and cost parameter altogether.  In addition, they introduced $\ell^1$-penalization on the cost parameter and proposed an iterative algorithm called SISTA, which alternates between the Sinkhorn step for the potentials and the proximal gradient step for the cost parameter. Our modified Sinkhorn algorithm is related to SISTA, in the sense that both combine the Sinkhorn step for one type of parameters and a gradient-based method for another type of parameters. However, their convergence proof relies on some specific structures of the problem (e.g., the presence of $\ell^1$-penalty and the cost being parameterized as a linear combination of basis functions), which our setting does not share. In particular, in \cite{carlier2023sista}, the $\ell^1$-penalization plays an important role in establishing a suitable coercivity estimate for the cost parameter.  A problem similar to \cite{carlier2023sista} was considered in \cite{gallouet2025metric}. The analyses in both  \cite{carlier2023sista} and \cite{gallouet2025metric} are restricted to discrete marginals, as their analyses explicitly involve the number of atoms or the probability masses. Finally, VQR is a special case of weak OT with moment constraints considered in the recent
preprint \cite{carlier2025weak}, where an adaptation of SISTA is discussed in Section 6. However, no formal convergence guarantees are provided in \cite{carlier2025weak}.

\subsection{Organization}
The rest of the paper is organized as follows. Section \ref{sec: duality} reviews duality results for entropic VQR from \cite{kato2026entropic} and presents quantitative bounds for dual potentials and a local PL inequality for the dual objective. Sections \ref{sec: sinkhorn} and \ref{sec: modified sinkhorn} consider the Sinkhorn and modified Sinkhorn algorithms, respectively, and present their convergence guarantees. Section \ref{sec: numerical} reports numerical experiments. Sections \ref{sec: proof sec 2 and 3} and \ref{sec: proof sec 4} contain all proofs for Sections \ref{sec: duality}--\ref{sec: modified sinkhorn}. Section \ref{sec: discussion} leaves some discussion on future research directions. The Appendices contain an example demonstrating sharpness of the step size condition considered in Section \ref{sec: modified sinkhorn} and present additional experiments. 
%Appendix \ref{sec: appendix} contains an auxiliary result concerning the projection step used in the modified Sinkhorn algorithm. 

\subsection{Notation}
Throughout the paper, we will obey the following notation. 
\begin{itemize}
\item 
On a Euclidean space, let $\| \cdot \|$ and $\langle \cdot, \cdot \rangle$ denote the standard Euclidean norm and inner product, respectively. 
\item Let $\| \cdot \|_{\op}$ denote the operator norm for matrices. 
\item For a Polish metric space $M$, let $\calP(M)$ denote the space of all Borel probability measures on $M$. 
\item For any $\mu \in \calP(M), p \in [1,\infty]$, and $d \in \N$, let $L^p(\mu; \R^d)$ denote the $L^p(\mu)$-space of Borel measurable mappings $M \to \R^{d}$, endowed with the $L^p(\mu)$-norm  $\| f \|_{L^p(\mu)} := \left (\int \|f(x)\|^p \, d\mu(x) \right)^{1/p}$ (with the obvious modification when $p=\infty$). 
For $p=2$, $L^2(\mu;\R^{d})$ is a Hilbert space with inner product $\langle f,g \rangle_{L^2(\mu)} := \int \langle f,g \rangle \, d\mu$. 
\item We use $\calC(M;\R^{d})$ to denote the space of continuous mappings $M \to \R^d$. 
\item We write $L^p(\mu) = L^p(\mu;\R)$ and $\calC(M) = \calC(M;\R)$.
\item For $a,b \in \R$,  we use the notation $a \wedge b = \min \{ a,b \}$ and $a \vee b = \max\{ a,b \}$. In addition, we write $a^+ = a \vee 0$ and $a^- = (-a) \vee 0$. 
\item Finally, let $\N_0$ denote the set of nonnegative integers. 
\end{itemize}

\section{Duality for entropic VQR}
\label{sec: duality}

\subsection{Review of duality theory}
We start with fixing the notation. 
Let $(X,Y) \in \R^{d_x} \times \R^{d_y}$ be a pair of covariate and response vectors. We denote by $\calX$ and $\calY$ the supports of $X$ and $Y$, respectively. 
Let $\mu \in \calP(\R^{d_y})$ be a reference measure with support $\calU$. Throughout the rest of the paper, we maintain the following assumption. 

\begin{assumption}
The supports $\calU \subset \R^{d_y}, \calX \subset \R^{d_x}$, and $\calY \subset \R^{d_y}$ are compact. In addition,  $\E[X]=0$ and the matrix $\Sigma_X:=\E[XX^\top]$ is invertible.
\end{assumption}

The assumption encompasses both discrete and continuous marginals. The assumption that $\Sigma_X$ is invertible is needed to ensure dual attainment; see \cite{kato2026entropic}.

For any $\pi \in \Pi(\mu,\nu)$, we denote by $\pi_u$ the conditional distribution of $(\tilde{X},\tilde{Y})$ given $U$ when $(U,\tilde{X},\tilde{Y}) \sim \pi$, i.e., for any bounded measurable function $\varphi: \calU \times \calX \times \calY \to \R$, 
\begin{equation*}
\int \varphi \, d\pi = \int_{\calU} \left ( \int_{\calX \times \calY} \varphi (u,x,y) \, d\pi_u(x,y) \right)\, d\mu(u).
\end{equation*}
For a given regularization parameter $\varepsilon > 0$, the entropic VQR problem is formulated as
\begin{equation}
\mathsf{T}(\mu,\nu) := \inf_{\pi \in \calQ(\mu,\nu)} \int c \, d\pi + \varepsilon \KL{\pi}{\mu\otimes \nu},
\label{eq: primal}
\end{equation}
where $c(u,y) := \|u-y\|^2/2$ is the ground cost and $\calQ(\mu,\nu)$ is the feasible set 
\begin{equation*}
\calQ(\mu,\nu) := \left \{ \pi \in \Pi(\mu,\nu) : \int_{\calX \times \calY} x \, d\pi_u(x,y) = 0 \ \text{$\mu$-a.e. $u$} \right \}.
\end{equation*} 
The primal problem (\ref{eq: primal}) admits a unique optimal solution $\bar{\pi} \in \calQ(\mu,\nu)$ (cf. Proposition~2.1 in \cite{kato2026entropic}). Throughout the paper, we assume that $\varepsilon > 0$ is fixed and suppress the dependence of various parameters on $\varepsilon$.

Duality plays a crucial role in the development of the Sinkhorn algorithm for entropic OT. For entropic VQR, the dual objective is given by
\begin{equation*}
D(f,g,h) := \int f \,d\mu + \int h \, d\nu - \varepsilon \left ( \iota (f,g,h)-1 \right)
\end{equation*}
with
\begin{equation*}
\iota (f,g,h) :=  \int_{\calU \times \calX \times \calY} \exp \left ( \frac{f(u) +  \langle g(u), x\rangle + h(x,y) - c(u,y)}{\varepsilon}   \right) \,d(\mu\otimes \nu)(u,x,y).
\end{equation*}
The dual problem for (\ref{eq: primal}) reads 
\begin{equation}
\mathsf{D} (\mu,\nu) := \sup_{(f,g,h)} D(f,g,h), 
\label{eq: dual}
\end{equation}
where the supremum is taken over $(f,g,h) \in L^1(\mu) \times L^1(\mu;\R^{d_x}) \times L^1(\nu)$. We call any triplet of functions $(f,g,h)$ achieving the supremum above \textit{dual potentials}. 

The authors' recent work \cite{kato2026entropic} studies duality theory for entropic VQR (see also \cite{carlier2025weak}). We recall the duality results in \cite{kato2026entropic}. Under our assumption, strong duality holds, $\mathsf{T}(\mu,\nu) = \mathsf{D} (\mu,\nu)$, and there exist dual potentials $(\barf,\barg,\barh)$ achieving the supremum in the dual problem (\ref{eq: dual}). Given dual potentials, the optimal primal solution $\bar{\pi}$ has a density (with respect to $\mu \otimes \nu$) of the form 
\begin{equation*}
\frac{d\bar{\pi}}{d(\mu \otimes \nu)} (u,x,y) = \exp \left ( \frac{\barf(u) + \langle \barg(u),x\rangle + \barh(x,y) - c(u,y)}{\varepsilon} \right ).
\end{equation*}
The dual potentials are unique up to an affine shift, i.e., if $(f,g,h)$ is another triplet of dual potentials, then there exist $a \in \R$ and $v \in \R^{d_x}$ such that
\begin{equation*}
\begin{split}
&f(u) = \barf(u) + a, \ g (u)= \barg(u)+ v, \ \text{$\mu$-a.e. $u$}, \\
&h (x,y) = \barh(x,y) -a-\langle v,x \rangle, \ \text{$\nu$-a.e. $(x,y)$.}
\end{split}
\end{equation*}
Finally, for a given triplet of functions $(f,g,h) \in L^1(\mu) \times L^1(\mu;\R^{d_x}) \times L^1(\nu)$, they solve the dual problem (\ref{eq: dual}) if and only if they satisfy the following system of functional equations that are akin to the Schr\"{o}dinger system,
\begin{align}
&f(u) = -\varepsilon \log \int_{\calX \times \calY} \exp \left (\frac{\langle g(u), x\rangle + h(x,y) -c(u,y)}{\varepsilon}\right ) \,d\nu(x,y)\quad \text{$\mu$-a.e. $u$}, \label{eq: f}\\
&\int_{\calX \times \calY} x \exp \left ( \frac{\langle g(u),x\rangle + h(x,y) - c(u,y)}{\varepsilon} \right ) \, d\nu (x,y) = 0 \quad \text{$\mu$-a.e. $u$}, \label{eq: g} \\
&
h(x,y) = -\varepsilon \log \int_{\calU} \exp \left (\frac{f(u) + \langle g(u), x\rangle  -c(u,y)}{\varepsilon}\right ) \,d\mu(u) \quad \text{$\nu$-a.e. $(x,y)$.} \label{eq: h}
\end{align}

In the rest of this paper, we make the following conventions on dual potentials.

\begin{remark}[Conventions on dual potentials]
\label{rem: potential}
Under our assumption, one can choose versions of dual potentials $(f,g,h)$ so that (\ref{eq: f})--(\ref{eq: h}) hold for \textit{all} $u \in \R^{d_y}$ and $(x,y) \in \R^{d_x+d_y}$, respectively, and these versions are smooth functions on $\R^{d_y}, \R^{d_y}$, and $\R^{d_x + d_y}$, respectively. In what follows, we always choose such versions and restrict them to $\calU, \calU$, and $\calX \times \calY$, respectively. In particular, $(f,g,h)$ are continuous on their respective domains, i.e., $(f,g,h) \in \calC(\calU) \times \calC(\calU;\R^{d_x}) \times \calC(\calX \times \calY)$. In addition, we often normalize dual potentials $(f,g,h)$ in such a way that
$\int f \, d\mu = 0$ and $\int g \, d\mu =0$. 
Correspondingly, we define the function spaces
\begin{equation*}
\calC_{\diamond}(\calU) : = \left \{ f \in \calC(\calU) : \int f \, d\mu = 0 \right \} \quad \text{and} \quad \calC_{\diamond}(\calU; \R^{d_x}) : = \left \{ g \in \calC(\calU; \R^{d_x}) : \int g \, d\mu = 0 \right \}.
\end{equation*}
Among $\calC_\diamond(\calU) \times \calC_\diamond(\calU;\R^{d_x}) \times \calC(\calX \times \calY)$, there is a unique triplet of dual potentials, $(\barf,\barg,\barh)$, which satisfies (\ref{eq: f})--(\ref{eq: h}) for all $u \in \calU$ and $(x,y) \in \calX \times \calY$, respectively. 
\end{remark}

\subsection{Quantitative bounds on dual potentials}

We present quantitative bounds on dual potentials; such quantitative estimates will be needed for the subsequent convergence analysis. For entropic OT, such bounds follow directly from the Schr\"{o}dinger system and Jensen's inequality; see, e.g., Lemma~2.1 in \cite{nutz2022entropic}. 
For entropic VQR, such quantitative estimates turn out to be much harder to obtain, due to the implicit functional equation characterizing one dual potential $g$.
In what follows, for functions $f: \calU \to \R, g: \calU \to \R^{d_x}$, and $h: \calX \times \calY \to \R$, we use the notation $\|f\|_\infty := \sup_{u \in \calU}|f(u)|, \| g \|_\infty := \sup_{u \in \calU} \|g(u)\|$, and $\|h\|_\infty := \sup_{(x,y) \in \calX \times \calY}|h(x,y)|$.
In addition, we set
\begin{equation}
M_x : = \sup_{x \in \calX}\|x\| \quad \text{and} \quad \|c\|_\infty := \sup_{(u,y)\in\calU \times \calY} c(u,y). 
\label{eq: initial parameter}
\end{equation}

\begin{proposition}[Quantitative bounds on dual potentials]
\label{prop: potential bound}
Let $(\barf,\barg,\barh) \in \calC_\diamond (\calU) \times \calC_\diamond (\calU;\R^{d_x}) \times \calC(\calX \times \calY)$ be dual potentials.  Then we have
\begin{equation*}
\|\barf\|_\infty \le \| c \|_\infty, \ \  \|\barg\|_\infty \le 2 \|\Sigma_X^{-1}\|_{\op}M_x \left ( \frac{5}{2}\|c\|_\infty + \varepsilon \log \left (\frac{3}{2} \right) \right), \ \  \| \barh \|_\infty \le \| c \|_\infty + \| \barg \|_\infty M_x.
\end{equation*}
\end{proposition}

One key idea in the proof is to find, for any direction $v \in \mathbb{S}^{d_x-1} := \{ x \in \R^{d_x} : \|x\|=1 \}$, a probability measure $q^{(v)}$ on $\calX \times \calY$ whose marginal mean of the first coordinate is proportional to $v$.
Expanding $\KL{q^{(v)}}{\bar{\pi}_u}$ and using nonnegativity of the KL divergence, we obtain an upper bound on $\langle \barg(u),v \rangle$. Choosing $v = \barg(u)/\|\barg(u)\|$ yields an upper bound on $\|\barg(u)\|$. 

\begin{remark}[Scaling of $X$]
\label{rem: scaing}
The dual objective is equivariant under a scale transformation $X \to aX$ for any $a \in \R \setminus \{ 0 \}$. That is, let $\nu^{(a)}$ denote the distribution of $(aX,Y)$, and let $D^{(a)}$ denote the dual objective for $(\mu,\nu^{(a)})$. Then $D^{(a)} (f,g/a,h(\cdot/a,\cdot)) = D(f,g,h)$ for $(f,g,h) \in \calC(\calU) \times \calC(\calU; \R^{d_x}) \times \calC(\calX \times \calY)$. As such, there is no loss of generality in assuming that $M_x=1$ in the analysis below, although we do not impose that normalization.
\end{remark}

\subsection{Local PL inequality for dual objective}
As discussed in the introduction, part of our convergence proofs rests on establishing PL inequalities of the dual objective along the iterates. 
Consider minimizing a smooth function $F: \R^d \to \R$ that is bounded below. A PL inequality refers to an inequality of the form
\[
\| \nabla F(z) \|^2 \ge \alpha \big (  F(z) - \inf F \big), \ \forall z \in \R^d,
\]
for some constant $\alpha > 0$. Under PL inequalities, various optimization algorithms are known to converge linearly \cite{attouch2013convergence,karimi2016linear,bolte2017error,lewis2025complexity}.

For the dual objective functional $D(f,g,h)$, one can define its (negative) gradients in the $f, g$, and $h$ directions as 
\begin{equation}
\begin{split}
\ell_f(f,g,h)(u) &:= e^{f(u)/\varepsilon} \int e^{(\langle g^(u),x\rangle+h(x,y)-c(u,y))/\varepsilon} \, d\nu(x,y) -1, \\
\ell_g(f,g,h) (u) &:= e^{f(u)/\varepsilon} \int x e^{(\langle g(u),x\rangle+h(x,y)-c(u,y))/\varepsilon}\, d\nu(x,y), \\
\ell_h(f,g,h)(x,y) &:= e^{h(x,y)/\varepsilon} \int  e^{(f(u)+\langle g(u),x\rangle-c(u,y))/\varepsilon} \, d\mu(u) -1.
\end{split}
\label{eq: gradient}
\end{equation}
Indeed, observe that
\begin{equation*}
\frac{d}{ds} D(f + s \varphi,g ,h ) \Big|_{s=0}= -\langle \varphi,\ell_f(f,g,h) \rangle_{L^2(\mu)}, \ \varphi \in L^\infty(\mu).
\end{equation*}
The interpretation of other gradients follows similarly. Define a norm on $\calH:=L^2(\mu) \times L^2(\mu;\R^{d_x}) \times L^2(\nu)$ as
\begin{equation}
\| (f,g,h) \|_{\calH} := \sqrt{\|f\|_{L^2(\mu)}^2 +\|g\|_{L^2(\mu)}^2 + \|h\|_{L^2(\nu)}^2 }.
\label{eq: norm}
\end{equation}
The following proposition establishes a PL inequality for the dual objective in an $L^\infty$-neighborhood of the dual potentials, which will suffice for our purpose. 
\begin{proposition}[Local PL inequality for dual objective]
\label{prop: PL}
Let $(f, g, h) \in \calC_\diamond(\calU) \times \calC_\diamond(\calU;\R^{d_x}) \times \calC(\calX \times \calY)$.
Suppose that, for some $K > 0$,
\[
f(u) + \langle g(u), x\rangle + h(x,y) - c(u,y) \ge -K
\]
and
\[
\barf(u) + \langle \barg(u), x\rangle + \barh(x,y) - c(u,y) \ge -K
\]
for all $(u,x,y) \in \calU \times \calX \times \calY$. Then
\begin{equation*}
\big\|(\ell_f(f,g,h),\ell_g(f,g,h),\ell_h(f,g,h))\big\|_{\calH}^2 \ge \frac{2(1 \wedge \underline{\lambda})e^{-K/\varepsilon}}{\varepsilon} \big(D(\barf,\barg,\barh) - D (f,g,h) \big),
\end{equation*}
where $\underline{\lambda}$ is defined by
\begin{equation}
\underline{\lambda} := \text{smallest eigenvalue of $\Sigma_X$}.
\label{eq: eigenvalue}
\end{equation}
\end{proposition}

In addition to the PL inequality, we will establish slope-ascent conditions for both algorithms, which are algorithm-specific.\footnote{We borrow the term ``slope-ascent conditions'' from \cite{lewis2025complexity}.} Verifying the PL inequality and slope-ascent conditions for the algorithms requires establishing uniform-in-time bounds on the iterates, which shares the bulk of the convergence proofs

\section{Sinkhorn algorithm}
\label{sec: sinkhorn}

 For standard entropic OT, the Sinkhorn algorithm iteratively solves the Schr\"{o}dinger system, which corresponds to the Euler-Lagrange equations for the dual objective. Viewing the Schr\"{o}dinger-like system (\ref{eq: f})--(\ref{eq: h}) as the Euler-Lagrange equations for the (concave) dual objective $D(f,g,h)$, one can directly adapt the Sinkhorn algorithm to entropic VQR.

\begin{definition}[Sinkhorn algorithm for entropic VQR]
Start from $(f^0,g^0,h^0) \in \calC_{\diamond}(\calU) \times \calC_{\diamond}(\calU;\R^{d_x}) \times \calC(\calX \times \calY)$. At the $t$-th iterate, we update $(f^{t},g^{t},h^{t})$ as follows.
\begin{enumerate}[leftmargin=0.25in]
\item[1.] For each $u \in \calU$, find a vector $\tildeg^{t+1}(u) \in \R^{d_x}$ that satisfies
\begin{equation}
\int x \exp \left ( \frac{\langle \tildeg^{t+1}(u),x\rangle + h^{t}(x,y) - c(u,y)}{\varepsilon} \right ) \, d\nu (x,y) = 0.
\label{eq: update g}
\end{equation}
\item[2.] Set
\begin{equation*}
\begin{split}
\tildef^{t+1}(u) &= -\varepsilon \log \int \exp \left ( \frac{\langle \tildeg^{t+1}(u),x\rangle + h^{t}(x,y) - c(u,y)}{\varepsilon} \right ) \, d\nu (x,y), \ u \in \calU, \\
\tildeh^{t+1}(x,y) &= - \varepsilon \log \int \exp \left ( \frac{\tildef^{t+1}(u) +  \langle \tildeg^{t+1}(u),x\rangle- c(u,y)}{\varepsilon} \right ) \, d\mu (u), \ (x,y) \in \calX \times \calY.
\end{split}
\end{equation*}
\item[3.] Finally, normalize $(\tildef^{t+1},\tildeg^{t+1},\tildeh^{t+1})$ as 
\begin{equation*}
\begin{split}
&f^{t+1}(u) = \tildef^{t+1}(u)-\int \tildef^{t+1} \, d\mu, \quad g^{t+1}(u) = \tildeg^{t+1} (u)-  \int \tildeg^{t+1} \, d\mu, \\
&h^{t+1}(x,y) = \tildeh^{t+1}(x,y) + \int \tildef^{t+1} \, d\mu + \left \langle \int \tildeg^{t+1} \, d\mu, x \right \rangle. 
\end{split}
\end{equation*}
\end{enumerate}
\end{definition}

\begin{remark} 
\label{rem: sinkhorn}
Several remarks on the Sinkhorn algorithm are in order.
\begin{enumerate}[leftmargin=0.25in]
    \item[(i)] From the proofs of Proposition~2.7 and Theorem 2.2 in \cite{kato2026entropic}, one can show that there exists a unique vector $\tildeg^{t+1} (u) \in \R^{d_x}$ satisfying (\ref{eq: update g}) for each $u \in \calU$, and $\tildeg^{t+1}$ is continuous on $\calU$. In addition, by construction, $\tildef^{t+1}$ and $\tildeh^{t+1}$ are continuous. In particular, this implies that the integrals $\int \tildef^{t+1} \, d\mu$ and $\int \tildeg^{t+1} \, d\mu$ are well-defined and finite. 
    \item[(ii)]
    Alternatively, the iterates can be characterized as 
\begin{equation}
\begin{split}
&\tildeg^{t+1} \in \argmax_{g \in \calC(\calU;\R^{d_x})} D(f^{t},g,h^{t}), \\
&\tildef^{t+1} \in \argmax_{f \in \calC(\calU)} D(f,\tildeg^{t+1},h^{t}), \\
&\tildeh^{t+1} \in \argmax_{h \in \calC(\calX \times \calY)} D(\tildef^{t+1},\tildeg^{t+1},h).
\end{split}
\label{eq: interpretation}
\end{equation}
The update order for $\tildeg^{t+1}$ and $(\tildef^{t+1},\tildeh^{t+1})$ can be interchanged. In addition, within the updates for $(\tildef^{t+1},\tildeh^{t+1})$, the order can be interchanged. Namely, one may first update $h^t$ and then $f^t$. These modifications retain linear convergence. 
\item[(iii)] The normalization step serves two purposes. First, the dual objective $D(f,g,h)$ is invariant under an affine shift, i.e., the value of $D(f,g,h)$ does not change even if we replace $(f,g,h)$ with $(f+a,g+v,h-a-\langle v,x\rangle)$ for any $a \in \R$ and $v \in \R^{d_x}$, so a suitable normalization is needed to guarantee local strong concavity of $D(f,g,h)$ around $(\barf,\barg,\barh)$  (cf. the proof of Proposition~\ref{prop: PL}). Related to the first point, without a proper normalization, the Sinkhorn iterates would be unbounded as $t$ grows. 
\item[(iv)] An obvious drawback of the (vanilla) Sinkhorn algorithm described above is the need to solve the implicit functional equation (\ref{eq: update g}) at each iteration. 
This difficulty will be addressed in the next section, where we combine the Sinkhorn iteration with one-step gradient ascent applied to solving (\ref{eq: update g}) at each step. 
\end{enumerate}
\end{remark}

We shall study convergence of the Sinkhorn algorithm. 
As discussed at the end of the previous section, key to our convergence analysis is to establish uniform-in-time bounds on the Sinkhorn iterates. 
%Since the dual objective $D(f,g,h)$ is essentially governed by the exponential function, whose second derivative is bounded and bounded away from zero on bounded sets, one may expect that linear convergence would follow once we can verify that the Sinkhorn iterates are uniformly bounded over $t$. 
For standard entropic OT, such estimates follow rather directly; see Lemma~3.1 in \cite{carlier2022linear}. 
Similar to Proposition~\ref{prop: potential bound}, however, obtaining quantitative bounds on the Sinkhorn iterates is much harder for entropic VQR. 
To present the result, we set
\begin{equation}
L_c := \sup_{u \in \calU}\|u\| + \sup_{y \in \calY}\|y\| \quad \text{and} \quad D_0 := D (f^0,g^0,h^0). 
\label{eq: parameter}
\end{equation}
\begin{proposition}[Quantitative bounds on Sinkhorn iterates]
\label{prop: sinkhorn potential}
For  $t \in \N_0$,
\begin{equation*}
\begin{split}
&\big \| \tildef^{t+1} \big \|_\infty \vee \big \| f^{t+1} \big \|_\infty \le L_c \diam (\calU) + \|c\|_\infty + D_0^- + \| h^0 \|_\infty=: K_f, \\
&\big \|\tildeg^{t+1}\big \|_\infty \vee \big \| g^{t+1} \big \|_\infty\le 4 \|\Sigma_X^{-1} \|_{\op} M_x  \left ( 4 \| c \|_\infty -\frac{3}{2} D_0 + K_f + \| h^0 \|_\infty + \varepsilon \log \left ( \frac{3}{2} \right )\right ) =: K_g, \\
&\big \| \tildeh^{t+1} \big \|_\infty \vee \big \| h^{t+1} \big \|_\infty \le \| c \|_\infty + K_f + K_g M_x =: K_h.
\end{split}
\end{equation*}
\end{proposition}

The initial dual objective value $D_0$ appears on the bounds because one needs to find bounds on $\int \tildef ^{t+1}\, d\mu$ and $\| h^t \|_{L^1(\nu)}$. To this end, we use monotonicity of the dual objective values along the iterates (cf. Lemma~\ref{lem: monotonicity} below) and weak duality.

By adjusting the constants if necessary, we will assume, without loss of generality, 
\begin{equation*}
M_x \ge 1, \ K_f \ge \| \barf \|_\infty \vee \| f^0 \|_\infty, \ K_g \ge \| \barg \|_\infty \vee \| g^0 \|_\infty, \ K_h \ge \| \barh \|_\infty \vee \| h^0 \|_\infty.
\end{equation*}
Recall that quantitative bounds on $(\barf,\barg,\barh)$ are available from Proposition~\ref{prop: potential bound}. We set 
\begin{equation*}
\bar{K} := K_f  + K_gM_x + K_h + \| c \|_\infty.
\end{equation*}
Recall the notation $\underline{\lambda}$ in (\ref{eq: eigenvalue}). 
We are ready to present the linear convergence result for the (vanilla) Sinkhorn algorithm. 

\begin{theorem}[Linear convergence of Sinkhorn algorithm]
\label{thm: linear conv}
Let $(\barf,\barg,\barh) \in \calC_\diamond(\calU) \times \calC_\diamond(\calU;\R^{d_x}) \times \calC(\calX \times \calY)$ be dual potentials. 
For $t \in \N_0$, we have
\begin{equation*}
\bar{D} - D (f^{t},g^{t},h^{t}) \le \big(1+\tau\big)^{-t} \big ( \bar{D} - D_0 \big),
\end{equation*}
where $\bar{D} := D(\barf,\barg,\barh)$ and $\tau > 0$ is given by
\begin{equation}
\tau := \frac{(1 \wedge \underline{\lambda})e^{-5\bar{K}/\varepsilon} }{3M_x^2}.
\label{eq: tau}
\end{equation}
Furthermore, as $t \to \infty$, we have
\begin{equation*}
\big \| f^{t} - \barf \big \|_{L^2(\mu)}^2 + \big \| g^{t} - \barg \big \|_{L^2(\mu)}^2 + \big \| h^{t} - \barh \big \|_{L^2(\nu)}^2 = O\big( \big(1+\tau\big)^{-t} \big).
\end{equation*}
\end{theorem}

\begin{remark}
The proof shows that
\[
\alpha_0 \big(\big \| \tildef^{t+1} - f^{t} \big \|_{L^2(\mu)}^2 + \big \| \tildeg^{t+1} - g^t \big \|_{L^2(\mu)}^2 + \big \| \tildeh^{t+1} - h^{t} \big \|_{L^2(\nu)}^2\big) \le D (f^{t+1},g^{t+1},h^{t+1})-D (f^{t},g^{t},h^{t})
\]
for some $\alpha_0 > 0$; see Lemma~\ref{lem: ascent} below. Hence, we also have 
\begin{equation*}
\big \| \tildef^{t+1} - \barf \big \|_{L^2(\mu)}^2 + \big \| \tildeg^{t+1} - \barg \big \|_{L^2(\mu)}^2 + \big \| \tildeh^{t+1} - \barh \big \|_{L^2(\nu)}^2 = O\big( \big(1+\tau\big)^{-t} \big).
\end{equation*}
\end{remark}

The theorem presents convergence rates,  with explicit constants, for the duality gap and for the dual potentials, although we made no attempt to optimize the constants. The contraction rate in Theorem \ref{thm: linear conv} is reminiscent of the linear convergence results of the Sinkhorn algorithm for standard entropic OT; cf. \cite{peyre2019computational} and \cite{carlier2022linear}. The exponential dependence of the contraction rate on $1/\varepsilon$  appears to be tight in most difficult cases; see Remark 4.15 in \cite{peyre2019computational}. Several recent works established contraction rates that avoid exponential dependence on $1/\varepsilon$, albeit with slower rates \cite{altschuler2017near, chakrabarty2021better,dvurechensky2018computational, ghosal2025convergence} or for a restricted class of marginals \cite{chizat2026sharper}. Extending such analyses to entropic VQR is left for future research.

\section{Modified Sinkhorn algorithm}
\label{sec: modified sinkhorn}

\subsection{Modified Sinkhorn algorithm}
As pointed out in Remark \ref{rem: sinkhorn}~(iv) above, an obvious drawback of the (vanilla) Sinkhorn algorithm is the need to solve the implicit functional equation (\ref{eq: update g}) at each iteration, which would require a subroutine. However, resorting to a numerical subroutine in Step 1 only guarantees that the key mean-zero constraint (\ref{eq: update g}) holds approximately but not exactly, which leaves a theory-implementation gap. To account for that,   we replace Step 1 in the Sinkhorn algorithm with one-step gradient ascent, and establish linear convergence of the modified scheme, thereby closing the theory-implementation gap.  The modified scheme is described as follows.

\begin{definition}[Modified Sinkhorn algorithm for entropic VQR]
\label{def:modified-sinkhorn}
Start from $(f^0,g^0,h^0) \in \calC_{\diamond}(\calU) \times \calC_\diamond(\calU;\R^{d_x}) \times \calC(\calX \times \calY)$. Let $\eta >0$.
At the $t$-th iterate, we update $(f^{t},g^{t},h^{t})$ as follows. 
\begin{enumerate}
\item[1.] Set 
\begin{equation*}
\begin{split}
\tildeh^{t+1}(x,y) &= - \varepsilon \log \int \exp \left ( \frac{f^{t}(u) +  \langle g^{t}(u),x\rangle- c(u,y)}{\varepsilon} \right ) \, d\mu (u), \ (x,y) \in \calX \times \calY, \\
\tildef^{t+1}(u) &= -\varepsilon \log \int \exp \left ( \frac{\langle g^{t}(u),x\rangle + \tildeh^{t+1}(x,y) - c(u,y)}{\varepsilon} \right ) \, d\nu (x,y), \ u \in \calU.
\end{split} 
\end{equation*}
\item[2.] Compute 
\begin{equation*}
% \tildeg^{t+1} = \proj_{\calK} \left ( g^t - \eta \hat{d}_g^{t} \right).
\tildeg^{t+1} = g^t - \eta d_g^{t},
\end{equation*}
where 
\begin{equation*}
d_g^{t} (u)= e^{\tildef^{t+1}(u)/ \varepsilon} \int x e^{(\langle g^{t}(u),x\rangle+\tildeh^{t+1}(x,y)-c(u,y))/\varepsilon} \, d\nu(x,y) \in \R^{d_x},  \ u \in \calU.
\end{equation*}
\item[3.] Finally, normalize $(\tildef^{t+1},\tildeg^{t+1},\tildeh^{t+1})$ as
\begin{equation*}
\begin{split}
&f^{t+1}(u) = \tildef^{t+1}(u) - \int \tildef^{t+1} \, d\mu, \quad g^{t+1} (u)= \tildeg^{t+1}(u) - \int \tildeg^{t+1} \,d\mu, \\
&h^{t+1}(x,y) = \tildeh^{t+1}(x,y) + \int \tildef^{t+1} \, d\mu + \left\langle \int \tildeg^{t+1}\, d\mu,  x \right\rangle.
\end{split}
\end{equation*}
\end{enumerate}
\end{definition}

\begin{remark}
A few comments on the modified Sinkhorn algorithm are in order. 
\begin{enumerate}
\item[(i)] The vector-valued function $d_g^t$ can be interpreted as a negative gradient of the dual objective with respect to $g$ evaluated at $(\tildef^{t+1},g^t,\tildeh^{t+1})$. Indeed, recalling the notation in (\ref{eq: gradient}), one has
\[
d_g^t = \ell_g(\tildef^{t+1},g^t,\tildeh^{t+1}). 
\]
Hence, $\tildeg^{t+1}$ corresponds to one-step gradient ascent applied to the objective $g \mapsto D(\tildef^{t+1},g,\tildeh^{t+1})$.
\item[(ii)] The update order can be interchanged with some adjustment of the normalization step. The current order is chosen so as to facilitate the convergence proof. In particular, as $\int f^t \, d\mu = 0$ and $\int g^t \, d\mu = 0$ by construction, Jensen's inequality yields $\tildeh^{t+1} \le \| c \|_\infty$, which will be frequently used in the proof of Proposition~\ref{prop: sinkhorn potential modified} below.
\end{enumerate}

\end{remark}

Similar to the (vanilla) Sinkhorn case, the main effort in the convergence analysis is devoted to establishing uniform-in-time bounds on the modified Sinkhorn iterates, in particular, on the iterates in the $g$-direction. In the previous case, $\tildeg^{t+1}$ satisfies the exact mean-zero constraint (\ref{eq: update g}), so that a similar technique from the proof of Proposition~\ref{prop: potential bound} can be adapted to establish the desired estimates. Such an analysis cannot be directly extended to the modified Sinkhorn case. In addition, it is \textit{a priori} unclear whether suitable coercivity estimates hold for the dual objective in the $g$-direction. These aspects constitute a new challenge. 
Before presenting the result, 
recall the notation $M_x$ defined in (\ref{eq: initial parameter}), and set $L_c:=\sup_{u \in \calU}\|u\| + \sup_{y \in \calY}\|y\|$ and $D_0:=D(f^0,g^0,h^0)$ as in the previous section. Without loss of generality, we will assume that $M_x \ge 1$. 

\begin{proposition}[Quantitative bounds on modified Sinkhorn iterates]
\label{prop: sinkhorn potential modified}
Suppose that $0 < \eta < \frac{2\varepsilon}{M_x^2}$. For $t \in \N_0$, we have
\[
\begin{split}
\big \| \tildeg^{t+1} \big \|_{\infty} \vee \big \| g^{t+1} \big \|_{\infty} &\le  \eta M_x +\left [  \| g^0 \|_\infty \vee \left ( \frac{2 K_g^{\star \star}+6\eta (1 \vee \delta) M_x^2}{\delta} \right )\right] \\
&\quad + \frac{2e^{1/2}}{\delta} \left ( 1-\frac{\eta M_x^2}{2\varepsilon} \right)^{-1} \left ( 1+\frac{2M_x}{\delta}  \right )(\bar{D}-D_0) =: K_g^{\star},
\end{split}
\]
where $\delta := 1/(2\|\Sigma_X^{-1}\|_{\op}M_x)$ and $K_g^{\star \star}$ is given by
\[
K_g^{\star \star}:= 5\|c\|_\infty -\frac{3}{2}D_0 + \varepsilon \log \left ( \frac{3}{2} \right ). 
\]
In addition, we have
\[
\begin{split}
&\| \tildef^{t+1} \|_\infty \vee \| f^{t+1} \|_\infty \le L_c \diam (\calU) + 2K_g^{\star}M_x + \|c\|_\infty - D_0 =: K_f^{\star}, \\
&\| \tildeh^{t+1} \|_\infty \vee \| h^{t+1} \|_\infty \le \|c\|_\infty + 2K_f^{\star}+2K_g^{\star}M_x =: K_h^{\star}.
\end{split}
\]
\end{proposition}
Again, without loss of generality, we will assume
\begin{equation*}
K_f^\star \ge \| \barf \|_\infty \vee \| f^0 \|_\infty, \ K_g^\star \ge \| \barg \|_\infty \vee \| g^0 \|_\infty, \ K_h^\star \ge \| \barh \|_\infty \vee \| h^0 \|_\infty.
\end{equation*}

The step size condition, $0 < \eta < \frac{2\varepsilon}{M_x^2}$, guarantees that the dual objective monotonically increases along the iterates. Specifically, in Lemma~\ref{lem: monotonicity modified}, we will establish the following coercivity estimate:
\begin{equation}
D(\tildef^{t+1},\tildeg^{t+1},\tildeh^{t+1}) - D(\tildef^{t+1},g^{t},\tildeh^{t+1}) \ge e^{-1/2} \eta \left ( 1-\frac{\eta M_x^2}{2\varepsilon} \right) \| d_g^t \|_{L^2(\mu)}^2,
\label{eq: coercivity}
\end{equation}
which plays a key role in the proof of Proposition~\ref{prop: sinkhorn potential modified}. The preceding step size condition is sharp, in the sense that one can find marginals for which, when $\eta > \frac{2\varepsilon}{M_x^2}$, one iteration of the modified Sinkhorn algorithm strictly \textit{decreases} the dual objective; see Appendix \ref{sec: counterexample}. 

We now present the linear convergence result for the modified Sinkhorn algorithm. 

\begin{theorem}[Linear convergence of modified Sinkhorn algorithm]
\label{thm: linear conv modified}
Suppose that $0 < \eta < \frac{2\varepsilon}{M_x^2}$. There exists $\tau^{\star} > 0$ such that 
\[
\bar{D} - D(f^t,g^t,h^t) \le (1+\tau^{\star})^{-t}\big (\bar{D}-D_0), \ t \in \N_0,
\]
where $\bar{D} := D(\barf,\barg,\barh)$. 
Furthermore, as $t \to \infty$,
\[
\big \| f^t-\barf \big\|_{L^2(\mu)}^2 + \big\| g^t-\barg \big\|_{L^2(\mu)}^2 + \big\| h^t - \barh \big\|_{L^2(\nu)}^2 = O\big( (1+\tau^{\star})^{-t}\big).
\]
\end{theorem}

\begin{remark}
The proof shows that
\[
\alpha_0 \big (\| \tildef^{t+1}-f^t\|_{L^2(\mu)}^2 + \| \tildeg^{t+1}-g^t \|_{L^2(\mu)}^2 + \| \tildeh^{t+1} - h^t \|_{L^2(\nu)}^2 \big )
\le D(f^{t+1},g^{t+1},h^{t+1}) - D(f^t,g^t,h^t) 
\]
for some constant $\alpha_0 > 0$; see Lemma~\ref{lem: slope ascent modified} below. Hence, we also have
\[
\big \| \tildef^{t+1}-\barf\big\|_{L^2(\mu)}^2 + \big\| \tildeg^{t+1}-\barg \big \|_{L^2(\mu)}^2 + \big \| \tildeh^{t+1} - \barh \big\|_{L^2(\nu)}^2 = O\big( (1+\tau^{\star})^{-t}\big).
\]
\end{remark}

An inspection of the proof shows that $\tau^\star$ can be chosen 
as
\[
\tau^{\star} = \frac{2(1 \wedge \underline{\lambda})e^{-\bar{K}^{\star}/\varepsilon}}{\varepsilon} \times \left [ \frac{e^{-2K_h^{\star}/\varepsilon}}{2\varepsilon} \bigwedge \left (  \frac{e^{-1/2}}{M_x^2\eta} \left ( 1-\frac{\eta M_x^2}{2\varepsilon} \right) \right ) \right ] \times \left ( \frac{2}{\underline{\lambda}\eta^2} + \frac{5M_x^2e^{2\bar{K}^{\star}/\varepsilon}}{\varepsilon^2} \right)^{-1},
\]
where $\bar{K}^{\star} := K_f^{\star}+K_g^{\star}M_x + K_h^{\star} + \| c \|_\infty$. The exact value of $\tau^\star$ is likely an artifact of the proof technique, but we do not pursue the best possible constant for the sake of simplicity of the analysis. 

As for the step size $\eta$, the coercivity estimate in (\ref{eq: coercivity}) implies that 
\[
\eta = \frac{\varepsilon}{M_x^2}
\]
would be a reasonable choice. With this choice of $\eta$, $\tau^\star$ is bounded below by a constant multiple of 
\[
\frac{(1 \wedge \underline{\lambda})e^{-3\bar{K}^\star/\varepsilon}}{2M_x^4/\underline{\lambda} + 5M_x^2e^{2\bar{K}^\star/\varepsilon}},
\]
which is further bounded below, up to a constant, by 
\[
\frac{(1 \wedge \underline{\lambda})e^{-5\bar{K}^\star/\varepsilon}}{M_x^2}
\]
when $\varepsilon$ is small enough. As such, the size of $\tau^\star$ is comparable to the vanilla Sinkhorn case from Theorem \ref{thm: linear conv} when $\varepsilon$ is small.

\subsection{Computing derivatives of $\barf$ and $\barg$}
\label{sec: derivative}

In applications, one would be interested in computing the derivatives of $\barf$ and $\barg$. However, in practice, we run the algorithm for discrete or discretized distributions, so directly evaluating the derivatives of $f^t$ and $g^t$ would be nontrivial. In the following, we will derive alternative expressions for the derivatives of $\barf$ and $\barg$, which provide a simple way to approximately compute them without the need to directly differentiate $f^t$ and $g^t$.

Suppose that $\mu$ is absolutely continuous with support $\calU$ being (compact and) convex, and that $(X,Y)$ satisfies a quasi-linear representation of the form 
\begin{equation*}
Y  = \beta_0(U) + \beta_1(U)^\top X, \ U \sim \mu, \ \E[X \mid U] = 0, \ \text{a.s.}
\end{equation*}
for some mappings $\beta_0: \calU \to \R^{d_y}$ and $\beta_1: \calU \to \R^{d_x \times d_y}$ such that  $u \mapsto \beta_0(u) + \beta_1(u)^\top x$ agrees with the gradient of a convex function $\mu$-a.e.~$u$ for each $x \in \calX$; cf. Section 3.2 in \cite{carlier2016vector}.
Theorem 3.2 in \cite{carlier2016vector} provides a characterization of $\beta_0$ and $\beta_1$ using solutions to the dual problem for the unregularized VQR problem (\ref{eq: vqr}), which reads
\begin{equation}
\sup_{(\varphi,\psi,\xi)} \int_{\calU} \varphi \, d\mu + \int_{\calX \times \calY} \xi \, d\nu,
\label{eq: dual VQR}
\end{equation}
where the supremum is taken over all $(\varphi,\psi,\xi) \in \calC(\calU) \times \calC(\calU; \R^{d_x}) \times \calC(\calX \times \calY)$ satisfying $ \varphi (u) + \langle \psi(u),x \rangle + \xi(x,y) \le c(u,y)$ for all $(u,x,y) \in \calU \times \calX \times \calY$.
Under regularity conditions, Theorem 3.2 in \cite{carlier2016vector} shows that $\beta_0$ and $\beta_1$ agree with $u - \nabla \varphi$ and $- J\psi$, respectively, where $(\varphi,\psi,\xi)$ solve the dual problem (\ref{eq: dual VQR}). Here $\nabla \varphi$ and $J\psi$ are the gradient and Jacobian matrix of $\varphi$ and $\psi$, respectively, i.e.,
\begin{equation*}
\nabla \varphi(u) := \left ( \frac{\partial \varphi(u)}{\partial u_j}\right )_{1 \le j \le d_y} \in \R^{d_y} \quad \text{and} \quad 
J\psi(u) := \left ( \frac{\partial \psi_i(u)}{\partial u_j}\right )_{\substack{1 \le i \le d_x \\ 1 \le j \le d_y}} \in \R^{d_x \times d_y}. 
\end{equation*}

As such, entropic analogs of $\beta_0(u)$ and $\beta_1(u)$ can be defined by $u-\nabla \barf(u)$ and $-J\barg(u)$.
Recall that $\barf$ and $\barg$ can be extended to smooth functions on $\R^{d_y}$ that satisfy (\ref{eq: f}) and (\ref{eq: g}) for all $u \in \R^{d_y}$ (cf. Remark \ref{rem: potential}). Observe that $\bar{\pi}_u$ has a version of the form
\begin{equation*}
\frac{d\bar{\pi}_u}{d\nu} (x,y) = e^{(\barf(u)+\langle \barg(u),x\rangle + \barh(x,y)-c(u,y))/\varepsilon}. 
\end{equation*}
Differentiating both sides of (\ref{eq: f}) and (\ref{eq: g}) with respect to $u$ and using $\int x \, d\pi_u = 0$, we arrive at the expressions
\begin{equation*}
\nabla \barf(u) =u -  \underbrace{\E_{\bar{\pi}_u} \big [ \tilde{Y}\big]}_{=:B_0(u)}\quad \text{and} \quad J\barg(u) = - \underbrace{\Big ( \E_{\bar\pi_u} \big [ \tilde{X}\tilde{X}^\top\big ]\Big)^{-1}\E_{\bar\pi_u} \big [ \tilde{X}\tilde{Y}^\top\big ] }_{=:B_1(u)},
\end{equation*}
where $\E_{\bar{\pi}_u}$ means that the expectation is taken with respect to $(\tilde{X},\tilde{Y}) \sim \bar{\pi}_u$. Nonsingularity of the matrix $\E_{\bar\pi_u} \big [ \tilde{X}\tilde{X}^\top\big ]$ follows by that of $\E[XX^\top]$ and the fact that 
\[
\barf(u)+\langle \barg(u),x\rangle + \barh(x,y)-c(u,y) \ge \barf(u) - M_x \| \barg(u) \| - \|\barh\|_\infty - \sup_{y \in \calY}c(u,y) > -\infty.
\]
The expressions for $B_0$ and $B_1$ above are well-defined even without the assumptions that $\mu$ is absolutely continuous and $\calU$ is convex. Below, we will consider computing $B_0$ and $B_1$ using the modified Sinkhorn algorithm. 

To approximate $\bar{\pi}$, we use $\pi^t$, where
\begin{equation*}
\frac{d\pi^t}{d(\mu \otimes \nu)} (u,x,y) = e^{(\tilde{f}^{t+1}(u)+\langle g^t(u),x\rangle + \tildeh^{t+1}(x,y)-c(u,y))/\varepsilon},
\end{equation*}
which has marginal $\mu$ on $\calU$ by construction, so $\pi^t_u$ has a density
\begin{equation*}
\frac{d\pi^t_u}{d\nu}(x,y) = e^{(\tilde{f}^{t+1}(u)+\langle g^t(u),x\rangle + \tildeh^{t+1}(x,y)-c(u,y))/\varepsilon}. 
\end{equation*}
As such, one can approximately compute $B_0$ and $B_1$ by
\begin{equation*}
B_0^t(u) = \E_{\pi_u^t}\big[\tilde{Y}\big] \quad \text{and} \quad B_1^t (u)= \Big (\E_{\pi_u^t}\big[\tilde{X}\tilde{X}^\top\big ] \Big)^{-1}\E_{\pi_u^t}\big[\tilde{X}\tilde{Y}^\top\big ] .
\end{equation*}
We shall establish the following guarantee for $B_0^t$ and $B_1^t$. 
\begin{proposition}
\label{prop: derivative}
Under the setting of Theorem \ref{thm: linear conv modified}, as $t \to \infty$, we have 
\begin{equation*}
\big \| B_0^t- B_0 \big\|_{L^2(\mu)}^2 = O\big ( (1+\tau^{\star}\big )^{-t/2} \big) \quad \text{and} \quad \big\|B_1^t - B_1 \big\|_{L^2(\mu)}^2 = O\big ( (1+\tau^{\star}\big )^{-t/2}\big),
\end{equation*}
where $\| B_1^t - B_1 \|_{L^2(\mu)}^2 = \int \| B_1^t(u) - B_1(u) \|_{\op}^2 \, d\mu(u)$. 
\end{proposition}

The proof relies on the results from Theorem \ref{thm: linear conv modified} and Pinsker's inequality (cf. Theorem 7.10 in \cite{polyanskiy2025information}).

\section{Numerical experiments}
\label{sec: numerical}

In this section, we study the empirical behavior of the modified Sinkhorn
algorithm. We shall focus on the modified algorithm as its exact implementation adheres to a rigorous convergence guarantee (cf. the discussion at the start of the previous section). The experiments have the following objectives.  First, we examine whether
the dual objective exhibits linear convergence predicted by
Theorem \ref{thm: linear conv modified}, and how the observed contraction changes with the regularization
level and the covariate dimension.  Second, we assess conservativeness of the
sufficient step size condition given in Theorem \ref{thm: linear conv modified}. 
Appendix \ref{sec: additional experiments} contains numerical comparisons between the vanilla and modified Sinkhorn algorithms and an experiment pertaining to the problem considered in Section \ref{sec: derivative}. All experiments were carried out using the programming language \texttt{Julia} \cite{bezanson2017julia}.

\subsection{Discrete implementation}
\label{sec: discrete implementation}

Algorithm 1 below details the implementation of the modified Sinkhorn algorithm when the marginals $\mu$ and $\nu$ are discrete with $m$ and $n$ atoms, represented as probability simplex vectors  $a \in \Delta_m$ and $b \in \Delta_n$, respectively, where $\Delta_k := \{p \in \R^k_{\ge 0} : \sum_{i=1}^k p_i = 1\}$. All logarithms and exponentials of vectors or matrices are understood componentwise.
In addition,  we preprocess the design matrix $\mathbb{X}$ so that $b^\top \mathbb{X}=0$.

At each iteration $t$, updating the dual potentials $f^t$ and $h^t$ involves evaluating the matrices $H$ and $F$, which entails $O(m n d_x)$ arithmetic operations.
Similarly, computing the gradient direction $D_g^t$ entails $O(m n d_x)$  operations. 
 As such, the computational cost of Algorithm 1 is $O(mnd_x)$ per iteration, and if $d_x$ is treated as constant, it is $O(mn)$ per iteration, which is comparable to the (standard) Sinkhorn algorithm (cf. Chapter 4 in \cite{peyre2019computational}).  The complexity for finding $\delta_0$-optimal solutions for the dual objective is 
    \[
    O \left ( mnd_x \frac{\log ((\bar{D}-D_0)/\delta_0)}{\log (1+\tau^\star)} \right )
    \]
 under the notation of Theorem 4.1.

 \medskip

{\small 
\begin{algorithm}[H]
\caption{Discrete Modified Sinkhorn (matrix form)}
\KwIn{$a\in\Delta_m$, $b\in\Delta_n$, $\mathbb{X}\in\mathbb R^{n\times d_x}$, $C\in\mathbb R^{m\times n}$, $\varepsilon,\eta>0$}
\KwIn{$f^0\in\mathbb R^m$, $ G^0\in\mathbb R^{m\times d_x}$, $ h^0\in\mathbb R^n$}

\For{$t=0,1,2,\dots$}{
    \tcp*[l]{Update $f$ and $h$}
    $Q^t \leftarrow G^t\mathbb{X}^\top$\; 
    $H \leftarrow \bigl(f^t\mathbf 1_n^\top +  Q^t - C\bigr)/\varepsilon$\;
    $\tilde{h}^{t+1} \leftarrow -\varepsilon \log\!\bigl(\exp(H)^\top a\bigr)$\;
    $F \leftarrow \bigl( Q^t + \mathbf 1_m (\tilde{h}^{t+1})^\top - C\bigr)/\varepsilon$\;
    $\tilde{f}^{t+1} \leftarrow -\varepsilon \log\!\bigl(\exp(F)b\bigr)$\;
    \tcp*[l]{Update $G$}
    $W \leftarrow \exp (F)\operatorname{diag} (b)$\;
    $D_g^t \leftarrow \operatorname{diag}\!\bigl(e^{\tildef^{t+1}/\varepsilon}\bigr)\,W\mathbb{X}$\;
    $\tilde{G}^{t+1} \leftarrow  G^t - \eta D_g^t$\;
    \tcp*[l]{Normalization}
    $\alpha \leftarrow a^\top \tilde{f}^{t+1}$, $\beta^\top \leftarrow a^\top \tilde{G}^{t+1}$\; 
    $f^{t+1} \leftarrow \tilde{f}^{t+1} - \alpha \mathbf 1_m$,
    $G^{t+1} \leftarrow \tilde{G}^{t+1} - \mathbf 1_{m} \beta^\top$,
    $h^{t+1} \leftarrow \tilde{h}^{t+1} + \alpha \mathbf 1_n + \mathbb{X} \beta$\;
}
\end{algorithm}
}

\subsection{Experimental settings}
In all experiments in this section, we set $d_y = 2$ and use the empirical distribution of $m=1000$ independent draws from $\operatorname{Unif}([0,1]^2)$ as the reference measure $\mu$.\footnote{The random seed used in the experiments was set to 123 for each figure.} As for $\nu$, we consider the following synthetic and real data settings.

\textbf{Synthetic Gaussian data.}
We set
\[
  \nu=\frac1n\sum_{i=1}^n\delta_{(X_i,Y_i)},
  \qquad
  (X_i,Y_i)\stackrel{\mathrm{i.i.d.}}{\sim}\calN(0,\Sigma),
\]
where $n=1000$ and $\Sigma=(0.5^{|j-k|})_{1\le j,k\le d_x+d_y}$.
We consider $d_x\in\{1,4,16,64\}$.

\textbf{ANSUR II data.}
We use the male sample from the U.S. Army Anthropometric Survey II (ANSUR II)\footnote{Available from \url{https://www.openlab.psu.edu/ansur2/}.},
which contains $n=4082$ observations.  Following the two-dimensional VQR
illustration in~\cite{carlier2022vector}, we take
\[
  Y=(\texttt{weight},\texttt{thigh circumference})^\top\in\mathbb{R}^2.
\]
For the covariates, we use the first 64 anthropometric measurement variables in
the data set.  For each $d_x\in\{1,4,16,64\}$, $X$ consists of the first $d_x$
of these variables; neither response variable is included in the
covariates.  The empirical distribution of $(X,Y)$ is used for $\nu$.

In these experiments, we further preprocess the design matrix so that $M_x=1$ (cf. Remark \ref{rem: scaing}). We always use the zero initialization, $(f^0,G^0,h^0) = (0,O,0)$. For $t \in \N_0$, let $D_t = D(f^t,g^t,h^t)$. 

\subsection{Log duality gaps and contraction factors}

We first look at empirical convergence of the duality gaps and contraction factors
\[
\rho_t = \frac{D_{t}-D_{t-1}}{D_{t-1}-D_{t-2}}
\]
which would approach some constant under linear convergence.
 As for the step size, we use
\[
\eta = \varepsilon
\]
as suggested in the previous section (recall that we preprocessed the data so that $M_x=1$). 
In addition, since no closed-form expression for the optimal dual value is available in either
experiment,   for each configuration, we run the algorithm for 800 iterations
and use $D_{800}$ as a numerical proxy for $\bar D$. Figure~\ref{fig:convergence-gap} reports
$\log(D_{800}-D_t)$ for $0\le t\le250$, with
$\varepsilon\in\{1,4^{-1},4^{-2},4^{-3}\}$ and $\eta=\varepsilon$.

% ----------------------------------------------------------------------
% Figure 1(a): log dual objective gaps
% ----------------------------------------------------------------------
\begin{figure}[p]
    \centering

    \includegraphics[width=0.78\textwidth]
        {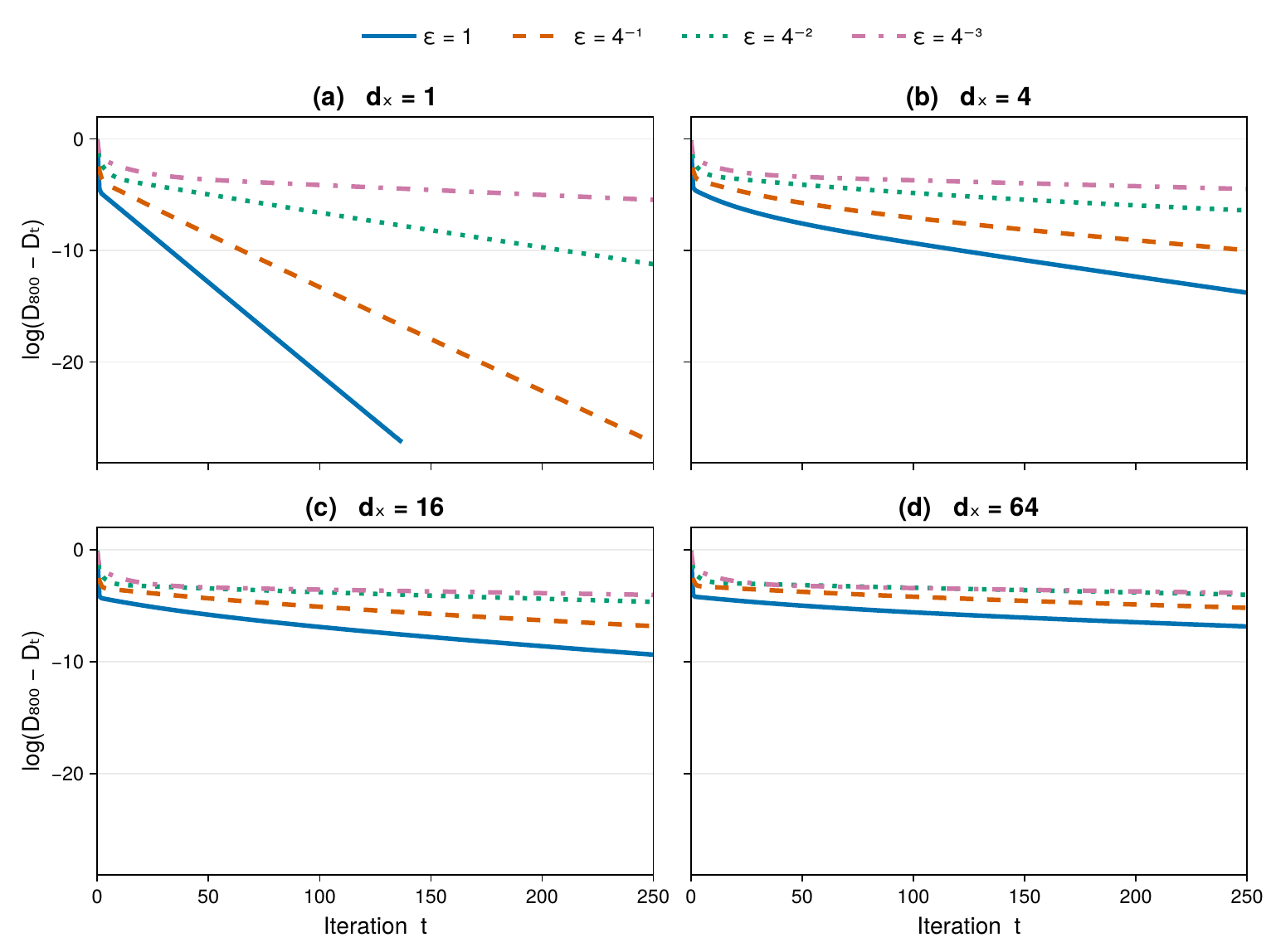}

    \includegraphics[width=0.78\textwidth]
        {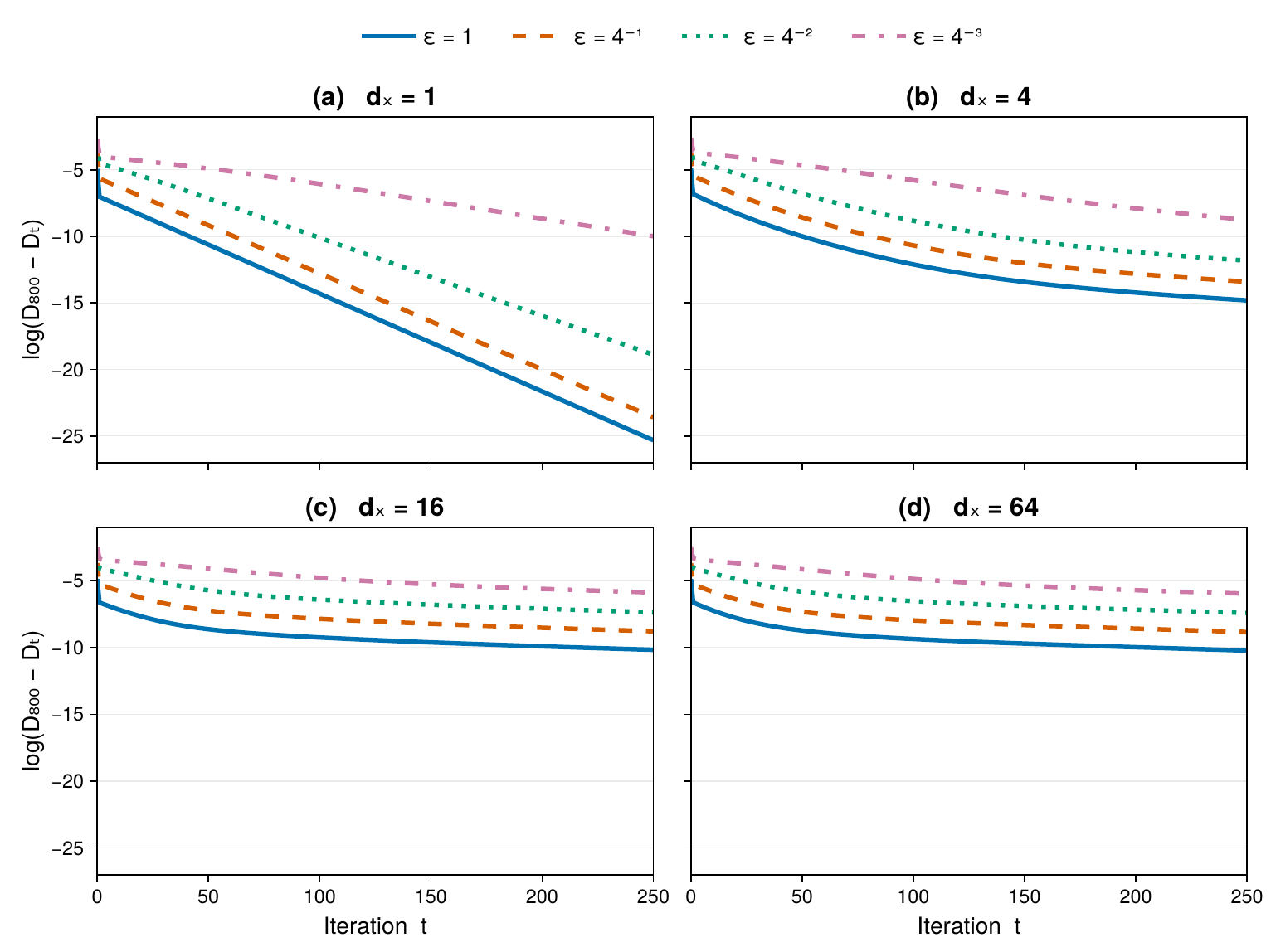}

    \caption{{\footnotesize
    \textbf{Empirical convergence of the modified Sinkhorn algorithm:
    dual objective gaps.}
    The top four panels correspond to the synthetic Gaussian experiment and
    the bottom four panels to the ANSUR II experiment. Within each group,
    the panels correspond to $d_x\in\{1,4,16,64\}$ and the curves correspond
    to $\varepsilon\in\{1,4^{-1},4^{-2},4^{-3}\}$.
    For each configuration, the algorithm is run for $800$ iterations and
    $D_{800}$ is used as a numerical proxy for the optimal dual value.
    The plots report $\log(D_{800}-D_t)$ for $0\leq t\leq250$.
    The covariates are normalized so that $M_x=1$, and the step size is
    $\eta=\varepsilon$.}
    }
    \label{fig:convergence-gap}
\end{figure}

The log gaps in Figure~\ref{fig:convergence-gap} are approximately affine over substantial
parts of the iteration range, which is the behavior expected under linear
convergence.  The linear pattern is especially clear when $d_x$ is small. Some curves in $d_x=1$ terminate early because the gap
$D_{800}-D_t$ reaches the machine precision.  For large $d_x$, and particularly for ANSUR II, there is a more
visible initial transient followed by an approximately linear tail.
For a fixed covariate dimension, larger values of $\varepsilon$ generally
produce steeper decreases, while the curves become flatter as
$\varepsilon$ decreases.  This is  consistent with the
exponential dependence on $1/\varepsilon$ in the theoretical contraction
bound.
The empirical convergence also slows as $d_x$ increases, which is consistent with the dependence of Theorem~4.1 on the smallest eigenvalue
$\underline{\lambda}$ of $\Sigma_X$.  Indeed, under the normalization $M_x=1$, $\underline{\lambda} \le 1/d_x$. 

% ----------------------------------------------------------------------
% Figure 1(b): empirical contraction factors
%
% Decrement the counter because the preceding caption has already
% advanced it to 1. The caption below advances it back to 1 and prints 1(b).
% ----------------------------------------------------------------------

\begin{figure}[p]
    \centering

    \includegraphics[width=0.78\textwidth]
        {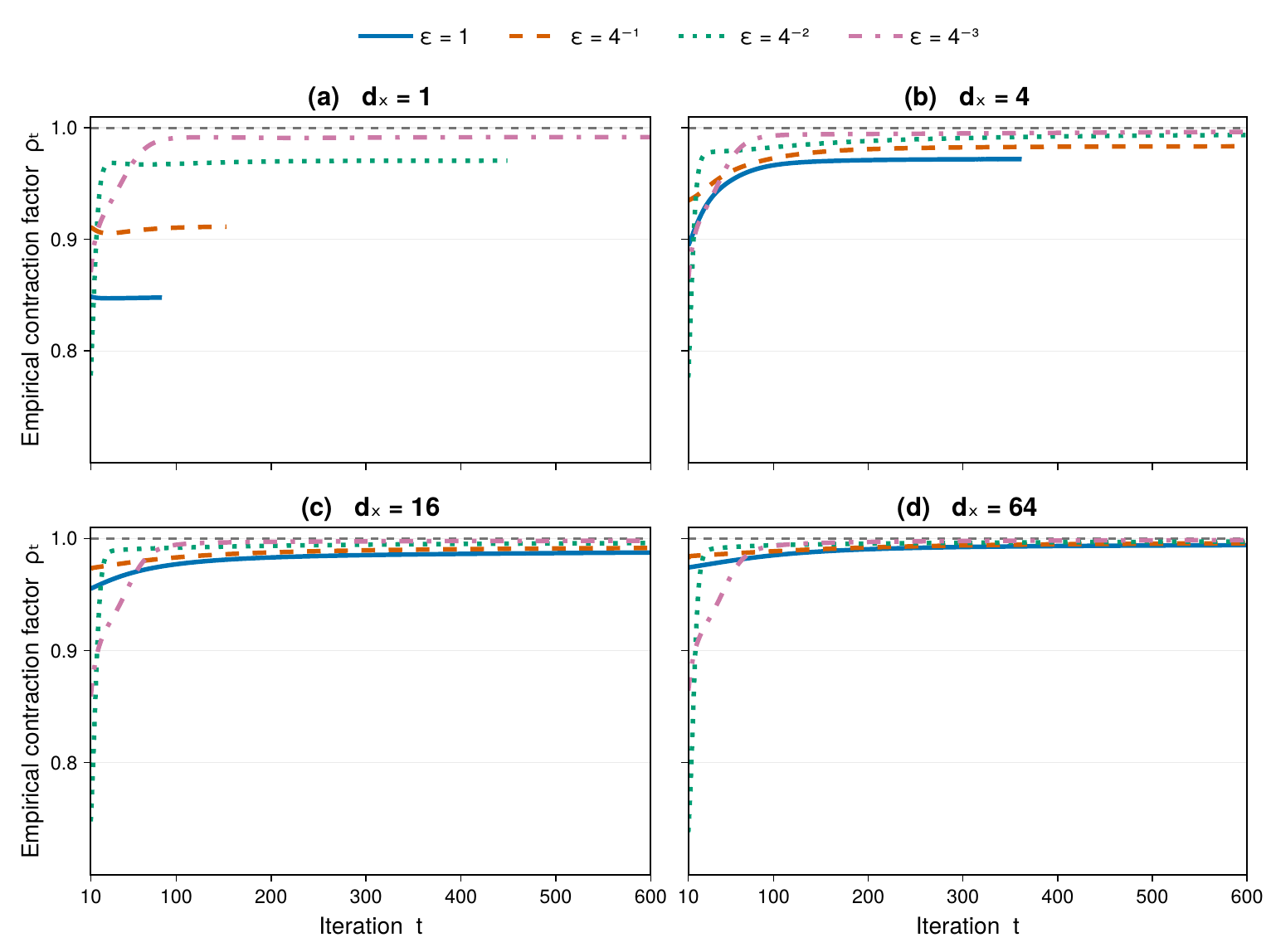}

    \includegraphics[width=0.78\textwidth]
        {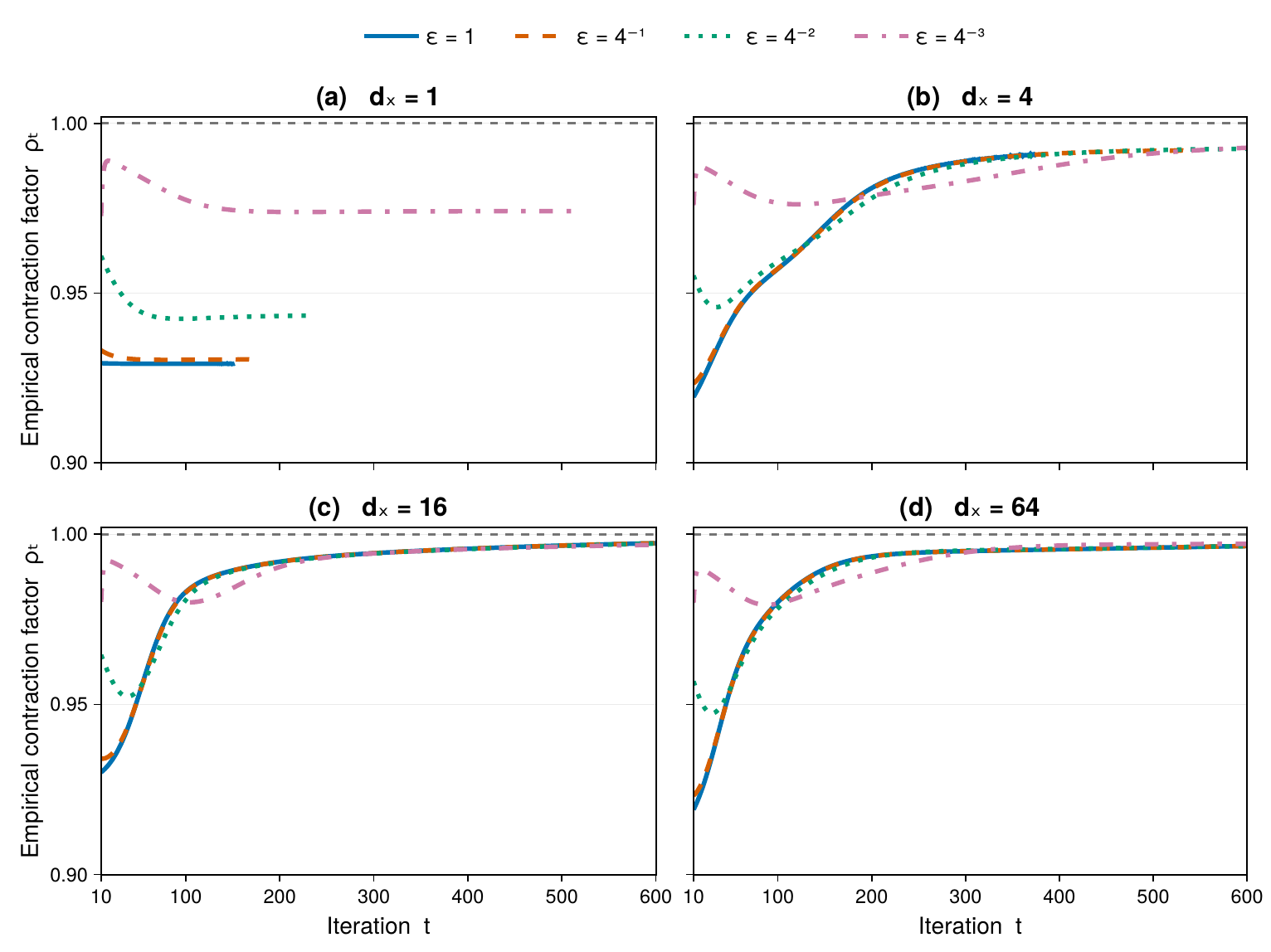}

    \caption{{\footnotesize
    \textbf{Empirical convergence of the modified Sinkhorn algorithm:
    contraction factors.}
    The top four panels correspond to the synthetic Gaussian experiment and
    the bottom four panels to the ANSUR II experiment. The empirical
    contraction factor is $\rho_t=(D_t-D_{t-1})/(D_{t-1}-D_{t-2})$ and is displayed for $10\leq t\leq600$.
    A value is reported only when both consecutive objective increments
    exceed
    $10^{-9}\max\{1,|D_t|,|D_{t-1}|,|D_{t-2}|\}$. The dataset configurations, values of
    $(d_x,\varepsilon)$, normalization, and step size are the same as in
    preceding figure.}
    }
    \label{fig:contraction-factor}
\end{figure}

Figure~\ref{fig:contraction-factor} reports the empirical contraction factors.
To avoid ratios dominated by roundoff error, we display $\rho_t$ only when
both $D_t-D_{t-1}$ and $D_{t-1}-D_{t-2}$ exceed
\[
  10^{-9}\max\{1,|D_t|,|D_{t-1}|,|D_{t-2}|\}.
\]
As such, some curves stop before iteration
600 because the consecutive objective increments fall below the
threshold.
For $d_x=1$, the
ratios converge rather rapidly, and the limiting ratio becomes
closer to one as $\varepsilon$ decreases.  For larger $d_x$, the ratios exhibit
a longer transient and stabilize much closer to one, corresponding to the
flatter log gaps in Figure~\ref{fig:convergence-gap}.  The same qualitative behavior appears
for both the synthetic and ANSUR II data, although the real-data experiment
shows more variation.  
Overall, Figures~\ref{fig:convergence-gap} and \ref{fig:contraction-factor}  provide empirical support for linear
convergence and for the predicted deterioration of the rate as
$\varepsilon$ decreases and $d_x$ increases.

\subsection{Stable step sizes}

Theorem \ref{thm: linear conv modified} uses the sufficient condition $0<\frac{\eta M_x^2}{\varepsilon}<2$
to guarantee monotonicity of the dual objective along the iterates.  Appendix \ref{sec: counterexample} shows that the
constant $2$ is sharp in general. That said, the example in Appendix \ref{sec: counterexample} pertains to a worst-case construction in $d_x=1$, so the admissible range of $\eta$, for which the dual objective is monotonically nondecreasing along the iterates, can be wider for given empirical data. 

For each pair $(d_x,\varepsilon)$, we call a step size \textit{stable} if
\[
  D_t\ge D_{t-1}
  -10^{-10}\max\{1,|D_t|,|D_{t-1}|\}
  \qquad\text{for every }t\le100.
\]
Starting from a stable step size ($\eta=\varepsilon / M_x^2$), we repeatedly double $\eta$ until this condition
fails and then perform 10 bisection steps to approximate the largest stable
value.  Figure~\ref{fig:stepsize-stability} reports the resulting value of
$\eta M_x^2/\varepsilon$.

\begin{figure}[t]
    \centering
    \includegraphics[width=0.90\textwidth]
        {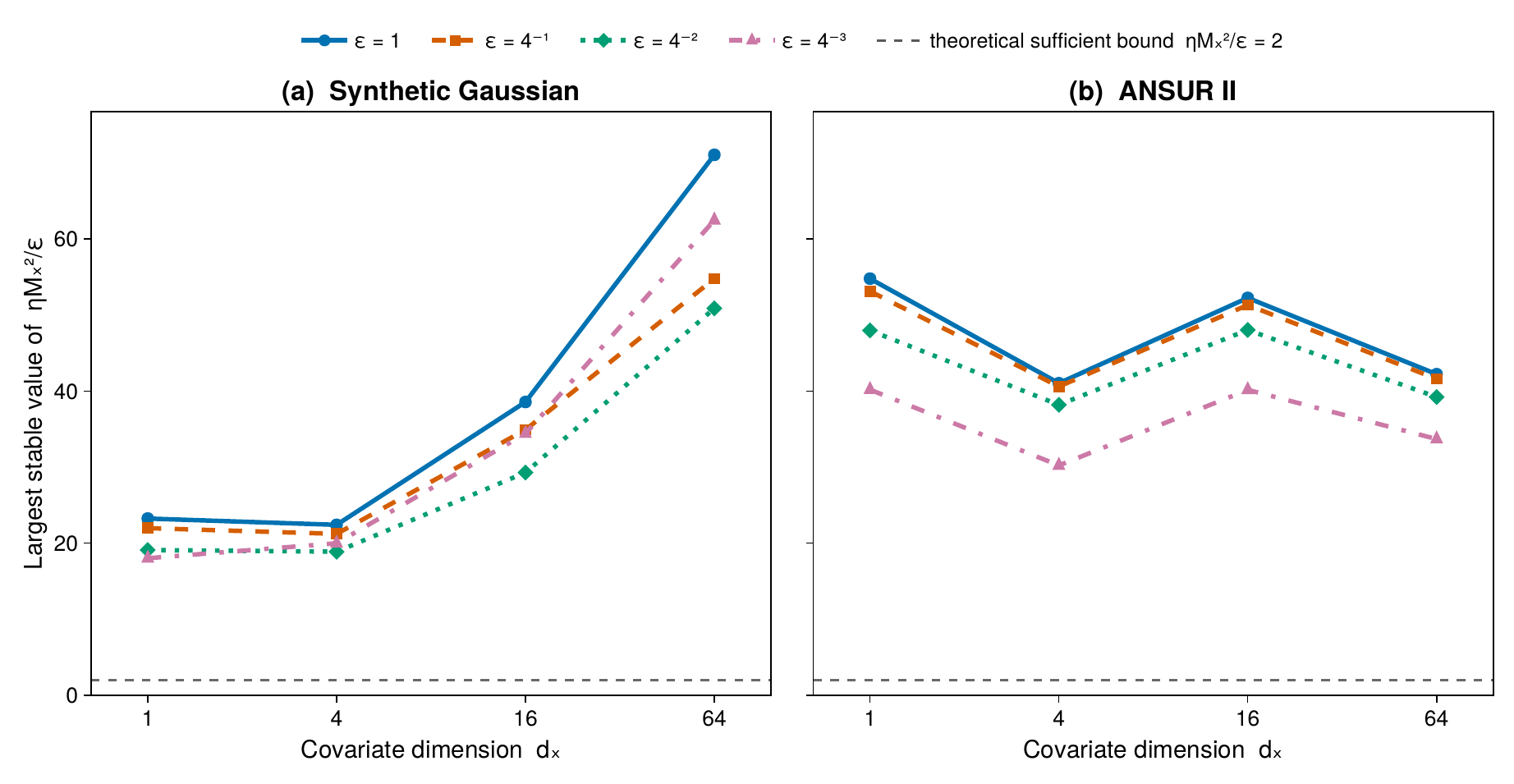}
    \caption{{\footnotesize 
    \textbf{Empirical step-size stability of the modified Sinkhorn algorithm.}
    The figure reports the largest value of
    $\eta M_x^2/\varepsilon$ classified as stable for
    $d_x\in\{1,4,16,64\}$ and
    $\varepsilon\in\{1,4^{-1},4^{-2},4^{-3}\}$.
    Panel (a) corresponds to the synthetic Gaussian experiment and panel (b)
    to the ANSUR II experiment, with the same data construction and
    preprocessing as in Figures~\ref{fig:convergence-gap} and \ref{fig:contraction-factor}.}
    %A step size is classified as stable if the dual objective is
    %numerically nondecreasing throughout the first $100$ iterations from the
    %zero initialization. For each $(d_x,\varepsilon)$, we first double $\eta$ until monotonicity fails, and then use 10 bisection steps to approximate the largest value for which monotonicity is preserved. 
    %The horizontal dashed line marks
    %$\eta M_x^2/\varepsilon=2$, the sufficient monotonicity condition in
    %Lemma~\ref{lem: monotonicity modified}.}
    }
    \label{fig:stepsize-stability}
\end{figure}

For each configuration, the empirical stability boundary is far
above $2$.  Across the two experiments, the estimated
boundaries range from roughly 18 to 72, with substantial variation across
$d_x$, $\varepsilon$, and the data set.  The
result shows that the theoretical condition on the step size can be quite conservative for a
given empirical problem.  However, the estimated boundaries concern numerical monotonicity only over the first 100 iterations from the zero initialization, and the largest monotone step size need not coincide with the most computationally efficient choice. Data-adaptive step-size rules and a systematic study of their computational performance are left for future work.

\FloatBarrier

\subsection{Small $\varepsilon$ regime}

To examine a regime in which the theoretical contraction factor is expected to
be especially close to one, we repeat the contraction-factor experiment with
$\varepsilon=10^{-3}$ and $\eta=\varepsilon$, running the algorithm for 3000
iterations.  Figure~\ref{fig:eps1e3} reports the resulting $\rho_t$ values.

\begin{figure}[t]
    \centering
    \includegraphics[width=0.90\textwidth]
        {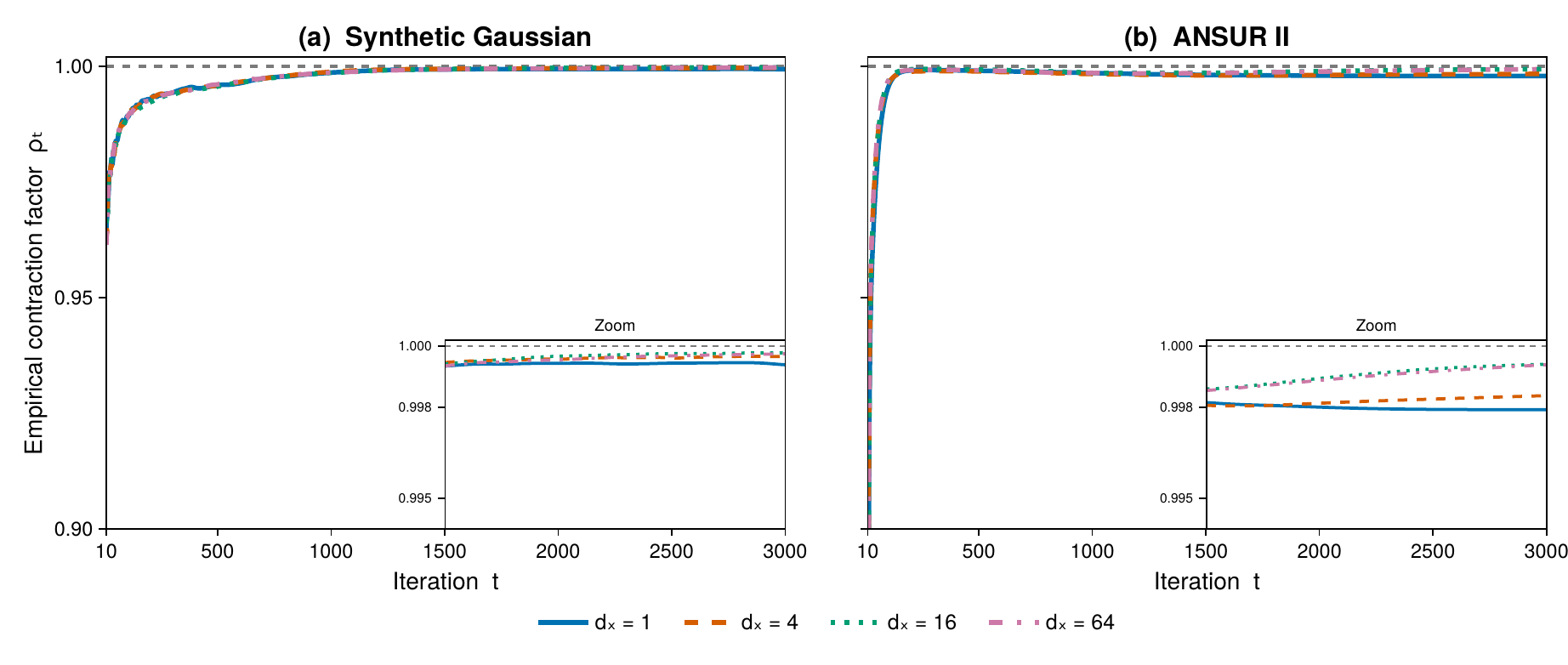}
    \caption{{\footnotesize
    \textbf{Empirical contraction factors for $\varepsilon=10^{-3}$.}
    The setting is the same as in Figure~\ref{fig:contraction-factor}, except that $\varepsilon=10^{-3}$ and the contraction factors are reported for $10\leq t\leq3000$.
    Panel (a) corresponds to the synthetic Gaussian experiment and panel (b) to the ANSUR II experiment. 
    }
    }
    \label{fig:eps1e3}
\end{figure}

For both the synthetic and ANSUR II data, the empirical contraction factors
cluster just below one after a short
transient.  Differences across the four covariate dimensions are small relative
to the dominant effect of the very small regularization level.  As such, the
algorithm remains stable with the theoretically suggested step size, but many
more iterations are needed to obtain a substantial reduction in the objective
gap.  This behavior is consistent with the deterioration of the contraction
bound as $\varepsilon \downarrow 0$.

\section{Proofs for Sections \ref{sec: duality} and \ref{sec: sinkhorn}}
\label{sec: proof sec 2 and 3}

In what follows, we will repeatedly use the following elementary inequalities
\begin{align}
&e^b-e^a \ge e^a (b-a), \ \forall a,b \in \R, \label{eq: exponential} \\
&e^b - e^a - e^a(b-a) \ge \frac{e^{-K}}{2} (b-a)^2, \ \forall a,b \ge -K,
\label{eq: exponential 2} \\
&|e^b-e^a| \le e^{K}|b-a|, \ \forall a,b \in [-K,K],
\label{eq: exponential 3}
\end{align}
where $K > 0$. 

\subsection{Proof of Proposition~\ref{prop: potential bound}}

The unique optimal coupling $\bar{\pi}$ has a density of the form
\begin{equation*}
\frac{d\bar{\pi}}{d(\mu \otimes \nu)} (u,x,y) = e^{p_u(x,y)} \quad \text{with} \quad p_u (x,y) := \frac{1}{\varepsilon} \big ( \barf(u) + \langle \barg(u),x \rangle + \barh(x,y) - c(u,y) \big ).
\end{equation*}
For each $u \in \calU$, a version of $\bar{\pi}_u$ is given by
\begin{equation*}
\frac{d\bar{\pi}_u}{d\nu}(x,y) = e^{p_u(x,y)}. 
\end{equation*}

We first find an upper bound on $\| \barf \|_\infty$.  By Jensen's inequality,
\begin{equation*}
\begin{split}
\barf(u) &= -\varepsilon \log \int e^{(\langle \barg(u),x\rangle+\barh(x,y)-c(u,y))/\varepsilon} \, d\nu (x,y)  \\
&\le - \int \big(\langle \barg(u),x\rangle + \barh(x,y) - c(u,y)\big) \, d\nu (x,y)\\
&= -\int \barh\,d\nu + \int c(u,\cdot) \,d\nu \le \|c\|_\infty, \ \forall u \in \calU, \\
\barh (x,y) &=-\varepsilon \log \int e^{(\barf(u)+\langle \barg(u),x\rangle-c(u,y))/\varepsilon} \, d\mu (u) \\
&\le - \int \big (\barf(u) + \langle \barg(u),x \rangle - c(u,y) \big) \,d\mu (u)\\
&= \int c(\cdot,y) \, d\mu \le \|c\|_\infty, \ \forall (x,y) \in \calX \times \calY,
\end{split}
\end{equation*}
where we used $\int \barh \, d\nu = \int \barf \, d\mu + \int \barh \, d\nu = \mathsf{D}(\mu,\nu) = \mathsf{T} (\mu,\nu) \ge 0$ by strong duality and the normalization that $\int \barf\, d\mu =0$ and $\int \barg \, d\mu=0$.
Observe that, for any $u \in \calU$, 
\begin{equation*}
\begin{split}
    0 &\le \KL{\bar{\pi}_u}{\nu} =\int \log\left(\frac{d\bar{\pi}_u}{d\nu}\right) d\bar{\pi}_u  \\
    &= \frac{1}{\varepsilon} \int \big( \barf(u) + \langle \barg(u), x \rangle + \barh(x,y) - c(u,y) \big) \,d\bar{\pi}_u (x,y)
\end{split}
\end{equation*}
so that, as $\int x \, d\bar{\pi}_u(x,y) = 0$, we have
\begin{equation*}
    \barf(u) \ge \int c \,d\bar{\pi}_u- \int \barh\, d\bar{\pi}_u \ge - \int \barh\, d\bar{\pi}_u\ge-\|c\|_\infty.
\end{equation*}
Conclude that $\|\barf\|_{\infty} \le \|c\|_\infty$.

Next, we establish an upper bound on $\| \barg \|_\infty$. 
For any $v \in \mathbb{S}^{d_x-1} := \{ x \in \R^{d_x} : \| x \| = 1 \}$, define a probability measure $q^{(v)}$ on $\calX \times \calY$ by
\begin{equation*}
dq^{(v)}(x,y) = e^{\hat{p}_v(x,y)}\, d\nu(x,y) \quad \text{with} \quad \hat{p}_v(x,y) := \log \big(1 + \delta \langle \Sigma_X^{-1}v, x \rangle \big),
\end{equation*}
where we choose $\delta = 1/(2\|\Sigma_X^{-1}\|_{\op}M_x)$, so that $1 + \delta \langle \Sigma_X^{-1}v, x \rangle \in [1/2, 3/2]$ for all $(x,v)  \in \calX \times \mathbb{S}^{d_x-1}$. 
As $\E[X] = 0$, $q^{(v)}$ is a probability measure.
Using inequality (\ref{eq: exponential}), we obtain
\begin{equation*}
0 = \int \big( e^{p_u} - e^{\hat{p}_v} \big) \,d\nu \ge \int (p_u - \hat{p}_v)  e^{\hat{p}_v} \, d\nu = \int_{\calX \times \calY} ( p_u - \hat{p}_v ) \,dq^{(v)}.
\end{equation*}
Substitute the definition of $p_u$ into the inequality to get
\begin{equation*}
0 \ge \int_{\mathcal{X} \times \mathcal{Y}} \left( \frac{1}{\varepsilon} \big( \barf(u) + \langle \barg(u), x \rangle + \barh(x,y) - c(u,y) \big) - \hat{p}_v(x,y) \right) dq^{(v)}(x,y).
\end{equation*}
Since $\int_{\calX \times \calY} x \,dq^{(v)} (x,y)= \delta v$, we have
\begin{equation*}
\begin{split}
\delta \langle \barg(u), v \rangle &\le \int_{\calX \times \calY} \big( c(u,y) - \barf(u) - \barh(x,y) + \varepsilon \hat{p}_v(x,y) \big) \,dq^{(v)}(x,y) \\
&\le \|c\|_\infty +\|\barf\|_\infty - \int \barh \, dq^{(v)}+ \varepsilon \log \left(\frac{3}{2} \right).
\end{split}
\end{equation*}

We shall find an upper bound on $-\int \barh \,dq^{(v)}$. 
Recall $\barh \le \| c \|_\infty$. 
Using the fact that $\int \barh \, d\nu \ge 0$, we have
\begin{equation*}
\begin{split}
-\int \barh \,dq^{(v)} &= -\int \barh \, d\nu - \int \barh \delta \langle \Sigma_X^{-1}v, x \rangle\, d\nu \\ 
&\le \int (\|c\|_{\infty} - \barh) \delta \langle \Sigma_X^{-1}v, x \rangle\, d\nu - \int \|c\|_\infty \delta \langle \Sigma_X^{-1}v, x \rangle\,d\nu \\
&= \int (\|c\|_{\infty} - \barh) \delta \langle \Sigma_X^{-1}v, x \rangle\, d\nu \\
&\le \frac{1}{2} \int (\|c\|_{\infty} - \barh) \, d\nu \le \frac{1}{2}\|c\|_\infty.
\end{split}
\end{equation*}
Now, choosing $v = \barg(u)/ \|\barg(u)\|$ yields
\begin{equation*}
\begin{split}
\|\barg\|_{\infty} &\le \frac{1}{\delta} \left( \frac{3}{2}\|c\|_\infty + \| \barf \|_\infty + \varepsilon \log\left(\frac{3}{2}\right) \right) \\
&\le 2\|\Sigma_X^{-1}\|_{\op} M_x \left( \frac{5}{2}\|c\|_\infty + \varepsilon \log\left(\frac{3}{2}\right) \right).
\end{split}
\end{equation*}

Finally, a lower bound on $\barh$ can be obtained as
\begin{equation*}
\barh(x,y)
\ge -\varepsilon \log \int e^{(\|\barf\|_\infty + \|\barg\|_\infty M_x)/\varepsilon} \, d\mu =-\|\barf\|_\infty - \|\barg\|_\infty M_x. 
\end{equation*}
Combining the upper bound on $h$, we obtain the conclusion of the proposition. \qed

\subsection{Proof of Proposition~\ref{prop: PL}}
Let $F(u,x,y) := f(u) + \langle g(u),x \rangle + h(x,y)-c(u,y)$ and $\bar{F}(u,x,y) :=\barf(u) + \langle \barg(u),x \rangle + \barh(x,y)-c(u,y)$. 
Observe that, as $\int f \, d\mu=0, \int g \, d\mu = 0$, and $\int x \, d\nu =0$, 
\begin{equation*}
\begin{split}
\int (\bar{F}-F)^2 \, d(\mu \otimes \nu) &= \big \| \barf - f \big \|_{L^2(\mu)}^2 + \big \| \langle \barg - g,x \rangle \big \|_{L^2(\mu \otimes \nu)}^2 + \big \| \barh  - h \big \|_{L^2(\nu)}^2 \\
&\ge \big \| \barf - f \big \|_{L^2(\mu)}^2 + \underline{\lambda} \big \|  \barg - g \big \|_{L^2(\mu)}^2 + \big \| \barh  - h \big \|_{L^2(\nu)}^2.
\end{split}
\end{equation*}
Using inequality (\ref{eq: exponential 2}) with $a = F/\varepsilon$ and $b = \bar{F} /\varepsilon$, we have 
\begin{equation*}
\begin{split}
&D (f,g,h) - D(\barf,\barg,\barh) \\
&= \int (F - \bar{F}) \, d(\mu \otimes \nu) + \varepsilon \int \big(e^{\bar{F}/\varepsilon} - e^{F/\varepsilon} \big)\,  d(\mu \otimes \nu) \\
&\ge \int (h - \barh) \, d\nu + \int (\bar{F}-F)e^{F/\varepsilon} \, d(\mu \otimes \nu)  + \frac{e^{-K/\varepsilon}}{2\varepsilon} \int (\bar{F}-F)^2 \, d(\mu \otimes \nu) \\
&\ge  \left \langle \barf-f, e^{f/\varepsilon} \int e^{(\langle g,x\rangle+h-c)/\varepsilon} \, d\nu \right \rangle_{L^2(\mu)}  \\
&\quad + \left \langle \barg-g, e^{f/\varepsilon} \int x e^{(\langle g,x\rangle+h-c)/\varepsilon} \, d\nu \right \rangle_{L^2(\mu)} \\
&\quad + \left \langle \barh-h, -1+e^{h/\varepsilon} \int  e^{(f+\langle g,x\rangle-c)/\varepsilon} \, d\mu \right \rangle_{L^2(\nu)} \\
&\quad + \frac{(1 \wedge \underline{\lambda})e^{-K/\varepsilon}}{2\varepsilon} \big (\big \| \barf-f \big \|_{L^2(\mu)}^2 + \big \| \barg-g \big \|_{L^2(\mu)}^2 + \big \| \barh-h \big \|_{L^2(\nu)}^2 \big) \\
&=\left \langle \barf-f, \ell_f(f,g,h) \right  \rangle_{L^2(\mu)} + \left \langle \barg-g, \ell_g(f,g,h) \right \rangle_{L^2(\mu)}  + \left \langle \barh-h, \ell_h(f,g,h) \right \rangle_{L^2(\nu)} \\
&\quad + \frac{(1 \wedge \underline{\lambda})e^{-K/\varepsilon}}{2\varepsilon} \big (\big \| \barf-f \big \|_{L^2(\mu)}^2 + \big \| \barg-g \big \|_{L^2(\mu)}^2 + \big \| \barh-h \big \|_{L^2(\nu)}^2 \big),
\end{split}
\end{equation*}
where we used the fact that, as $\int \barf \, d\mu = \int f \, d\mu = 0$, $\langle \barf-f, 1 \rangle_{L^2(\mu)} = 0$. 
The desired inequality follows from the elementary inequality $2ab \ge - \kappa a^2-\kappa^{-1}b^2$ for $a,b \in \R$ and $\kappa > 0$. To see this, we set $\kappa =\frac{(1 \wedge \underline{\lambda})e^{-K/\varepsilon}}{\varepsilon}$ and observe that
\begin{equation*}
\begin{split}
&\left \langle \barf-f, \ell_f(f,g,h) \right  \rangle_{L^2(\mu)} + \left \langle \barg-g, \ell_g(f,g,h) \right \rangle_{L^2(\mu)}  + \left \langle \barh-h, \ell_h(f,g,h) \right \rangle_{L^2(\nu)} \\
&\ge - \frac{\kappa}{2} \big(\big \| \barf-f \big \|_{L^2(\mu)}^2 + \big \| \barg-g \big \|_{L^2(\mu)}^2 + \big \| \barh-h \big \|_{L^2(\nu)}^2\big) \\
&\quad - \frac{1}{2\kappa} \big (\big \| \ell_f(f,g,h) \|_{L^2(\mu)}^2 + \| \ell_g(f,g,h) \|_{L^2(\mu)}^2 + \| \ell_h(f,g,h) \|_{L^2(\nu)}^2 \big).
\end{split}
\end{equation*}
This completes the proof.
\qed

\subsection{Proof of Proposition~\ref{prop: sinkhorn potential}}

Proposition~\ref{prop: sinkhorn potential} follows by combining Lemmas \ref{lem: bound on f and h} and \ref{lem: bound on g} below. Before that, we first prove the following lemma concerning monotonicity of the dual objective along the Sinkhorn iterates. 
\begin{lemma}[Monotonicity lemma]
\label{lem: monotonicity}
For $t \in \N_0$, we have 
\begin{equation*}
\begin{split}
D (f^{t},g^{t},h^{t}) &\le D  (f^{t},\tildeg^{t+1},h^{t}) \le D  (\tildef^{t+1},\tildeg^{t+1},h^{t}) \\
&\le D(\tildef^{t+1},\tildeg^{t+1},\tildeh^{t+1}) = D(f^{t+1},g^{t+1},h^{t+1}).
\end{split}
\end{equation*}
\end{lemma}

\begin{proof}
The lemma directly follows from the definition of the Sinkhorn iterates; cf. equation (\ref{eq: interpretation}). For the sake of completeness, we provide an explicit proof.
Observe that
    \begin{equation*}
    D(f^{t}, \tildeg^{t+1}, h^{t}) - D(f^{t}, g^{t}, h^{t}) =\varepsilon \left ( \iota(f^{t}, g^{t}, h^{t}) - \iota(f^{t}, \tildeg^{t+1}, h^{t}) \right).
    \end{equation*}
    An application of inequality (\ref{eq: exponential}) yields
    \begin{equation*}
    e^{\langle g^{t}, x \rangle / \varepsilon} - e^{\langle \tildeg^{t+1}, x \rangle / \varepsilon} \ge e^{\langle \tildeg^{t+1}, x \rangle/ \varepsilon} \frac{\langle g^{t} - \tildeg^{t+1}, x \rangle}{\varepsilon}.
    \end{equation*}
    Multiply both sides by  $e^{(f^{t}+h^{t}-c)/\varepsilon}$ and integrate over $\mu \otimes \nu$ to obtain
    \begin{equation*}
    \begin{split}
   &\varepsilon \left ( \iota(f^{t}, g^{t}, h^{t}) - \iota(f^{t}, \tildeg^{t+1}, h^{t}) \right) \\
   &\ge \int_{\calU} \left\langle g^{t} - \tildeg^{t+1}, e^{f^{t}/\varepsilon} \underbrace{\int x \exp\left(\frac{ \langle \tildeg^{t+1}, x \rangle + h^{t} -c}{\varepsilon}\right) \,d\nu}_{=0} \right\rangle \,d\mu.
   \end{split}
    \end{equation*}
    This establishes $D(f^{t}, \tildeg^{t+1}, h^{t}) \ge D(f^{t}, g^{t}, h^{t})$. 

    Next, since $\int e^{(\tildef^{t+1} (u)+ \langle \tildeg^{t+1}(u),x\rangle + h^{t} - c)/\varepsilon} \, d\nu = 1$ for all $u \in \calU$, we have 
    \begin{equation*}
    \iota(\tildef^{t+1}, \tildeg^{t+1}, h^{t}) =1.
    \end{equation*}
    On the other hand, 
    \begin{equation*}
    \iota(f^{t}, \tildeg^{t+1}, h^{t}) = \int e^{(f^{t} -\tildef^{t+1})/\varepsilon} \, d\mu,
    \end{equation*}
    so that we have
    \begin{multline*}
        D(\tildef^{t+1}, \tildeg^{t+1}, h^{t}) - D(f^{t}, \tildeg^{t+1}, h^{t}) = \varepsilon \iota(f^{t}, \tildeg^{t+1}, h^{t}) - \varepsilon + \int (\tildef^{t+1} - f^{t}) \,d\mu \\
          = \varepsilon \int \left( e^{(f^{t}-\tildef^{t+1})/\varepsilon} - 1-(f^{t}-\tildef^{t+1})/\varepsilon \right)\, d\mu
        \ge 0.
\end{multline*}
    Likewise, we have $\iota (\tildef^{t+1}, \tildeg^{t+1}, \tildeh^{t+1}) = 1$ and
    \begin{equation*}
    \begin{split}
    &D(\tildef^{t+1}, \tildeg^{t+1}, \tildeh^{t+1}) - D(\tildef^{t+1}, \tildeg^{t+1}, h^{t}) \\
    &= \varepsilon \int \left( e^{(h^{t}-\tildeh^{t+1})/\varepsilon} - 1 - (h^{t}-\tildeh^{t+1})/\varepsilon  \right)\, d\nu \ge 0,
    \end{split}
    \end{equation*}
    completing the proof. 
\end{proof}

Recall the notations $M_x, L_c$, and $D_0$ defined in (\ref{eq: initial parameter}) and (\ref{eq: parameter}). Observe that 
\begin{equation*}
|c(u,y) - c(u',y)| \le L_c \|u-u'\|, \ \forall u,u' \in \calU, y \in \calY. 
\end{equation*}
We establish bounds on $\| \tildef^{t+1} \|_{\infty} \vee \| f^{t+1} \|_\infty$ and $\| \tildeh^{t+1} \|_{\infty} \vee \| h^{t+1} \|_\infty$. 
Observe that weak duality for entropic VQR implies that 
\begin{equation*}
D(f,g,h) \le \int c \, d(\mu \otimes \nu) 
\end{equation*}
for any $(f,g,h) \in L^1(\mu) \times L^1(\mu;\R^{d_x}) \times L^1(\nu)$.

\begin{lemma}
\label{lem: bound on f and h}
For $t \in \N_0$,
\begin{equation*}
\begin{split}
\| \tildef^{t+1} \|_{\infty} \vee \| f^{t+1} \|_\infty &\le L_c \diam (\calU) + \| c \|_\infty + D_0^- + \|h^0\|_\infty= K_f, \quad \text{and} \\
\| \tildeh^{t+1} \|_\infty  \vee \| h^{t+1} \|_\infty &\le \| c \|_\infty + K_f + \big( \big \| \tildeg^{t+1} \big \|_\infty \vee \big \| g^{t+1} \big \|_\infty\big) M_x.
\end{split}
\end{equation*}
\end{lemma}
\begin{proof}
Pick any $u, u' \in \calU$. 
Let 
\begin{equation*}
A(u') := \int \exp\left(\frac{ \langle \tildeg^{t+1}(u'), x \rangle + h^{t}(x,y) -c(u',y)}{\varepsilon}\right) \,d\nu(x,y) = e^{-\tildef^{t+1}(u')/\varepsilon}.
\end{equation*}
Define a probability measure $\rho_{u'}$ on $\calX \times \calY$ by
\begin{equation*}
d\rho_{u'}(x,y) := \frac{1}{A(u')} \exp\left(\frac{\langle \tildeg^{t+1}(u'), x \rangle +h^{t}(x,y) -c(u',y)}{\varepsilon}\right) \,d\nu(x,y).
\end{equation*}
We observe that
\begin{equation*}
\begin{split}
    &-\frac{\tildef^{t+1}(u)}{\varepsilon} + \frac{\tildef^{t+1}(u')}{\varepsilon} \\
    &= \log \int \exp\left(\frac{\langle \tildeg^{t+1}(u), x \rangle + h^{t}(x,y)-c(u,y)}{\varepsilon}\right) d\nu(x,y) - \log A(u') \\
    &= \log \int \exp\left( \frac{c(u',y) -c(u,y) + \langle \tildeg^{t+1}(u) - \tildeg^{t+1}(u'), x \rangle}{\varepsilon} \right) \,d\rho_{u'}(x,y) \\
    &\ge \frac{1}{\varepsilon}\int \left( c(u',y) - c(u,y) + \langle \tildeg^{t+1}(u) - \tildeg^{t+1}(u'), x \rangle \right) \,d\rho_{u'}(x,y)
\end{split}
\end{equation*}
by Jensen's inequality.
By the definition of $\tildeg^{t+1}$,
$
\int_{\calX \times \calY} x \, d\rho_{u'}(x,y) = 0,
$
which yields that
\begin{equation*}
\begin{split}
    \tildef^{t+1}(u') - \tildef^{t+1}(u) &\ge \int (c(u',y) - c(u,y)) \,d\rho_{u'}(x,y) \\
    &\ge -L_c\|u-u'\|.
\end{split}
\end{equation*}
Interchanging the roles of $u$ and $u'$, we obtain
\begin{equation*}
|\tildef^{t+1}(u) - \tildef^{t+1}(u')| \le L_c \|u - u'\|,
\end{equation*}
which gives $\| f^{t+1} \|_\infty \le L_c \diam (\calU)$. We need to find bounds on $\int \tildef^{t+1} \, d\mu$. Using Lemma~\ref{lem: monotonicity} and weak duality, we have
\begin{equation}
-D_0^- \le D_0 \le D (\tildef^{t+1}, \tildeg^{t+1},h^{t}) = \int \tildef^{t+1} \, d\mu + \int h^{t} \, d\nu \le \int c \, d(\mu \otimes \nu) \le \|c\|_\infty. \label{eq: integral bound}
\end{equation}
For $t \ge 1$, using Lemma~\ref{lem: monotonicity} and weak duality again, we obtain
\begin{equation}
-D_0^- \le D_0 = D(f^0,g^0,h^0) \le D(f^t,g^t,h^t) = \int h^{t} \, d\nu \le \int c \, d(\mu \otimes \nu) \le \| c \|_\infty. \label{eq: lower bound h} \\
\end{equation}
For $t=0$, we use the bound $|\int h^0 \, d\nu| \le \| h^0\|_\infty$. In either case, we obtain $|\int \tildef^{t+1} \, d\mu| \le \| c \|_\infty + D_0^{-} + \|h^0\|_\infty$, which gives the desired bound on $\| \tildef^{t+1}\|_\infty$.  

The upper bound on $\|\tildeh^{t+1}\|_\infty$ follows by noting that
\begin{equation*}
\tildeh^{t+1} (x,y) = -\varepsilon \log \int \exp \left ( \frac{\tildef^{t+1}(u) + \langle \tildeg^{t+1}(u),x\rangle -c(u,y)}{\varepsilon} \right ) \, d\mu(u).
\end{equation*}
An analogous identity holds for $h^{t+1}$. 
This completes the proof.
\end{proof}

It remains to establish an upper bound on $\big \| \tildeg^{t+1} \big \|_\infty \vee \big \| g^{t+1} \big \|_\infty$. 
\begin{lemma}
\label{lem: bound on g}
For $t \in \N_0$,
\begin{equation*}
\big \| \tildeg^{t+1} \big \|_\infty \vee \big \| g^{t+1} \big \|_\infty \le 4\|\Sigma_X^{-1}\|_{\op} M_x  \left ( 4 \| c \|_\infty - \frac{3}{2} D_0 + K_f + \| h^0 \|_\infty +\varepsilon \log \left ( \frac{3}{2} \right )\right ) = K_g.
\end{equation*}
\end{lemma}
\begin{remark}
Combining the preceding lemma, we have 
\begin{equation*}
\big \| \tildeh^{t+1} \big \|_\infty \vee \big \| h^{t+1} \big \|_\infty \le \| c \|_\infty + K_f + K_g M_x = K_h,  \ t \in \N_0.
\end{equation*}
\end{remark}

\begin{proof}
Pick any $u \in \calU$. 
As in the proof of Proposition~\ref{prop: potential bound}, for any $v \in \mathbb{S}^{d_x-1}$, define a probability measure $q^{(v)}$ on $\calX \times \calY$ by
\begin{equation*}
dq^{(v)}(x,y) = e^{\hat{p}_v(x,y)}\,d\nu(x,y) \quad \text{with} \quad \hat{p}_v(x,y) = \log \big(1 + \delta \langle \Sigma_X^{-1}v, x \rangle \big),
\end{equation*}
where we choose $\delta = 1/(2\|\Sigma_X^{-1}\|_{\op}M_x)$, so that $1 + \delta \langle \Sigma_X^{-1}v, x \rangle \in [1/2, 3/2]$ for all $(x,v) \in \calX \times \mathbb{S}^{d_x-1}$. 
By construction, $\int_{\calX \times \calY} x \,dq^{(v)} (x,y)= \delta v$.
Consider a function $p_u$ on $\calX \times \calY$ defined by 
\begin{equation*}
p_u(x,y) =\frac{1}{\varepsilon} \big (   \tildef^{t+1}(u) +\langle \tildeg^{t+1}(u), x \rangle + h^{t}(x,y)-c(u,y) \big).
\end{equation*}
By construction, $\int e^{p_u} \, d\nu = 1$.

Now, using inequality (\ref{eq: exponential}), we obtain
\begin{equation*}
\begin{split}
0  &= \int \big( e^{p_u} - e^{\hat{p}_v} \big) \, d\nu \\
&\ge \int \left(\frac{ \tildef^{t+1}(u)+\langle \tildeg^{t+1}(u), x \rangle  +h^{t}(x,y)-c(u,y)}{\varepsilon} - \hat{p}_v(x,y) \right) \,dq^{(v)}(x,y).
\end{split}
\end{equation*}
Rearranging terms, we have 
\begin{equation*}
\begin{split}
\langle \tildeg^{t+1}(u), v \rangle &\le \frac{1}{\delta} \int \left(c(u,y)   -\tildef^{t+1}(u) -h^{t}(x,y) + \varepsilon \hat{p}_v(x,y) \right) \,dq^{(v)}(x,y) \\
&\le 2\|\Sigma_X^{-1}\|_{\op} M_x   \left ( \|c\|_\infty + K_f - \int h^{t} \, dq^{(v)}  + \varepsilon \log \left ( \frac{3}{2} \right ) \right).
\end{split}
\end{equation*}

We shall find a lower bound for $\int h^{t} \, dq^{(v)}$.
For $t \ge 1$, Jensen's inequality yields
\begin{equation}
h^t(x,y) \le -\int \big(f^t+\langle g^t,x\rangle-c(\cdot,y) \big) \, d\mu = \int c(\cdot,y) \, d\mu \le \|c\|_\infty, \ \forall (x,y) \in \calX \times \calY. \label{eq: upper bound h}
\end{equation}
Combining inequalities (\ref{eq: lower bound h}) and (\ref{eq: upper bound h}), we have
\begin{equation*}
\big \| h^{t} \big \|_{L^1(\nu)} = 2\int (h^{t})^+ \, d\nu - \int h^{t} \, d\nu \le 2 \| c \|_\infty - D_0.
\end{equation*}
This implies that 
\begin{equation*}
\int h^{t} \, dq^{(v)} \ge - \big \| h^{t} \big \|_{L^1(q^{(v)})} \ge - \frac{3}{2} \big \| h^{t} \big \|_{L^1(\nu)} \ge -\frac{3}{2} \big ( 2\| c \|_\infty -D_0 \big).
\end{equation*}
For $t=0$, we have $\int h^{0} \, dq^{(v)} \ge -\|h^0\|_\infty$. 
Putting everything together, we conclude
\begin{equation*}
\langle \tildeg^{t+1}(u),v \rangle \le 2\|\Sigma_X^{-1}\|_{\op} M_x  \left ( 4\| c \|_\infty - \frac{3}{2} D_0 + K_f + \| h^0 \|_\infty+ \varepsilon \log \left ( \frac{3}{2} \right )\right ).
\end{equation*}
This yields the desired bound for $\big \| \tildeg^{t+1} \big \|_\infty \vee \big \| g^{t+1} \big \|_\infty $.
\end{proof}

\subsection{Proof of Theorem \ref{thm: linear conv}}
Given the quantitative bounds on the Sinkhorn iterates, the claim of Theorem \ref{thm: linear conv} essentially follows by adapting the proof of Theorem 3.3 in \cite{carlier2022linear}. We shall organize the proof to adapt the techniques developed in \cite{attouch2013convergence,bolte2017error,lewis2025complexity} (among others)
and establish a PL inequality and slope-ascent conditions along the Sinkhorn iterates.
Recalling the notation in (\ref{eq: gradient}), we set
\[
(\ell_f^t,\ell_g^t,\ell_h^t) = \big (\ell_f(f^t,g^t,h^t), \ell_g(f^t,g^t,h^t),\ell_h(f^t,g^t,h^t) \big).
\]
Note that $\ell^t_h=0$ for $t \ge 1$ by construction. The following PL inequality follows directly from Propositions \ref{prop: PL} and \ref{prop: sinkhorn potential}. Recall the norm $\| \cdot \|_{\calH}$ defined in (\ref{eq: norm}).

\begin{lemma}[PL inequality]
\label{lem: PL}
For $t \in \N_0$, we have
\begin{equation*}
\big\|(\ell_f^{t},\ell_g^{t},\ell_h^{t})\big\|_{\calH}^2 \ge \frac{2(1 \wedge \underline{\lambda})e^{-\bar{K}/\varepsilon}}{\varepsilon} \big(D(\barf,\barg,\barh) - D (f^{t},g^{t},h^{t}) \big).
\end{equation*}
\end{lemma}

Next, we establish slope-ascent conditions for the Sinkhorn iterates.

\begin{lemma}[Slope-ascent conditions]
\label{lem: ascent}
For $t \in \N_0$,
\begin{equation*}
\begin{split}
&D(f^{t+1},g^{t+1},h^{t+1})- D(f^{t},g^{t},h^{t})\\
&\quad \ge \frac{e^{-2\bar{K}/\varepsilon} }{2\varepsilon} \big ( \big \|f^t-\tildef^{t+1}\big \|_{L^2(\mu)}^2 + \big \| \langle g^t-\tildeg^{t+1},x \rangle \big \|_{L^2(\mu \otimes \nu)}^2+\big \|h^t-\tildeh^{t+1}\big \|_{L^2(\nu)}^2 \big ), \quad \text{and} \\
&\sqrt{\big \| \tildef^{t+1}-f^t \big \|_{L^2(\mu)}^2 + \big \| \tildeh^{t+1}-h^t \big \|_{L^2(\nu)}^2} \ge \frac{ \varepsilon e^{-\bar{K}/\varepsilon}}{\sqrt{3}M_x}\big\|(\ell_f^{t+1},\ell_g^{t+1},\ell_h^{t+1})\big\|_{\calH}.
\end{split}
\end{equation*}
\end{lemma}

\begin{proof}
Decompose $D(f^{t+1},g^{t+1},h^{t+1})- D(f^{t},g^{t},h^{t})$ as
\begin{equation*}
\begin{split}
D(f^{t+1},g^{t+1},h^{t+1})- D(f^{t},g^{t},h^{t}) &= D(f^{t},\tildeg^{t+1},h^{t}) - D(f^{t},g^{t},h^{t})  \\
&\quad +D(\tildef^{t+1},\tildeg^{t+1},h^{t}) - D(f^{t},\tildeg^{t+1},h^{t}) \\
&\quad +D(\tildef^{t+1},\tildeg^{t+1},\tildeh^{t+1}) - D(\tildef^{t+1},\tildeg^{t+1},h^{t}).
\end{split}
\end{equation*}

Using inequality (\ref{eq: exponential 2}), we have 
    \begin{equation*}
    e^{\langle g^{t}, x \rangle / \varepsilon} - e^{\langle \tildeg^{t+1}, x \rangle / \varepsilon} -e^{\langle \tildeg^{t+1}, x \rangle/ \varepsilon} \frac{\langle g^{t}(u) - \tildeg^{t+1}(u), x \rangle}{\varepsilon} \ge \frac{e^{-K_gM_x/\varepsilon}}{2\varepsilon^2} \langle g^{t}(u) - \tildeg^{t+1}(u), x \rangle^2.
    \end{equation*}
    Multiply both sides by  $e^{(f^{t}+h^{t}-c)/\varepsilon}$ and integrate over $\mu \otimes \nu$ to obtain
    \begin{equation*}
    \begin{split}
    &D(f^{t},\tildeg^{t+1},h^{t}) - D(f^{t},g^{t},h^{t}) \\
    &\ge \int \left\langle g^{t} -  \tildeg^{t+1}, e^{f^{t}/\varepsilon} \underbrace{\int x \exp\left(\frac{  \langle \tildeg^{t+1}, x \rangle + h^{t}-c}{\varepsilon}\right) \,d\nu}_{=0} \right\rangle \,d\mu \\
    &\quad + \frac{e^{-\bar{K}/\varepsilon}}{2\varepsilon} \big \| \langle g^{t} - \tildeg^{t+1}, x \rangle \big \|_{L^2(\mu \otimes \nu)}^2 \\
    &\ge \frac{e^{-\bar{K}/\varepsilon}}{2\varepsilon} \big \| \langle g^{t} - \tildeg^{t+1}, x \rangle \big \|_{L^2(\mu \otimes \nu)}^2.
    \end{split}
    \end{equation*}

    Next, using inequality (\ref{eq: exponential 2}) again, we observe that
\begin{equation*}
\begin{split}
D(\tildef^{t+1},\tildeg^{t+1},h^{t}) - D(f^{t},\tildeg^{t+1},h^{t})  &= \varepsilon \int \left( e^{(f^{t}-\tildef^{t+1})/\varepsilon} - 1-(f^{t}-\tildef^{t+1})/\varepsilon \right)\, d\mu \\
&\ge \frac{e^{-2K_f/\varepsilon}}{2\varepsilon} \big \| f^{t} -\tildef^{t+1} \big \|_{L^2(\mu)}^2, \quad \text{and} \\
D(\tildef^{t+1},\tildeg^{t+1},\tildeh^{t+1}) - D(\tildef^{t+1},\tildeg^{t+1},h^{t})  &= \varepsilon \int \left( e^{(h^{t}-\tildeh^{t+1})/\varepsilon} - 1-(h^{t}-\tildeh^{t+1})/\varepsilon \right)\, d\nu \\
&\ge \frac{e^{-2K_h/\varepsilon}}{2\varepsilon} \big \| h^{t} - \tildeh^{t+1} \big \|_{L^2(\nu)}^2.
\end{split}
\end{equation*}
Putting these together, we obtain the first inequality.

For the second inequality, we first observe that
\begin{equation*}
\begin{split}
\ell_f^{t+1} &= e^{\tildef^{t+1}/\varepsilon} \int e^{(\langle \tildeg^{t+1},x\rangle+\tildeh^{t+1}-c)/\varepsilon} \, d\nu -
\underbrace{e^{\tildef^{t+1}/\varepsilon} \int e^{(\langle \tildeg^{t+1},x\rangle+h^{t}-c)/\varepsilon} \, d\nu}_{=1}, \\
\ell_g^{t+1} &= e^{\tildef^{t+1}/\varepsilon} \int x e^{(\langle \tildeg^{t+1},x\rangle+\tildeh^{t+1}-c)/\varepsilon}\, d\nu- e^{f^{t}/\varepsilon} \underbrace{\int x e^{(\langle \tildeg^{t+1},x\rangle+h^{t}-c)/\varepsilon} \big) \, d\nu}_{=0}, \\
\ell_h^{t+1}&= 0.
\end{split}
\end{equation*}
Recall that we have assumed that $M_x \ge 1$. 
In view of Proposition~\ref{prop: sinkhorn potential} and using inequality (\ref{eq: exponential 3}), we observe that
\begin{equation*}
\begin{split}
|\ell_f^{t+1}| &\le \varepsilon^{-1}e^{\bar{K}/\varepsilon} \big \| \tildeh^{t+1}-h^t\big \|_{L^1(\nu)}, \\
\big \| \ell_g^{t+1} \big \| &\le M_x\varepsilon^{-1} e^{\bar{K}/\varepsilon} \big ( \big | \tildef^{t+1}-f^t\big | + \big \| \tildeh^{t+1}-h^t\big \|_{L^1(\nu)}\big), 
\end{split}
\end{equation*}
which yield that
\begin{equation*}
\begin{split}
\big \|\ell_f^{t+1}\big \|_{L^2(\mu)}^2 &\le \varepsilon^{-2} e^{2\bar{K}/\varepsilon} \big \| \tildeh^{t+1}-h^t \big \|_{L^2(\nu)}^2, \\
\big \| \ell_g^{t+1} \big \|_{L^2(\mu)}^2 &\le 2\varepsilon^{-2}M_x^2 e^{2\bar{K}/\varepsilon} \big (\big \| \tildef^{t+1}-f^t \big \|_{L^2(\mu)}^2 + \big \| \tildeh^{t+1}-h^t \big \|_{L^2(\nu)}^2 \big ).
\end{split}
\end{equation*}
These inequalities give the second inequality in the statement of the lemma. 
\end{proof}

We are now ready to prove Theorem \ref{thm: linear conv}.

\begin{proof}[Proof of Theorem \ref{thm: linear conv}]
The proof is an adaptation of the argument in the proof of Theorem 14 in \cite{bolte2017error}.
Let $u_{t} :=  D(\barf,\barg,\barh) - D(f^{t},g^{t},h^{t})$. 
Combining Lemmas \ref{lem: PL} and \ref{lem: ascent}, and recalling the choice of $\tau$ in (\ref{eq: tau}), we obtain
\begin{equation*}
u_t - u_{t+1} \ge \tau u_{t+1}.
\end{equation*}
Solving this inequality yields the first claim.

We shall prove the second claim. If $u_{t_0}=0$ for some $t_0$, then $u_t=0$ for all $t \ge t_0$ by monotonicity of the dual objective along the Sinkhorn iterates (Lemma~\ref{lem: monotonicity}), which entails $(f^t,g^t,h^t) = (\barf,\barg,\barh)$ for all $t \ge t_0$ by uniqueness of dual potentials, so the second claim of the theorem follows trivially. 

Suppose now that $u_{t} > 0$ for all $t$.
On $\calH=L^2(\mu) \times L^2(\mu;\R^{d_x}) \times L^2(\nu)$, consider a metric
\begin{equation}
\mathsf{d}\big ( (f,g,h), (\tildef,\tildeg,\tildeh) \big) = \sqrt{\big \|f-\tildef\big \|_{L^2(\mu)}^2 + \big \| \langle g-\tildeg,x \rangle \big \|_{L^2(\mu \otimes \nu)}^2+\big \|h-\tildeh\big \|_{L^2(\nu)}^2},  \label{eq: metric}
\end{equation}
which is equivalent to the norm $\| \cdot \|_{\calH}$ in the sense that 
\begin{equation*}
\big(1 \wedge \sqrt{\underline{\lambda}}\big) \big \| (f-\tildef,g-\tildeg,h-\tildeh) \big \|_{\calH} \le \mathsf{d}\big ( (f,g,h), (\tildef,\tildeg,\tildeh) \big) \le M_x \big \| (f-\tildef,g-\tildeg,h-\tildeh) \big \|_{\calH}. 
\end{equation*}
The slope-ascent conditions in Lemma~\ref{lem: ascent} imply 
\begin{equation}
\begin{split}
&u_{t} - u_{t+1} \ge \alpha \mathsf{d}_{t+1}^2, \quad \text{with} \quad \mathsf{d}_{t+1} := \mathsf{d}\big ((f^{t},g^{t},h^{t}), (\tilde f^{t+1},\tilde g^{t+1},\tilde h^{t+1}) \big), \\
&\mathsf{d}_{t+1} \ge \beta \big\| (\ell_f^{t+1},\ell_g^{t+1},\ell_h^{t+1}) \big\|_{\calH}, 
\end{split}
\label{eq: slope}
\end{equation}
where $\alpha = \frac{e^{-2\bar{K}/\varepsilon} }{2\varepsilon}$ and $\beta =  \frac{(1 \wedge \sqrt{\underline{\lambda}})\varepsilon e^{-\bar{K}/\varepsilon} }{\sqrt{3} M_x}$.
The PL inequality in Lemma~\ref{lem: PL} implies
\begin{equation*}
2\sqrt{u_{t}} \le  \gamma^{-1} \big \| (\ell_f^{t},\ell_g^{t},\ell_h^{t})\big \|_{\calH} \quad \text{with} \quad \gamma = \sqrt{\frac{(1 \wedge \underline{\lambda})e^{-\bar{K}/\varepsilon}}{2\varepsilon}}, 
\end{equation*}
which yields that
\begin{equation*}
\begin{split}
\sqrt{u_{t}} - \sqrt{u_{t+1}} \ge \frac{1}{2\sqrt{u_{t}}} (u_{t} - u_{t+1}) \ge \underbrace{\alpha \beta \gamma}_{=:\zeta^{-1}} \frac{\mathsf{d}_{t+1}^2} {\mathsf{d}_t} \ge \zeta^{-1}(2\mathsf{d}_{t+1}-\mathsf{d}_t),
\end{split}
\end{equation*}
where we used the inequality $\mathsf{d}_{t+1}^2 \ge 2 \mathsf{d}_t \mathsf{d}_{t+1} - \mathsf{d}_{t}^2$.
Summing over $t$ gives
\begin{equation*}
\zeta(\sqrt{u_1}-\sqrt{u_{t+1}}) + \mathsf{d}_{1} \ge \sum_{s=1}^{t} \mathsf{d}_{s+1} + \mathsf{d}_{t+1} \ge \sum_{s=1}^t \mathsf{d}_{s+1}.
\end{equation*}
We observe that, for $a = \int \tildef^{t+1} \, d\mu$ and $v = \int \tildeg^{t+1}\, d\mu$, 
\begin{equation*}
\begin{split}
\mathsf{d}_{t+1}^2 &= \big \| f^t-f^{t+1}\big \|_{L^2(\mu)}^2 +a^2 +\big \|\langle g^t-g^{t+1},x\rangle \big \|_{L^2(\mu  \otimes \nu)}^2 + \big \| \langle v,x \rangle \big \|_{L^2(\nu)}^2 \\
&\quad + \big \| h^{t} - h^{t+1} +a  + \langle v,x \rangle \big \|_{L^2(\nu)}^2 \\
&= \big \| f^t-f^{t+1}\big \|_{L^2(\mu)}^2  +\big \|\langle g^t-g^{t+1},x\rangle \big \|_{L^2(\mu  \otimes \nu)}^2  + \big \| h^{t} - h^{t+1}\big \|_{L^2(\nu)}^2 \\
&\quad +2a^2 +2 \| \langle v,x \rangle\|_{L^2(\nu)}^2+  2\big\langle h^{t}-h^{t+1},a+\langle v,x \rangle \big \rangle_{L^2(\nu)} \\
&\ge  \big \| f^t-f^{t+1}\big \|_{L^2(\mu)}^2  +\big \|\langle g^t-g^{t+1},x\rangle \big \|_{L^2(\mu  \otimes \nu)}^2  +\frac{1}{2} \big \| h^{t} - h^{t+1}\big \|_{L^2(\nu)}^2 \\
&\ge \underbrace{\frac{1}{2}\mathsf{d}^2\big( (f^t,g^t,h^t), (f^{t+1},g^{t+1},h^{t+1}) \big )}_{=:\Delta_{t+1}^2},
\end{split}
\end{equation*}
where we used the inequality
\begin{equation*}
\begin{split}
2\big\langle h^{t}-h^{t+1},a+\langle v,x \rangle\big \rangle_{L^2(\nu)} &\ge - \frac{1}{2} \big \| h^t-h^{t+1} \big \|_{L^2(\nu)}^2 - 2 \| a+\langle v,x \rangle \|_{L^2(\nu)}^2 \\
&=- \frac{1}{2} \big \| h^t-h^{t+1} \big \|_{L^2(\nu)}^2 - 2a^2-2\| \langle v,x \rangle \|_{L^2(\nu)}^2. 
\end{split}
\end{equation*}
%Similarly, we have $\mathsf{d}_{t+1}^2 \le (3/2)\Delta_{t+1}^2 \le 2 \Delta_{t+1}^2$. 
These inequalities yield
\begin{equation*}
\zeta(\sqrt{u_1}-\sqrt{u_{t+1}}) +\mathsf{d}_{1} \ge \sum_{s=1}^t \Delta_{s+1}
\end{equation*}
\iffalse
\begin{equation*}
\sqrt{u_{t}} - \sqrt{u_{t+1}} \ge \underbrace{\frac{\alpha \beta \gamma}{2\sqrt{2}}}_{=\zeta^{-1}} \frac{\Delta_{t+1}^2}{\Delta_t} \ge \zeta^{-1} (2\Delta_{t+1}-\Delta_t),
\end{equation*}
where we used the inequality $\Delta_{t+1}^2 \ge 2 \Delta_t \Delta_{t+1} - \Delta_{t}^2$. Summing over $t$ gives
\begin{equation*}
\zeta(\sqrt{u_1}-\sqrt{u_{t+1}}) + \Delta_1 \ge \sum_{s=1}^{t} \Delta_{s+1} + \Delta_{t+1} \ge \sum_{s=1}^t \Delta_{s+1}.
\end{equation*}
\fi
This implies that the sequence $(f^{t},g^{t},h^{t})$ is Cauchy in $\calH$. Since $\calH$ is complete, there exists $(f^\infty,g^\infty,h^\infty) \in \calH$ such that $(f^{t},g^{t},h^{t}) \to (f^\infty,g^\infty,h^\infty)$ in $\calH$. Observe that $\int f^\infty \, d\mu = 0$ and $\int g^\infty \, d\mu = 0$. By taking an a.s. convergent subsequence, and using Proposition~\ref{prop: sinkhorn potential}, we have $\| f^\infty \|_{L^\infty(\mu)} \le K_f, \| g^\infty \|_{L^\infty(\mu)} \le K_g$, and $\| h^\infty \|_{L^\infty(\nu)} \le K_h$. An application of the dominated convergence theorem yields that, along a subsequence,
\begin{equation*}
D(f^\infty,g^\infty,h^\infty) = \lim_{t \to \infty} D(f^{t},g^{t},h^{t}) = D(\barf,\barg,\barh).
\end{equation*}
By uniqueness of dual potentials, we conclude $f^\infty = \barf, g^\infty =  \barg$  $\mu$-a.e. and $h^\infty = \barh$ $\nu$-a.e.

Finally, we observe
\begin{equation*}
\frac{1}{\sqrt{2}}\mathsf{d} \big ( (f^t,g^t,h^t), (f^{t+m},g^{t+m},h^{t+m}) \big) \le \sum_{s=t}^{t+m-1} \Delta_{s+1} \le \zeta (\sqrt{u_t}-\sqrt{u_{t+m}}) + \mathsf{d}_{t}.
\end{equation*}
The first inequality in (\ref{eq: slope}) implies 
\begin{equation*}
\sqrt{\frac{u_{t-1}-u_{t}}{\alpha}} \ge \mathsf{d}_t,
\end{equation*}
so that
\begin{equation*}
\mathsf{d} \big ( (f^t,g^t,h^t), (f^{t+m},g^{t+m},h^{t+m}) \big)\le \zeta (\sqrt{2u_t}-\sqrt{2u_{t+m}}) + \sqrt{\frac{2(u_{t-1}-u_{t})}{\alpha}}.
\end{equation*}
Letting $m \to \infty$, we conclude 
\begin{equation*}
\mathsf{d} \big ( (f^t,g^t,h^t), (\barf,\barg,\barh) \big) \le \zeta \sqrt{2u_t} + \sqrt{\frac{2u_{t-1}}{\alpha}} = O\big( \big (1+\tau\big)^{-t/2} \big).
\end{equation*}
This completes the proof.
\end{proof}

\section{Proofs for Section \ref{sec: modified sinkhorn}}
\label{sec: proof sec 4}

\subsection{Proof of Proposition~\ref{prop: sinkhorn potential modified}}
The proposition follows from Lemmas \ref{lem: elemenary bounds} and \ref{lem: bounds on g modified} below. We start by establishing monotonicity of the dual objective values along the modified Sinkhorn iterates. 

\begin{lemma}[Monotonicity lemma]
\label{lem: monotonicity modified}
If $0 < \eta \le \frac{2\varepsilon}{M_x^2}$, then we have for $t \in \N_0$,
\[
\begin{split}
D (f^{t},g^{t},h^{t}) &\le D  (f^{t},g^{t},\tildeh^{t+1}) \le D(\tildef^{t+1},g^{t},\tildeh^{t+1}) \\
&\le D(\tildef^{t+1},\tildeg^{t+1},\tildeh^{t+1}) = D(f^{t+1},g^{t+1},h^{t+1}).
\end{split}
\]
\end{lemma}
\begin{proof}
The first two inequalities follow similarly to Lemma~\ref{lem: monotonicity}, so we shall focus on proving the third inequality. The key idea, partially inspired by the proof of Proposition~2.7 in \cite{kato2026entropic}, is to think of the gradient $d_g^t(u)$ as the mean parameter of a certain exponential family. We refer the reader to Chapter 3 in \cite{wainwright2008graphical} as an excellent reference on exponential families. 

Let $\nu_x$ denote the conditional distribution of $Y$ given $X$ and let $\kappa$ denote the marginal distribution of $X$, so that $d\nu(x,y) = d\nu_x(y) d\kappa(x)$.
Fix $t \in \N_0$. For $(u,\theta) \in \calU \times \R^{d_x}$, define 
\[
A_u(\theta) := \varepsilon \log \int_{\calX} e^{\langle \theta,x \rangle/\varepsilon} d\tilde{\kappa}_u(x), 
\]
where $\tilde{\kappa}_u$ is a Borel measure on $\calX$ defined by 
\[
\frac{d\tilde{\kappa}_u}{d\kappa}(x) :=\int_{\calY} e^{(\tildeh^{t+1}(x,y)-c(u,y))/\varepsilon} \, d\nu_x(y). 
\]
Since $A_u(\theta)$ is the (scaled) log-partition function of the exponential family of density functions $\{ x \mapsto e^{(\langle \theta,x\rangle - A_u(\theta))/\varepsilon} : \theta \in \R^{d_x} \}$ with respect to the base measure $\tilde{\kappa}_u$, its gradient and Hessian agree with the mean and $\frac{1}{\varepsilon}$ times the covariance matrix for the probability measure $e^{(\langle \theta,x\rangle - A_u(\theta))/\varepsilon}\, d\tilde{\kappa}_u(x)$; see Proposition~3.1 in \cite{wainwright2008graphical}. In particular, this implies that $A_u(\theta)$ is convex in $\theta$. Since $\tilde{\kappa}_u$ is supported in $\calX$, one obtains
\[
\| \nabla A_u(\theta) \| \le M_x \quad \text{and} \quad \| \nabla^2 A_u(\theta) \|_{\op} \le \frac{M_x^2}{\varepsilon}, \quad \forall \theta \in \R^{d_x}.
\]
Observe now that $\nabla A_u(g^{t}(u)) = d_g^t(u)$, 
which implies that 
\begin{equation}
\label{eq: gradient bound}
\big \| d_g^t \big \|_{\infty} \le M_x.
\end{equation}
Furthermore, as $\| \nabla^2 A_u(z) \|_{\op} \le \frac{M_x^2}{\varepsilon}$, one has
\begin{equation}
\begin{split}
A_u(\tildeg^{t+1}(u)) - A_u(g^t(u)) &\le -\eta \big \| d_g^t(u) \big \|^2 + \frac{\eta^2M_x^2}{2\varepsilon} \big \| d_g^t(u) \big \|^2 \\
&= -\eta \left ( 1-\frac{\eta M_x^2}{2\varepsilon} \right ) \big \| d_g^t(u) \big \|^2 \\
&= -\frac{1}{\eta} \left ( 1-\frac{\eta M_x^2}{2\varepsilon} \right ) \big \| \tildeg^{t+1}(u) - g^{t}(u) \big \|^2, 
\end{split}
\label{eq: slope in g}
\end{equation}
where the far right-hand side is nonpositive for our choice of $\eta$. The desired inequality follows from the observation that
\begin{equation}
D(\tildef^{t+1},\tildeg^{t+1},\tildeh^{t+1}) - D(\tildef^{t+1},g^{t},\tildeh^{t+1}) = \varepsilon \int_{\calU} \big( 1-e^{(A_u(\tildeg^{t+1}(u))-A_u(g^{t}(u)))/\varepsilon}  \big) \, d\mu(u),
\label{eq: dual increase in g}
\end{equation}
where we used the fact that $A_u(g^{t}(u)) = -\tildef^{t+1}(u)$. 
This completes the proof.
\end{proof}
\begin{remark}
Observe that
\[
\frac{\eta}{\varepsilon} \left ( 1-\frac{\eta M_x^2}{2\varepsilon} \right ) \big \| d_g^t(u) \big \|^2 \le \frac{\eta M_x^2}{\varepsilon}\left ( 1-\frac{\eta M_x^2}{2\varepsilon} \right ) \le \frac{1}{2}.
\]
In view of (\ref{eq: slope in g}) and (\ref{eq: dual increase in g}), and using $1-e^{-a} \ge e^{-1/2}a$ for $0 \le a \le 1/2$, one has
\[
D(\tildef^{t+1},\tildeg^{t+1},\tildeh^{t+1}) - D(\tildef^{t+1},g^{t},\tildeh^{t+1}) \ge \frac{e^{-1/2}}{\eta} \left ( 1-\frac{\eta M_x^2}{2\varepsilon} \right) \big \| \tildeg^{t+1}-g^t \big \|_{L^2(\mu)}^2.
\]
Likewise, provided that $0 < \eta <  \frac{2\varepsilon}{M_x^2}$, we have
\[
\begin{split}
\big \| d_g^t \big \|_{L^2(\mu)}^2 &\le \frac{e^{1/2}}{\eta} \left ( 1-\frac{\eta M_x^2}{2\varepsilon} \right)^{-1} \big ( D(\tildef^{t+1},\tildeg^{t+1},\tildeh^{t+1}) - D(\tildef^{t+1},g^{t},\tildeh^{t+1})\big) \\
&\le \frac{e^{1/2}}{\eta} \left ( 1-\frac{\eta M_x^2}{2\varepsilon} \right)^{-1} \big ( D(f^{t+1},g^{t+1},h^{t+1}) - D(f^{t},g^{t},h^{t})\big).
\end{split}
\]
\end{remark}

In what follows, we assume 
\[
0 < \eta <  \frac{2\varepsilon}{M_x^2}.
\]
The following bounds follow similarly to Lemma~\ref{lem: bound on f and h}.
\begin{lemma}
\label{lem: elemenary bounds}
For $t \in \N_0$, 
\begin{equation}
\begin{split}
&\big \| \tildef^{t+1} \big \|_\infty \vee \big \| f^{t+1} \big \|_\infty \le L_c \diam (\calU) + 2\big \|g^t\big \|_{\infty}M_x +\|c\|_\infty - D_0, \\
&\big\| \tildeh^{t+1} \big\|_{\infty} \vee \big\| h^{t+1} \big\|_{\infty}\le \|c\|_\infty +2 \big( \big \|f^{t} \big \|_\infty \vee \big \| \tildef^{t+1} \big \|_\infty \big) + 2\big ( \big\| g^t \big\|_{\infty} \vee \big \| \tildeg^{t+1} \big \|_\infty \big) M_x.
\end{split}
\label{eq: Khat}
\end{equation}
\end{lemma}

\begin{proof}
The bound on $\big \| \tildeh^{t+1} \big \|_{\infty} \vee \big \| h^{t+1} \big \|_{\infty}$ follows from the definition and Jensen's inequality. 
Pick any $u,u' \in \calU$, and define a probability measure $\rho_{u'}$ on $\calX \times \calY$ as
\begin{equation*}
d\rho_{u'}(x,y) \propto \exp\left(\frac{\langle g^{t}(u'), x \rangle +\tildeh^{t+1}(x,y) -c(u',y)}{\varepsilon}\right) \,d\nu(x,y).
\end{equation*}
Arguing as in the first part of the proof of Lemma~\ref{lem: bound on f and h}, we have
\begin{equation*}
\begin{split}
&-\tildef^{t+1}(u) + \tildef^{t+1}(u') \\
&\ge  \int \big (c(u',y) - c(u,y) + \langle g^{t}(u)-g^{t}(u'),x \rangle \big) \, d\rho_{u'}(x,y) \\
&\ge -L_c \diam (\calU) - 2\big \|g^t\big \|_{\infty} M_x. 
\end{split}
\end{equation*}
Interchanging the roles of $u$ and $u'$ and using the normalization $\int f^{t+1} \, d\mu = 0$, we have 
\begin{equation*}
\big \| f^{t+1} \big \|_\infty \le L_c \diam (\calU) + 2\big \|g^t\big \|_{\infty}M_x.
\end{equation*}

Next, by Lemma~\ref{lem: monotonicity modified} and weak duality, we have
\begin{equation*}
\begin{split}
&D_0 \le D(\tildef^{t+1},g^{t},\tildeh^{t+1}) = \int \tildef^{t+1} \, d\mu + \int \tildeh^{t+1} \, d\nu \le \|c\|_\infty, \quad \text{and} \\
&D_0 \le D(f^{t},g^{t},\tildeh^{t+1}) = \int \tildeh^{t+1} \, d\nu \le \|c\|_\infty. 
\end{split}
\end{equation*}
Combining these, we have $|\int \tildef^{t+1} \, d\mu| \le \|c\|_\infty-D_0$, which in turn yields the desired bound on $\big \| \tildef^{t+1}\big \|_\infty$. 
\end{proof}

Finding bounds on $g^{t+1}$ is considerably more complicated. Observe that $\int \tildeg^{t+1} \, d\mu = -\eta \int d_g^t \, d\mu =: -\eta v^t$ and 
\begin{equation}
\begin{split}
\big \| g^{t+1}(u) \big \|^2 &= \big \| g^t(u) - \eta (d_g^t(u)-v^t)\big \|^2 \\
&=\big \| g^t(u) \big \|^2 -2\eta \big \langle g^t(u),d_g^t(u) \big \rangle + 2\eta \big \langle g^t(u),v^t \big \rangle + \eta^2 \big \|d_g^t(u) - v^t\big \|^2 \\
&\le \big \| g^t(u) \big \|^2 -2\eta \big \langle g^t(u),d_g^t(u) \big \rangle + 2\eta \big \| g^t(u) \big \| \big \| v^t\big \| +  4\eta^2 M_x^2,
\end{split}
\label{eq: norm bound g}
\end{equation}
where we used inequality (\ref{eq: gradient bound}). One needs to find a lower bound on the inner product, $\big \langle g^t(u),d_g^t(u) \big \rangle$, which is given next.  

\begin{lemma}
For $t \in \N_0$,
\[
\big \langle g^t(u),d_g^t(u) \big \rangle \ge \delta \big \| g^t(u) \big \| - \underbrace{\left  ( 5\|c\|_\infty -\frac{3}{2}D_0 + \varepsilon \log \left ( \frac{3}{2} \right )\right )}_{= K_g^{\star \star}}, \quad \forall u \in \calU,
\]
where $\delta := 1/(2\|\Sigma_X^{-1}\|_{\op}M_x)$. 
\end{lemma}
\begin{proof}
Pick any $u \in \calU$. We follow the notation used in the proof of Lemma~\ref{lem: monotonicity modified}.
Recall from the proof of Lemma~\ref{lem: monotonicity modified} that $d_g^t (u) = \nabla A_u(g^t(u))$, so that 
\begin{equation}
\big \langle g^t(u),d_g^t(u) \big \rangle  = \big \langle  g^t(u),\nabla A_u(g^t(u)) \big \rangle. 
\label{eq: inner product}
\end{equation}
Pick any $\theta \in \R^{d_x}$. Since $A_u(\theta)$ is convex in $\theta$,
\begin{equation}
\langle \theta,\nabla A_u(\theta) \rangle \ge A_u(\theta)-A_u(0). 
\label{eq: initial lower bound}
\end{equation}
Observe that
\[
A_u(0) = \varepsilon \log \int_{\calX \times \calY} e^{(\tildeh^{t+1}(x,y)-c(u,y))/\varepsilon} \, d\nu(x,y) \le \sup_{\calX \times \calY}\tildeh^{t+1} \le \| c \|_\infty,
\]
where the last inequality follows from the definition of $\tildeh^{t+1}$ and Jensen's inequality. As such, for any $\theta \in \R^{d_x}$,
\[
\langle \theta,\nabla A_u(\theta) \rangle \ge A_u(\theta)-\| c \|_\infty. 
\]
To find a lower bound on $A_u(\theta)$, as in the proof of Proposition~\ref{prop: potential bound}, for any $v \in \mathbb{S}^{d_x-1}$, define a probability measure $q^{(v)}$ on $\calX \times \calY$ by 
\begin{equation*}
dq^{(v)}(x,y) = e^{\hat{p}_v(x,y)}\,d\nu(x,y) \quad \text{with} \quad \hat{p}_v(x,y) = \log \big(1 + \delta \langle \Sigma_X^{-1}v, x \rangle \big).
\end{equation*}
Consider a function $p_u(x,y)$ on $\calX \times \calY$ defined by
\[
p_u(x,y) = \frac{1}{\varepsilon} \big ( \langle \theta,x\rangle + \tildeh^{t+1}(x,y) - c(u,y) - A_u(\theta) \big). 
\]

Now, using inequality (\ref{eq: exponential}), we obtain
\[
0 = \int (e^{p_u} - e^{\hat{p}_v}) \, d \nu \ge \int (p_u - \hat{p}_v) \, dq^{(v)}. 
\]
Multiplying both sides by $\varepsilon$ and using the fact that $\int_{\calX \times \calY}x \, dq^{(v)}(x,y)=\delta v$, we have 
\[
\begin{split}
0 &\ge \delta \langle \theta,v \rangle + \int \tildeh^{t+1} \, dq^{(v)} - A_u(\theta) - \|c\|_\infty  - \varepsilon \int \hat{p}_v \, dq^{(v)} \\
&\ge \delta \langle \theta,v \rangle  + \int \tildeh^{t+1} \, dq^{(v)}  - A_u(\theta) - \|c\|_\infty - \varepsilon \log \left ( \frac{3}{2} \right ).
\end{split}
\]
Rearranging terms, we conclude 
\begin{equation}
A_u(\theta) \ge \delta \langle \theta,v \rangle  + \int \tildeh^{t+1} \, dq^{(v)}  - \|c\|_\infty - \varepsilon \log \left ( \frac{3}{2} \right ).
\label{eq: lower bound A}
\end{equation}
Finally, we need to find a lower bound on $\int \tildeh^{t+1} \, dq^{(v)}$. By the definition of $q^{(v)}$, 
\[
\begin{split}
\int \tildeh^{t+1} \, dq^{(v)} &\ge - \big \| \tildeh^{t+1} \big \|_{L^1(q^{(v)})} \ge -\frac{3}{2} \big \| \tildeh^{t+1}\big \|_{L^1(\nu)} \\
&\ge -\frac{3}{2} \left ( 2\int (\tildeh^{t+1})^+ d\nu - \int \tildeh^{t+1} \, d\nu \right ) \\
&\ge - \frac{3}{2}(2\|c\|_\infty - D_0),
\end{split}
\]
where we used $\tildeh^{t+1} \le \|c\|_\infty$, which follows from Jensen's inequality, and $\int \tildeh^{t+1} \, d\nu = D(f^t,g^t,\tildeh^{t+1}) \ge D_0$, which follows from Lemma~\ref{lem: monotonicity modified}. Plugging this lower bound into (\ref{eq: lower bound A}) and choosing $v = \theta/\| \theta \|$ yields
\begin{equation}
A_u(\theta) \ge \delta \| \theta \|  - \left  ( 4\|c\|_\infty -\frac{3}{2}D_0 + \varepsilon \log \left ( \frac{3}{2} \right )\right ). 
\label{eq: lower bound A 2}
\end{equation}
Combining (\ref{eq: inner product}), (\ref{eq: initial lower bound}), and (\ref{eq: lower bound A 2}) gives the desired result. 
\end{proof}

Given the preceding lemma, an induction argument leads to the following bound on $\big \| \tildeg^{t+1} \big \|_{\infty} \vee \big \| g^{t+1} \big \|_{\infty}$. 

\begin{lemma}
\label{lem: bounds on g modified}
For $t \in \N_0$, 
\[
\begin{split}
\big \| \tildeg^{t+1} \big \|_{\infty} \vee \big \| g^{t+1} \big \|_{\infty} &\le  \eta M_x + \left [ \| g^0 \|_\infty \vee \left ( \frac{2 K_g^{\star \star}+6\eta (1 \vee \delta) M_x^2}{\delta} \right )\right] \\
&\quad + \frac{2e^{1/2}}{\delta} \left ( 1-\frac{\eta M_x^2}{2\varepsilon} \right)^{-1} \left ( 1+\frac{2M_x}{\delta}  \right )(\bar{D}-D_0) =K_g^{\star}.
\end{split}
\]
\end{lemma}

\begin{remark}
As a consequence, in view of Lemma~\ref{lem: elemenary bounds}, we have
\[
\begin{split}
&\big \| \tildef^{t+1} \big \|_\infty \vee \big \| f^{t+1} \big \|_\infty \le L_c \diam (\calU) + 2K_g^{\star}M_x + \|c\|_\infty - D_0 = K_f^{\star}, \\
&\big \| \tildeh^{t+1} \big \|_\infty \vee \big \| h^{t+1} \big \|_\infty \le \|c\|_\infty + 2K_f^{\star}+2K_g^{\star}M_x = K_h^{\star}.
\end{split}
\]
\end{remark}
\begin{proof}
We divide the proof into two steps. Recall $v^t = \int d_g^t \, d\mu$, so that $g^{t+1} = \tilde{g}^{t+1} + \eta v^t=g^{t} - \eta (d_g^t-v^t)$. Pick any $u \in \calU$.

\medskip

\textbf{Step 1.}
Combining the preceding lemma and the inequality in (\ref{eq: norm bound g}), one has
\begin{equation}
\big \| g^{t+1}(u) \big \|^2 \le \big \| g^t (u) \big \|^2 - 2\eta (\delta - \big \| v^t \big \|) \big \|g^t(u)\big \| + 2\eta K_g^{\star \star}+4\eta^2M_x^2. 
\label{eq: recursive}
\end{equation}
We consider the following three cases:

\medskip

\textbf{Case 1: $\big \| v^t \big \| \le \delta/2$ and $\big \|g^t(u)\big \| > \frac{2 K_g^{\star \star}+4\eta M_x^2}{\delta}$.} In this case, by inequality (\ref{eq: recursive}), we have 
\[
\big \| g^{t+1} (u) \big \| \le \big \| g^t(u) \big \|.
\]

\medskip

\textbf{Case 2: $\big \| v^t \big \| \le \delta/2$ and $\big \|g^t(u)\big \| \le \frac{2 K_g^{\star \star}+4\eta M_x^2}{\delta}$.} In this case, we have
\[
\big \| g^{t+1}(u) \big \| \le \big \| g^t(u) \big \| + \eta \big \| d_g^t(u)-v^t\big \| \le \frac{2 K_g^{\star \star}+4\eta M_x^2}{\delta} + 2\eta M_x \le \frac{2 K_g^{\star \star}+6\eta (1 \vee \delta) M_x^2}{\delta},
\]
where we used inequality (\ref{eq: gradient bound}). 

\medskip

\textbf{Case 3: $\big \| v^t \big \| > \delta/2$.} In this case, we have 
\[
\big \| g^{t+1} (u) \big \| \le \big \| g^{t}(u) \big \| + \eta (M_x+\big \| v^t \big \|). 
\]

\medskip

By induction, in general, we obtain
\[
\big \| g^{t+1}(u) \big \| \le \| g^0 \|_\infty \vee \left ( \frac{2 K_g^{\star \star}+6\eta (1 \vee \delta) M_x^2}{\delta} \right ) + \sum_{t': \| v^{t'} \| > \delta/2} \eta \big(M_x+\big \|v^{t'}\big \|\big). 
\]

\textbf{Step 2.} In this step, we shall find upper bounds on $\sum_{t: \| v^t\| > \delta/2} 1$ and $\sum_{t: \| v^t\| > \delta/2} \big \| v^t \big \|$. 
Recall from the remark following Lemma~\ref{lem: monotonicity modified} that
\[
\big \| v^t \big \|^2 \le \big \| d_g^t \big \|_{L^2(\mu)}^2 \le \underbrace{\frac{e^{1/2}}{\eta} \left ( 1-\frac{\eta M_x^2}{2\varepsilon} \right)^{-1}}_{=:\zeta} \big ( D(f^{t+1},g^{t+1},h^{t+1}) - D(f^{t},g^{t},h^{t})\big),
\]
where the first inequality follows from Jensen's inequality. Hence,
\[
\sum_{t \le T}\big \|v^t\big \|^2 \le \zeta \big ( D(f^{T+1},g^{T+1},h^{T+1})-D_0\big) \le \zeta (\bar{D}-D_0)
\]
for any finite $T \in \N_0$. Sending $T \to \infty$, we have 
\[
\sum_{t \in \N_0}\big \|v^t\big \|^2 \le \zeta (\bar{D}-D_0). 
\]
This inequality yields
\[
\begin{split}
\sum_{t: \| v^t \| > \delta/2} 1 &\le \frac{4}{\delta^2} \sum_{t: \| v^t \| > \delta/2} \big \| v^t \big \|^2 \le \frac{4\zeta}{\delta^2} (\bar{D}-D_0), \\
\sum_{t: \| v^t \| > \delta/2} \big \| v^t \big \| &\le \left (\sum_{t: \| v^t \| > \delta/2} 1 \right)^{1/2} \left ( \sum_{t: \| v^t \| > \delta/2} \big \| v^t \big \|^2 \right)^{1/2} \le  \frac{2\zeta}{\delta} (\bar{D}-D_0),
\end{split}
\]
where we used the Cauchy-Schwarz inequality to deduce the second inequality. As such, we obtain
\[
\sum_{t: \| v^t\| > \delta/2} \eta\big(M_x+\big \|v^t\big \|\big) \le \frac{2\eta \zeta}{\delta} \left ( 1+\frac{2M_x}{\delta}  \right )(\bar{D}-D_0). 
\]
The desired conclusion follows from combining Steps 1 and 2. 
\end{proof}

\subsection{Proof of Theorem \ref{thm: linear conv modified}}
Given Proposition~\ref{prop: sinkhorn potential modified}, the proof of Theorem \ref{thm: linear conv modified} is largely similar to that of Theorem \ref{thm: linear conv}.
Recall
\[
\bar{K}^{\star} = K_f^{\star}+K_g^{\star}M_x + K_h^{\star} + \| c \|_\infty.
\]
As before, let 
\[
(\ell_f^t,\ell_g^t,\ell_h^t) := \big (\ell_f(f^t,g^t,h^t), \ell_g(f^t,g^t,h^t),\ell_h(f^t,g^t,h^t) \big).
\]
The following PL inequality follows readily from Propositions \ref{prop: PL} and \ref{prop: sinkhorn potential modified}.
\begin{lemma}[PL inequality]
\label{lem: PL modified}
For $t \in \N_0$, we have 
\[
\big \| (\ell_f^t,\ell_g^t,\ell_h^t) \big \|_{\calH}^2 \ge \frac{2(1 \wedge \underline{\lambda})e^{-\bar{K}^{\star}/\varepsilon}}{\varepsilon} \big (\bar{D} - D(f^t,g^t,h^t) \big). 
\]
\end{lemma}

Next, we establish slope ascent conditions. We set
\[
\alpha^{\star} =\frac{e^{-2K_h^{\star}/\varepsilon}}{2\varepsilon} \bigwedge \left (  \frac{e^{-1/2}}{M_x^2\eta} \left ( 1-\frac{\eta M_x^2}{2\varepsilon} \right) \right ) \quad \text{and} \quad \beta^{\star} = \left ( \frac{2}{\underline{\lambda}\eta^2} + \frac{5M_x^2e^{2\bar{K}^{\star}/\varepsilon}}{\varepsilon^2} \right )^{-1/2}.
\]

\begin{lemma}[Slope-ascent conditions]
\label{lem: slope ascent modified}
For $t \in \N_0$, we have 
\begin{align*}
&D(f^{t+1},g^{t+1},h^{t+1}) - D(f^t,g^t,h^t) \\
&\quad \ge \alpha^{\star} \big ( \big \|f^t-\tildef^{t+1}\big \|_{L^2(\mu)}^2 + \big \| \langle g^t-\tildeg^{t+1},x \rangle \big \|_{L^2(\mu \otimes \nu)}^2+\big \|h^t-\tildeh^{t+1}\big \|_{L^2(\nu)}^2 \big ) \\
\intertext{and}
&\sqrt{\big \|\tildef^{t+1} - f^t\big \|_{L^2(\mu)}^2 +\big \| \langle \tildeg^{t+1} - g^t,x\rangle \big \|_{L^2(\mu \otimes \nu)}^2} \ge \beta^{\star} \big \| (\ell_f^{t+1},\ell_g^{t+1}, \ell_h^{t+1})\big\|_{\calH}. 
\end{align*}
\end{lemma}
\begin{proof}
For the first inequality, decompose $D(f^{t+1},g^{t+1},h^{t+1}) - D(f^t,g^t,h^t)$ as 
\begin{equation*}
\begin{split}
D(f^{t+1},g^{t+1},h^{t+1}) - D(f^t,g^t,h^t) &=D(\tildef^{t+1},\tildeg^{t+1},\tildeh^{t+1}) - D(\tildef^{t+1},g^{t},\tildeh^{t+1}) \\
&\quad +D(\tildef^{t+1},g^{t},\tildeh^{t+1}) - D(f^{t},g^{t},\tildeh^{t+1}) \\
&
\quad +D(f^{t},g^{t},\tildeh^{t+1}) - D(f^{t},g^{t},h^{t}),
\end{split}
\end{equation*}
where we used the fact that $D(f^{t+1},g^{t+1},h^{t+1}) = D(\tildef^{t+1},\tildeg^{t+1},\tildeh^{t+1})$. From the remark following Lemma~\ref{lem: monotonicity modified}, we have 
\begin{equation*}
\begin{split}
D(\tildef^{t+1},\tildeg^{t+1},\tildeh^{t+1}) - D(\tildef^{t+1},g^{t},\tildeh^{t+1}) &\ge \frac{e^{-1/2}}{\eta} \left ( 1-\frac{\eta M_x^2}{2\varepsilon} \right) \big \| \tildeg^{t+1}-g^t \big \|_{L^2(\mu)}^2 \\
&\ge \frac{e^{-1/2}}{M_x^2\eta} \left ( 1-\frac{\eta M_x^2}{2\varepsilon} \right) \big\| \langle \tildeg^{t+1}-g^t,x \rangle  \big\|_{L^2(\mu \otimes \nu)}^2.
\end{split}
\end{equation*}
Next, using inequality (\ref{eq: exponential 2}), we observe that
\begin{equation*}
\begin{split}
D(\tildef^{t+1},g^{t},\tildeh^{t+1}) - D(f^{t},g^{t},\tildeh^{t+1})  &= \varepsilon \int \left( e^{(f^t-\tildef^{t+1})/\varepsilon} - 1-(f^t-\tildef^{t+1}) / \varepsilon \right)\, d\mu \\
&\ge \frac{e^{-2K_f^{\star}/\varepsilon}}{2\varepsilon} \big \| f^t-\tildef^{t+1} \big \|_{L^2(\mu)}^2, \quad \text{and} \\
 D(f^{t},g^{t},\tildeh^{t+1}) - D(f^{t},g^{t},h^{t})  &= \varepsilon \int \left( e^{(h^t-\tildeh^{t+1})/\varepsilon} - 1-(h^t-\tildeh^{t+1})/\varepsilon \right)\, d\nu \\
&\ge \frac{e^{-2K_h^{\star}/\varepsilon}}{2\varepsilon} \big \| h^t-\tildeh^{t+1}\big \|_{L^2(\nu)}^2.
\end{split}
\end{equation*}
Putting these together, we obtain the first inequality.

For the second inequality, we first observe that 
\begin{equation*}
\begin{split}
\ell_f^{t+1} &= e^{\tildef^{t+1}/\varepsilon} \int e^{(\langle \tildeg^{t+1},x\rangle+\tildeh^{t+1}-c)/\varepsilon} \, d\nu-e^{\tildef^{t+1}/\varepsilon}\int e^{(\langle g^t,x\rangle+\tildeh^{t+1}-c)/\varepsilon} \, d\nu, \\
\ell_g^{t+1} -d_g^t&= e^{\tildef^{t+1}/\varepsilon}\int x e^{(\langle \tildeg^{t+1},x\rangle+\tildeh^{t+1}-c)/\varepsilon}\, d\nu - e^{\tildef^{t+1}/\varepsilon}\int x e^{(\langle g^t,x\rangle+\tildeh^{t+1}-c)/\varepsilon}\, d\nu, \\
\ell_h^{t+1}&=e^{\tildeh^{t+1}/\varepsilon} \int  e^{(\tildef^{t+1}+\langle \tildeg^{t+1},x\rangle-c)/\varepsilon} \, d\mu -e^{\tildeh^{t+1}/\varepsilon} \int  e^{(f^t+\langle g^t,x\rangle-c)/\varepsilon} \, d\mu.
\end{split}
\end{equation*}
Using inequality (\ref{eq: exponential 3}), we obtain
\begin{equation*}
\begin{split}
&|\ell_f^{t+1}| \le \frac{e^{\bar{K}^{\star}/\varepsilon}}{\varepsilon}\big \| \langle \tildeg^{t+1} - g^t,x\rangle \big \|_{L^1(\nu)},\\
&\big \|\ell_g^{t+1}-d_g^t\big \| \le \frac{M_xe^{\bar{K}^{\star}/\varepsilon}}{\varepsilon} \big \| \langle \tildeg^{t+1} - g^t,x \rangle \big \|_{L^1(\nu)},\\
&|\ell_h^{t+1}| \le \frac{e^{\bar{K}^{\star}/\varepsilon}}{\varepsilon} \big(\big\|\tildef^{t+1} - f^t\big\|_{L^1(\mu)} +\big \| \langle \tildeg^{t+1} - g^t,x \rangle \big \|_{L^1(\mu)}\big). 
\end{split}
\end{equation*}
These estimates yield
\begin{equation*}
\begin{split}
&\big \|\ell_f^{t+1}\big \|_{L^2(\mu)}^2 \le \frac{e^{2\bar{K}^{\star}/\varepsilon}}{\varepsilon^2} \big \| \langle \tildeg^{t+1} - g^t,x\rangle \big \|_{L^2(\mu \otimes \nu)}^2, \\
&\big \|\ell_g^{t+1}-d_g^t\big \|_{L^2(\mu)}^2 \le \frac{M_x^2e^{2\bar{K}^{\star}/\varepsilon}}{\varepsilon^2} \big \| \langle \tildeg^{t+1} - g^t,x\rangle \big \|_{L^2(\mu \otimes \nu)}^2,\\
&\big \|\ell_h^{t+1}\big \|_{L^2(\nu)}^2 \le \frac{2e^{2\bar{K}^{\star}/\varepsilon}}{\varepsilon^2} \big(\big \|\tildef^{t+1} - f^t\big \|_{L^2(\mu)}^2 +\big \| \langle \tildeg^{t+1} - g^t,x\rangle \big \|_{L^2(\mu \otimes \nu)}^2\big).
\end{split}
\end{equation*}
Finally, we have
\begin{equation*}
\begin{split}
&\big \|\ell_g^{t+1}\big \|_{L^2(\mu)}^2 \le 2\big \| d_g^t\big \|_{L^2(\mu)}^2 + 2\big \| \ell_g^{t+1}-d_g^t\big \|_{L^2(\mu)}^2\\
&\le \frac{2}{\eta^2}\big \|\tildeg^{t+1}-g^t\big \|_{L^2(\mu)}^2 + \frac{2M_x^2e^{2\bar{K}^{\star}/\varepsilon}}{\varepsilon^2}\big \| \langle \tildeg^{t+1} - g^t,x\rangle \big \|_{L^2(\mu \otimes \nu)}^2 \\
&= \left ( \frac{2}{\underline{\lambda}\eta^2} + \frac{2M_x^2e^{2\bar{K}^{\star}/\varepsilon}}{\varepsilon^2} \right )\big \| \langle \tildeg^{t+1} - g^t,x\rangle \big \|_{L^2(\mu \otimes \nu)}^2. 
\end{split}
\end{equation*}
Combining these estimates, we obtain
\begin{equation*}
\big \| (\ell_f^{t+1},\ell_g^{t+1}, \ell_h^{t+1})\big\|_{\calH}^2 \le \left ( \frac{2}{\underline{\lambda}\eta^2} + \frac{5M_x^2e^{2\bar{K}^{\star}/\varepsilon}}{\varepsilon^2} \right ) \big(\big \|\tildef^{t+1} - f^t\big \|_{L^2(\mu)}^2 +\big \| \langle \tildeg^{t+1} - g^t,x\rangle \big \|_{L^2(\mu \otimes \nu)}^2\big),
\end{equation*}
completing the proof.
\end{proof}

\begin{proof}[Proof of Theorem \ref{thm: linear conv modified}]
Let $u_{t} :=  D(\barf,\barg,\barh) - D(f^{t},g^{t},h^{t})$. 
Combining Lemmas \ref{lem: PL modified} and \ref{lem: slope ascent modified}, we obtain
\begin{equation*}
u_t - u_{t+1} \ge \underbrace{\left [\frac{2 (1 \wedge \underline{\lambda})e^{-\bar{K}^{\star}/\varepsilon}}{\varepsilon} \times \alpha^{\star} \times (\beta^\star)^2 \right ]}_{=\tau^{\star}} u_{t+1}.
\end{equation*}
Solving this inequality yields the first claim. 

Consider a metric $\mathsf{d}$ on $\calH$ as in (\ref{eq: metric}). Observe from the slope-ascent conditions that 
\[
\begin{split}
&u_t - u_{t+1} \ge \alpha^\star\mathsf{d}_{t+1}^2, \quad \text{with} \quad  \mathsf{d}_{t+1}:= \mathsf{d}\big( (f^t,g^t,h^t), (\tildef^{t+1},\tildeg^{t+1},\tildeh^{t+1}) \big), \\
&\mathsf{d}_{t+1} \ge \beta^\star \big \|(\ell_f^{t+1},\ell_g^{t+1},\ell_h^{t+1}) \big \|_{\calH},
\end{split}
\]and from the PL inequality that
\[
2 \sqrt{u_t} \le \frac{1}{\gamma^{\star}}  \big \|(\ell_f^{t},\ell_g^{t},\ell_h^{t}) \big \|_{\calH} \quad \text{with} \quad \gamma^\star = \sqrt{\frac{(1 \wedge \underline{\lambda})e^{-\bar{K}^\star /\varepsilon}}{2\varepsilon}}. 
\]
The rest of the argument is completely analogous to Theorem \ref{thm: linear conv}.
\end{proof}

\subsection{Proof of Proposition~\ref{prop: derivative}}
Observe that 
\begin{equation*}
\begin{split}
\varepsilon \KL{\bar{\pi}}{\pi^t} &= -\int \Big ( \big(\tildef^{t+1}(u)-\barf(u)\big) + \big\langle g^t(u)-\barg(u),x\big\rangle\\
&\qquad \qquad +\big (\tildeh^{t+1}(x,y) - \barh(x,y)\big) \Big) \, d\bar{\pi}(u,x,y) \\
&= - \int (  \tildef^{t+1}-\barf) \, d\mu - \int (\tildeh^{t+1}-\barh) \, d\nu, 
\end{split}
\end{equation*}
where we used the fact that $\int x \, d\bar{\pi}_u = 0$. For any bounded measurable function $\varphi: \calX \times \calY \to \R$, Pinsker's inequality (cf. Theorem 7.10 in \cite{polyanskiy2025information}) yields
\begin{equation*}
\left | \int \varphi \, d\big(\pi_u^t - \bar{\pi}_u\big) \right |^2 \le 2\|\varphi\|_\infty^2\KL{\bar{\pi}_u}{\pi_u^t}.
\end{equation*}
Integrating with respect to $\mu$ and using the chain rule for the KL divergence (cf. Theorem 2.15 in \cite{polyanskiy2025information}) and Jensen's inequality, we conclude
\begin{equation*}
\begin{split}
\int \left | \int \varphi \, d\big(\pi_u^t - \bar{\pi}_u\big) \right |^2 \, d\mu(u) &\le 2\|\varphi\|_\infty^2 \int \KL{\bar{\pi}_u}{\pi_u^t} \, d\mu(u) = 2\|\varphi\|_\infty^2 \KL{\bar{\pi}}{\pi^t} \\
&\le \frac{2\|\varphi\|_\infty^2}{\varepsilon}\big ( \big \| \tildef^{t+1}-\barf\big \|_{L^2(\mu)} +  \big \| \tildeh^{t+1}-\barh\big \|_{L^2(\nu)} \big).
\end{split}
\end{equation*}
The right-hand side is $O((1+\tau^{\star})^{-t/2})$ by the remark following Theorem \ref{thm: linear conv modified}. This yields that
\begin{equation*}
\begin{split}
&\left\| \E_{\pi_u^t}\big[ \tilde{Y} \big] - \E_{\bar{\pi}_u}\big[ \tilde{Y} \big] \right \|_{L^2(\mu)}^2, \quad  \left\| \E_{\pi_u^t}\big[ \tilde{X}\tilde{Y}^\top \big] - \E_{\bar{\pi}_u}\big[ \tilde{X}\tilde{Y}^\top \big] \right \|_{L^2(\mu)}^2, \\ &\left\| \E_{\pi_u^t}\big[ \tilde{X}\tilde{X}^\top \big] - \E_{\bar{\pi}_u}\big[ \tilde{X}\tilde{X}^\top \big] \right \|_{L^2(\mu)}^2
\end{split}
\end{equation*}
are all $O((1+\tau^{\star})^{-t/2})$. Now, we observe
\begin{equation*}
\begin{split}
&\left \| \Big(\E_{\pi_u^t}\big[ \tilde{X}\tilde{X}^\top \big] \Big)^{-1}- \Big(\E_{\bar{\pi}_u}\big[ \tilde{X}\tilde{X}^\top \big]\Big)^{-1} \right \|_{\op}  \\
&\le \left \| \Big(\E_{\pi_u^t}\big[ \tilde{X}\tilde{X}^\top \big] \Big)^{-1} \right \|_{\op}
\left \| \Big(\E_{\bar{\pi}_u}\big[ \tilde{X}\tilde{X}^\top \big] \Big)^{-1} \right \|_{\op}\left\| \E_{\pi_u^t}\big[ \tilde{X}\tilde{X}^\top \big] - \E_{\bar{\pi}_u}\big[ \tilde{X}\tilde{X}^\top \big] \right \|_{\op}.
\end{split}
\end{equation*}
For any $v \in \mathbb{S}^{d_x-1}$,
\begin{equation*}
\E_{\pi_u^t}\big[ \langle v,\tilde{X} \rangle^2 \big] \ge e^{-\bar{K}^{\star}/\varepsilon} \E[\langle v,X \rangle^2] \ge e^{-\bar{K}^{\star}/\varepsilon} \underline{\lambda}. 
\end{equation*}
A similar estimate holds with $\pi_u^t$ replaced by $\bar{\pi}_u$. We conclude
\begin{equation*}
\left \| \Big(\E_{\pi_u^t}\big[ \tilde{X}\tilde{X}^\top \big] \Big)^{-1}- \Big(\E_{\bar{\pi}_u}\big[ \tilde{X}\tilde{X}^\top \big]\Big)^{-1} \right \|_{\op} \le e^{2\bar{K}^{\star}/\varepsilon} \underline{\lambda}^{-2} \left\| \E_{\pi_u^t}\big[ \tilde{X}\tilde{X}^\top \big] - \E_{\bar{\pi}_u}\big[ \tilde{X}\tilde{X}^\top \big] \right \|_{\op}.
\end{equation*}
The $L^2(\mu)$-norm of the left-hand side is $O((1+\tau^{\star})^{-t/4})$. Putting these estimates together, we obtain the desired result. 
\qed

\section{Discussion}
\label{sec: discussion}

We conclude this paper by discussing a few future research directions, except for the one already discussed after Theorem \ref{thm: linear conv}. 
First, it is of interest to explore general divergence penalties for the VQR problem other than the entropic one, such as the quadratic penalty that enforces sparsity of the support of the optimal coupling \cite{nutz2025quadratically}. There is a recent line of work on optimization of quadratically regularized optimal transport \cite{gonzalez2026polyak,gonzalez2025linear}.  As opposed to the standard OT case, duality theory for VQR (or more generally OT with moment constraints) with general divergence penalties has not been well understood. In addition, establishing convergence guarantees for adaptations of the Sinkhorn algorithm for VQR with quadratic penalization would require establishing quantitative estimates of the iterates. Our proofs for such estimates in the entropic VQR case rely on techniques involving the KL divergence and exponential families, which do not seem to extend directly to other divergence penalties. 

Finally, similar to many previous works on the Sinkhorn algorithm, we have assumed compactness of the supports, on which our analysis relies substantially. For standard entropic OT, there is a recent interest in establishing linear convergence for the Sinkhorn algorithm for marginals with unbounded supports; see \cite{conforti2023quantitative, eckstein2025hilbert}. An extension of our analysis to the unbounded support setting is left for future research.  

\section*{Acknowledgments}
During the preparation of the second version, the authors used ChatGPT (GPT‑5.6 Sol) for assistance with the numerical experiments, in developing the example in Appendix A, and in checking the mathematical correctness of the proofs and typographical errors. All resulting computations and mathematical arguments were independently verified by the authors, who take full responsibility for the paper.

\appendix

\section{Sharpness of the step size condition in Section \ref{sec: modified sinkhorn}}
\label{sec: counterexample}

In Section \ref{sec: modified sinkhorn}, we required the step size $\eta$ to be
\[
0 < \eta < \frac{2\varepsilon}{M_x^2}
\]
to guarantee that the dual objective increases monotonically along the iterates. We suppose here that $\eta > \frac{2\varepsilon}{M_x^2}$ and find marginals for which one iteration of the modified Sinkhorn algorithm strictly decreases the dual objective, which would imply sharpness of the step size condition.

For $a,b > 0$, we consider $d_x=d_y=1$ and set $\mu = \frac{1}{2}(\delta_{-a} + \delta_a)$ and $\nu = \frac{1}{2} ( \delta_{(-a,0)} + \delta_{(a,0)})$, which satisfy our baseline assumption. We initialize the algorithm with $(f^0,g^0,h^0)=(f,g,h)$ and obtain the update $(\tildef^1,\tildeg^1,\tildeh^1)=(\tildef,\tildeg,\tildeh)$. Consider
\[
f(\pm a) = 0, \ g(\pm a) = \pm b, \ h(\pm a,0) = \frac{a^2}{2} - \varepsilon \log \cosh (ab/\varepsilon). 
\]
A direct computation shows that $\tildef = f, \tildeh = h$, and $\tildeg (\pm a) = \pm b \mp \eta qa$ with $q = \tanh (ab/\varepsilon)$. Now, we observe that
\[
D(\tildef,\tildeg,\tildeh) - D(\tildef,g,\tildeh) = \varepsilon \big ( 1-\cosh (sq) + q \sinh (sq) \big ),
\]
where $s = \eta a^2/\varepsilon$. As $q \to 0$ (i.e., $b \to 0$), the right-hand side can be expanded as 
\[
\varepsilon s \left ( 1-\frac{s}{2} \right) q^2 + o(q^2),
\]
where the leading term becomes strictly negative if $s > 2$, i.e., $\eta > 2\varepsilon/a^2$. 

\section{Additional experiments}
\label{sec: additional experiments}

\subsection{Numerical comparison between vanilla and modified Sinkhorn algorithms}

In this appendix, we present a wall-clock time comparison between the vanilla and modified Sinkhorn algorithms. 
The numerical comparison between the vanilla and modified Sinkhorn algorithms is subtle, as it critically depends on how we numerically solve the implicit $g$-update in the vanilla Sinkhorn scheme, while doing so is outside the setting of Theorem~\ref{thm: linear conv}.
In addition, the vanilla Sinkhorn algorithm implements exact block coordinate ascent, which would be expected to incur fewer outer iterations than the modified one, while available packages to numerically solve the implicit $g$-update in inner iterations involve fine-tuning for step size, which would give a somewhat unfair advantage to the vanilla Sinkhorn algorithm compared with a naive implementation of the modified one.

With that in mind, we conduct two experiments. For both experiments, we use the synthetic Gaussian data as in Section \ref{sec: numerical} and the regularization level $\varepsilon=0.1$ for simplicity. In the first experiment,  we use (full) gradient ascent to solve the $g$-update in the vanilla Sinkhorn algorithm with the same step size as the modified Sinkhorn one, where the inner iteration is terminated when the infinity norm of the gradient in the $g$-direction (see (\ref{eq: gradient})) is at most $10^{-6}$. For each algorithm, the outer iteration is terminated when  $\|(\ell_f^t,\ell_g^t,\ell_h^t)\|_{\mathcal H} \le 10^{-5}$. We use the same step size, $\eta = \varepsilon/M_x^2$, for both algorithms. Indeed, from the coercive estimate in the remark following Lemma \ref{lem: monotonicity modified}, this step size guarantees convergence of the inner optimization for the vanilla Sinkhorn algorithm. The complexity of the vanilla Sinkhorn algorithm is $O( mnd_x + Kmnd_x)$ per outer iteration, where $K$ is the number of inner iterations. 

In the second experiment, we use adaptive step-size selection for the modified Sinkhorn algorithm and the inner gradient ascent in the vanilla Sinkhorn algorithm.
Specifically, the step size is selected using the Armijo backtracking line search in \texttt{LineSearches.jl}.
We use the package defaults for the backtracking rule and set the largest trial step size to $10\varepsilon/M_x^2$ (consistent with the stable step size experiment).
In addition to gradient ascent, we examine the damped Newton method for the $g$-update in the vanilla Sinkhorn algorithm, implemented using
\texttt{Optim.jl}.  
The gradient and Hessian are evaluated analytically rather than by numerical differentiation, and the Newton iteration is warm-started from the preceding outer iteration.
Apart from these problem-specific objective and derivative evaluations, we retain the default numerical settings of \texttt{Optim.jl}. The complexity of the vanilla Sinkhorn algorithm with the Newton $g$-update is $O( mnd_x + Km(nd_x^2 + d_x^3) )$ per outer iteration, where $K$ is the number of inner iterations. The stopping criterion for the outer iteration is $\|(\ell_f^t,\ell_g^t,\ell_h^t)\|_{\mathcal H} \le 10^{-5}$ for each implementation. 

Finally, we run algorithms under the same single-core computing environment. Before collecting timings, we execute a small preliminary run so that \texttt{Julia} has already compiled the relevant code. We measure only the time spent solving the optimization problem, excluding data generation and preprocessing. Before each timed run, we clear unused memory to reduce variation in the recorded times.

\begin{figure}[t]
    \centering
    \includegraphics[width=0.94\textwidth]
        {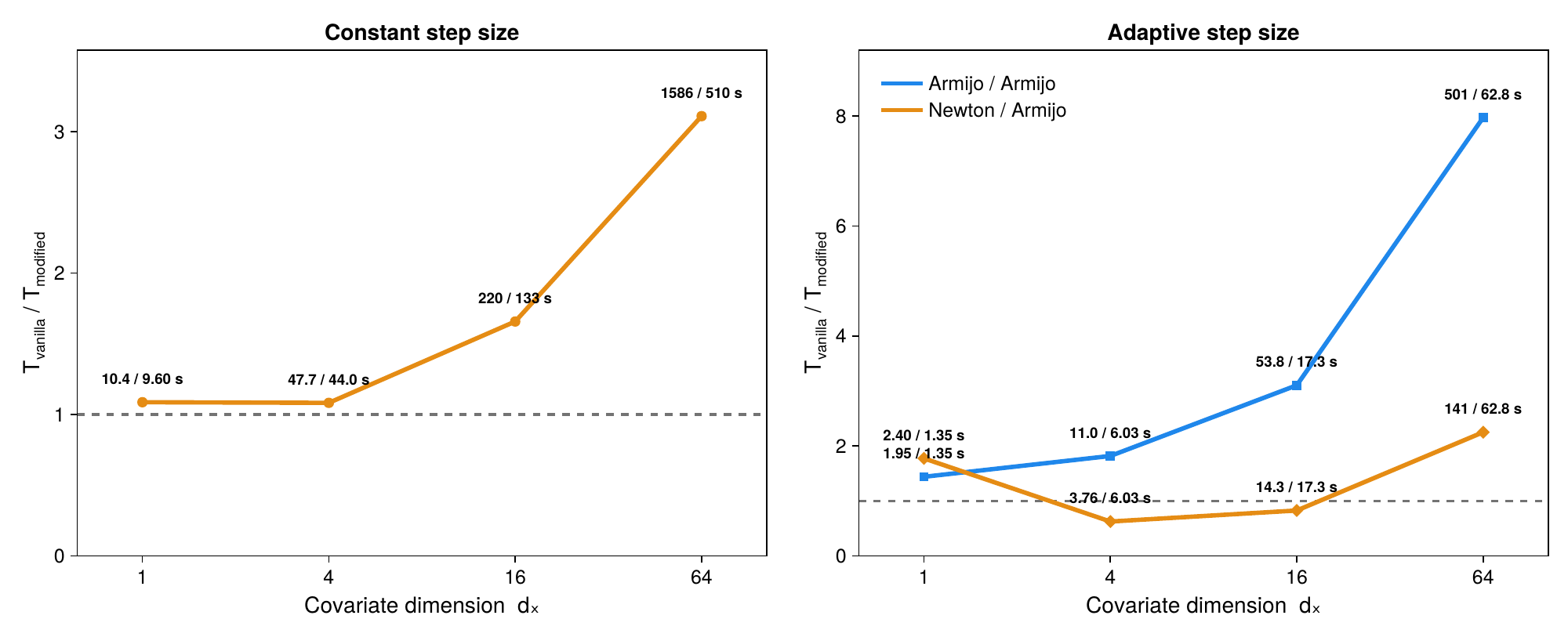}
    \caption{{\footnotesize
        \textbf{Wall-clock comparison of the vanilla and modified Sinkhorn algorithms on synthetic Gaussian data.}
        The vertical axis reports the ratio $T_{\mathrm{vanilla}}/T_{\mathrm{modified}}$ of the wall-clock times.
        Point labels report the corresponding wall-clock times, in seconds, in the order vanilla / modified. Here $\varepsilon=0.1$, and outer iterations are terminated when $\|(\ell_f^t,\ell_g^t,\ell_h^t)\|_{\mathcal H}\le 10^{-5}$. Left panel: both algorithms use the constant step size $\eta=\varepsilon/M_x^2$. Right panel: the legend gives the vanilla inner solver / modified update rule.
        }
    } 
    \label{fig:wall-clock comparison}
\end{figure}

Figure \ref{fig:wall-clock comparison} shows the result. With the common constant step size $\eta = \varepsilon/M_x^2$, both algorithms are comparable, but the modified one is faster for larger dimensions. With Armijo backtracking, both algorithms become faster, but the modified one is always advantageous. The damped-Newton update makes the vanilla algorithm more competitive at intermediate dimensions. When $d_x=64$, the vanilla algorithm is slower, which is consistent with the fact that the per-iteration cost of the Newton update is sensitive to $d_x$. 

Taken together, these results suggest that a carefully chosen second-order solver can offset the cost of the implicit $g$-update in moderate dimensions, whereas the modified scheme, which avoids an inner optimization loop, appears to scale more favorably with $d_x$ in these experiments. However, because the comparison necessarily depends on the numerical solver, stopping rules, and step-size tuning used for the implicit update, these timing results should be viewed with caution.

\FloatBarrier

\subsection{Experiment pertaining to Section \ref{sec: derivative}}
This experiment pertains to the problem considered in Section \ref{sec: derivative} and illustrates how the coefficient functions obtained from entropic VQR approach their unregularized counterparts as the regularization parameter $\varepsilon$ decreases. Consider the setting in Section \ref{sec: derivative}. To make explicit the dependence on $\varepsilon$, we write $B_0^{(\varepsilon)}$ and $B_1^{(\varepsilon)}$, which are well-defined even when $\varepsilon=0$ (replace the entropic coupling with an optimal solution to the unregularized VQR problem).  

We use the iris dataset\footnote{Available from \url{https://archive.ics.uci.edu/dataset/53/iris}.}, which contains $n = 150$ observations, and take 
\[
X = \texttt{sepal width}, \quad Y = (\texttt{sepal length},\texttt{petal length})^\top.
\]
The reference measure $\mu$ is the empirical distribution of $m=1000$ independent samples from $\operatorname{Unif}([0,1]^2)$, and $\nu$ is the empirical distribution of $(X,Y)$. We solve the unregularized VQR problem by linear programming, using \texttt{JuMP.jl} with the \texttt{HiGHS.jl} solver, and the entropic VQR problem using the modified Sinkhorn algorithm with the step size rule $\eta = \varepsilon/M_x^2$. For the modified Sinkhorn algorithm, we use the stopping criterion $\|(\ell_f^t,\ell_g^t,\ell_h^t)\|_{\mathcal H} \le 10^{-5}$. Figure \ref{fig:iris-regularization} shows that both errors decrease monotonically as $\varepsilon$ decreases. 

\begin{figure}[t]
    \centering
    \includegraphics[width=0.94\textwidth]
        {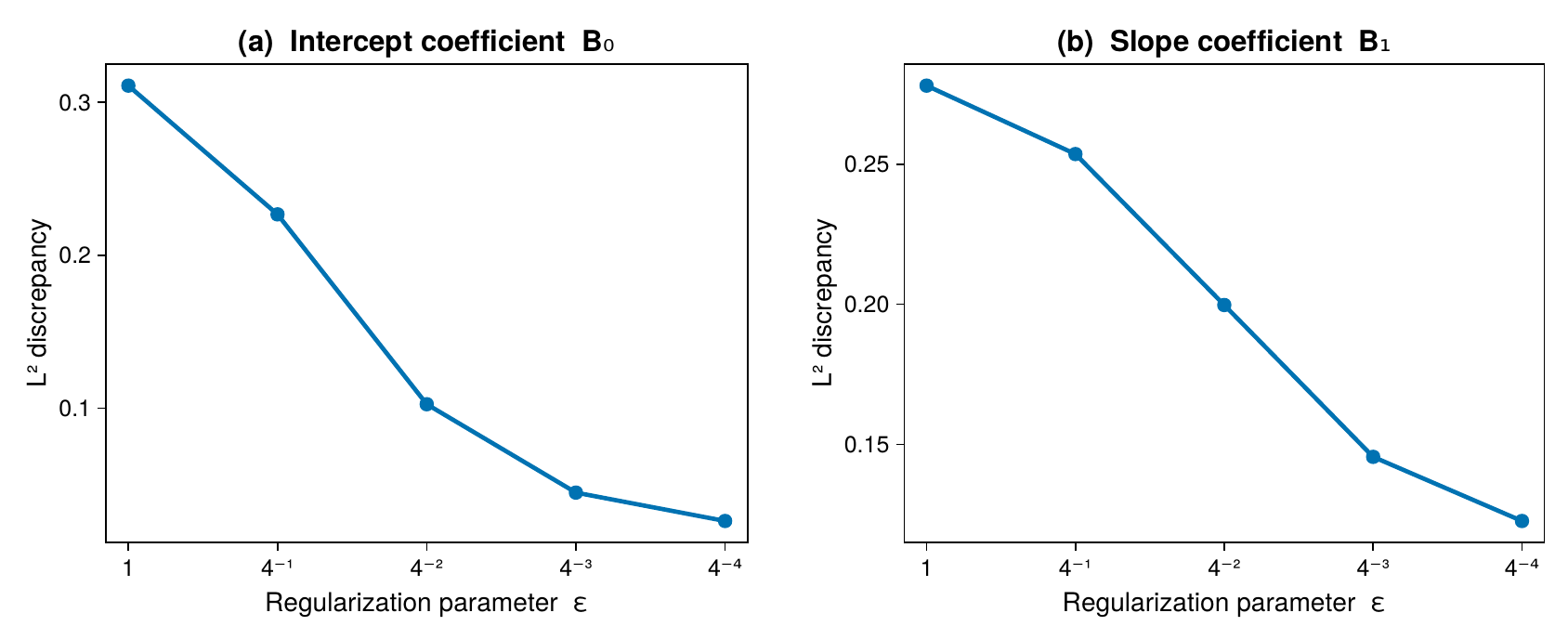}
    \caption{{\footnotesize
        \textbf{Effect of entropic regularization on the estimated coefficient
        functions in the iris experiment.}
        The left panel reports
        $\|B_0^{(\varepsilon)}-B_0^{(0)}\|_{L^2(\mu)}$ and the right panel reports
        $\|B_1^{(\varepsilon)}-B_1^{(0)} \|_{L^2(\mu)}$, where $B_0^{(0)}$ and $B_1^{(0)}$
        are computed from an optimal solution of the 
        unregularized VQR problem. The same $m=1000$ reference
        points are used in the unregularized and entropic problems.
        The regularization levels are
        $\varepsilon=4^{-p}$ for $p=0,1,2,3,4$.
        The coefficient functions are computed using centered sepal width on its original scale.
        }
    } 
    \label{fig:iris-regularization}
\end{figure}

   \FloatBarrier
\bibliographystyle{alpha}
\bibliography{reference}
\end{document}